\documentclass{amsart}
\usepackage[a4paper]{geometry}                
\usepackage{comment}                   
\usepackage[parfill]{parskip}    
\usepackage{graphicx}
\usepackage{amssymb}
\usepackage{hyperref}
\usepackage{pdfsync}
\usepackage{footmisc}
\usepackage[all]{xy}
\usepackage{enumitem}
\usepackage{xcolor}
\usepackage{datetime2}
\usepackage{tikz}
\usetikzlibrary{backgrounds}
\usepackage{ifthen}

\newtheorem{thm}{Theorem}[section]
\newtheorem*{thm*}{Theorem}

\newtheorem*{cor*}{Corollary}

\newtheorem{prop}[thm]{Proposition}
\newtheorem{lem}[thm]{Lemma}
\newtheorem{cor}[thm]{Corollary}
 
\theoremstyle{definition}

\newtheorem{defn}[thm]{Definition}

\newtheorem{question}{Question}

\newtheorem{rem}[thm]{Remark}
\newtheorem{expl}[thm]{Example}

\newcommand{\cC}{\mathcal{C}}
\newcommand{\cF}{\mathcal{F}}
\newcommand{\cK}{\mathcal{K}}

\newcommand{\ww}{\mathbf{w}}
\newcommand{\bm}{\mathbf{m}}
\newcommand{\maps}{\rightarrow}
\DeclareMathOperator{\Z}{Z}

\DeclareMathOperator{\ncl}{ncl}

\DeclareMathOperator{\lk}{lk}

\DeclareMathOperator{\st}{star}
\DeclareMathOperator{\Min}{min}

\newcommand{\mC}{\mathfrak{C}}

\newcommand{\mbf}[1]{\mathbf{#1}}
\newcommand{\BN}{\mathbb{N}}

\newcommand{\BZ}{\mathbb{Z}}
\newcommand{\BQ}{\mathbb{Q}}
\newcommand{\factor}[2]{{\raise0.7ex\hbox{$#1$} \!\mathord{\left/ {\vphantom {#1 {#2}}}\right.\kern-\nulldelimiterspace}\!\lower0.7ex\hbox{${#2}$}}}
\newcommand{\KK}{\ensuremath{\mathbb{K}}}
\newcommand{\GG}{\ensuremath{\mathbb{G}}}
\newcommand{\HH}{\ensuremath{\mathbb{H}}}

\newcommand{\la}{\langle}
\newcommand{\ra}{\rangle}

\newcommand{\Ts}{\mathfrak{T}}

\newcommand{\Hc}{\mathcal{H}}

\renewcommand{\b}{\mathfrak{b}}

\newcommand{\G}{\Gamma}
\newcommand{\pd}{\partial}

\newcommand{\biz}{\begin{itemize}}
\newcommand{\eiz}{\end{itemize}}
\begin{document}
\title{Limit groups over coherent right-angled Artin groups}
\author[]{Montserrat Casals-Ruiz}
\address{Ikerbasque - Basque Foundation for Science and Matematika Saila, UPV/EHU, Sarriena s/n, 48940, Leioa - Bizkaia, Spain}
\email{montsecasals@gmail.com}
\author[]{Andrew Duncan}
\address{School of Mathematics, Statistics and Physics, Newcastle University, Newcastle upon Tyne, NE1 7RU, United Kingdom}
\email{andrew.duncan@newcastle.ac.uk}
\author[]{Ilya Kazachkov}
\address{Ikerbasque - Basque Foundation for Science and Matematika Saila, UPV/EHU, Sarriena s/n, 48940, Leioa - Bizkaia, Spain}
\email{ilya.kazachkov@gmail.com}
\begin{abstract}
  A new class of groups $\cC$, containing all coherent RAAGs and all toral relatively hyperbolic groups, is defined.
  It is shown that, for a group $G$ in the class $\cC$, the $\BZ[t]$-exponential group $G^{\BZ[t]}$ may be constructed
  as an iterated centraliser extension. Using this fact, it is proved that $G^{\BZ[t]}$ is fully residually $G$ (i.e. it has the same universal theory as $G$) and so its finitely generated subgroups are limit groups over $G$. If $\GG$ is a coherent RAAG, then the converse also holds - any limit group over $\GG$ embeds
  into $\GG^{\BZ[t]}$. Moreover, it is proved that limit groups  over $\GG$ are finitely presented, coherent and CAT$(0)$, so in particular have solvable word and conjugacy problems. 
\end{abstract}
\maketitle
{\let\thefootnote\relax}
{\let\thefootnote\relax}

\section{Introduction}
When studying problems in $R$-modules or $R$-algebras it is often convenient to extend the ring of coefficients for instance from an integral domain to a field of fractions. The multiplicative notation employed in the study of groups naturally defines an action of the ring of integers on the group via exponentiation and, as in the case of modules and algebras, it is sometimes useful to extend this exponential action of $\mathbb Z$ to more general rings, bringing us to the notion of exponential group over a ring $A$. Exponential groups were initially considered by P.~Hall in \cite{Hall69} where he studied nilpotent group with exponents in a binomial ring. G.~Baumslag, see \cite{Baumslag60a}, considered groups with unique $p$th roots and their embedding into divisible groups, which are exponential groups over the ring $\mathbb Q$. Lyndon was the first to take an axiomatic point of view, introducing axioms for exponential groups over an associative ring $A$, or $A$-groups, while studying equations and the first order theory of the free group, and describing the set of solutions of an equation in one variable over a free group $F$ in terms of a free object $F^A$ in the category of $A$-groups, where $A=\BZ[t_1,\ldots ,t_d]$, that is, groups with exponents in the ring of polynomials $\mathbb Z[t]$, see \cite{Lyndon60a,Lyndon60b}.

When extending the exponents of a group $G$ to a ring $A$, there is a natural ``free" $A$-group, called the $A$-completion $G^A$ of $G$, defined via the universal property that any group homomorphism from $G$ to an $A$-group $H$, factors through $G^A$, via a canonical embedding from $G$ into $G^A$
and an $A$-homomorphism from $G^A$ to $H$.

Although the existence of an $A$-completion or a free exponential group may follow from general results on varieties of algebras, their structure may be very complex and difficult to describe and in particular a group may not embed into its $A$-completion. G.~Baumslag described the $\mathbb Q$-completion of a group with unique roots  as
a direct limit of a directed system of $A$-groups, for $A$ a subring of $\BQ$. More precisely, Baumslag's directed system involves the idea of an ``iterated centraliser extension'' (or ICE), that is a group built by repeated formation of free products with amalgamation, over centralisers of elements (see Definition \ref{defn:ice}). Indeed in this setting, $G^A$, built as an ICE, is a Fra\"{i}ss\'{e} limit of extensions of centralisers, see \cite{KMS20}. It is proved that, for $A$ a subring of $\BQ$,  the $A$-completion of the
ordinary free group $F$ is the free $A$-group $F^A$ in the variety of $A$-groups
(see \cite[Sections V--VIII]{Baumslag60a}). 

Lyndon gave an explicit description of the structure of the free $\mathbb Z[t]$-group $F^{\mathbb Z[t]}$ as the limit of an ascending chain of subgroups and proved that $F^{\mathbb Z[t]}$ is fully residually free, or put another way, is universally equivalent to the free group $F$. In \cite{MR94,MR95}, Miasnikov and Remeslennikov made a systematic study of $A$-groups and $A$-completions and described the structure of the $A$-completion of CSA groups, that is groups whose maximal abelian subgroups are malnormal. It is shown that if $A$ is an associative ring, with torsion-free additive subgroup, and $G$ is a torsion-free CSA group, then the $A$-completion $G^A$ of $G$ may be constructed as an ICE, whence many results on its structure may be deduced; for example that $G$ embeds into $G^A$, which is itself a torsion-free CSA group. In this case too, a free $A$-group is the $A$-completion of a free group.

In the introduction to \cite{Lyndon60b} Lyndon comments that ``... although the connection is perhaps remote,
my interest in the present problem derives from a question of A. Tarski, whether the `elementary theory' of
free groups is decidable''. It is now clear that the connection is anything but remote. Indeed, 
crucial properties of Lyndon's free $\BZ[t]$-group $F^{\BZ[t]}$
are that it serves as a universe for the class of finitely generated models of
the universal theory of the free group and that it has a very robust algebraic structure. More precisely,
in their work \cite{KM98}  (see also \cite{CG05}), the authors show that any finitely generated fully residually free group (equivalently, any finitely generated model of the universal
theory of the free group, also known as a limit group, see Section \ref{sec:limit}) embeds into $F^{\BZ[t]}$.

An immediate consequence of this result is the fact that limit groups split either as a free product or
over an infinite cyclic subgroup and, since finitely generated subgroups of limit groups are limit groups, this endows them with a hierarchical structure. This fact allows one to use induction to prove a variety of powerful results on limit groups. For example, it readily follows that limit groups are finitely presented; in \cite{ChZa16}, the authors use the hierarchical structure of limit groups to prove that they are conjugacy separable, and in \cite{Wilton08} the author establishes that limit groups are subgroup separable.

In a different direction, the  knowledge that limit groups embed into $F^{\BZ[t]}$ was key to
establishing the fact that limit groups act freely on $\BZ^n$-trees and allowed to use the techniques of non-Archimedian words to establish the Howson property, analogues of Hall's and Greenberg-Stallings's theorems for free groups, see \cite{NS11, Wilton08}, and to
address all the major algorithmic questions for limit groups, see
\cite{KMRS08,BKM07} and references there.

This type of result, showing limit groups over a group $G$ embed in the $\BZ[t]$-completion of $G$,
has been established for other groups with hyperbolic features. The proofs 
divide into two main steps. The first one is to show that the $\BZ[t]$-completion is a model of the universal theory of the base group $G$, i.e. is fully residually $G$. In their paper \cite{BMR02}, Baumslag, Miasnikov and Remeslennikov carried out this step in the case of a torsion-free hyperbolic group $G$ and distilled the key ideas  of  the proof. As in \cite{MR95} a class of groups that contains $G$ and is closed under extensions of centralisers is required.  In this case the group $G$ is torsion-free hyperbolic and the class chosen is the class of CSA groups satisfying the  Big Powers (BP)
property. The BP property is an algebraic condition that asserts that big powers of non-commuting elements generate a free group (see Definition \ref{defn:bp} below). It turns out that an extension of centralisers of a group $G$ which is CSA and has the BP property   is fully residually $G$, and so a model of its universal theory.  The strategy is then to show that an extension of a centraliser of a group in the class of CSA and BP groups is again in this class,  to use transfinite induction show that an ICE of such a group remains in the class, and to conclude that $G^{\BZ[t]}$ is fully residually $G$. These ideas were taken further in \cite{KM12} in the study of the $\BZ[t]$-completions of   
 torsion-free toral relatively hyperbolic groups.

The second main step of the proof is to show that any limit group over $G$ embeds into $G^{\BZ[t]}$. In order to prove this statement, one needs to understand the structure of limit groups over $G$. In the case of free groups, this was achieved by Kharlampovich and Miasnikov in \cite{MR95}, using the Makanin-Razborov process and by Sela in \cite{Sela1} using actions on real trees.

\medskip

The aim of this paper is to describe $\BZ[t]$-completions for (a class of) right-angled Artin groups, or RAAGs for short, see Definition \ref{defn:raags}. RAAGs form  a prominent class of groups which contains both free and free abelian groups and is widely studied in different branches of mathematics and computer science. We refer the reader to \cite{Charney07,DKR4} for further details.

In general, the subgroup structure of RAAGs is very complex.
  For example, Bestvina and Brady \cite{BestvinaBrady}  showed how to construct a subgroup of a RAAG to give  the first example of a group which is homologically
  finite (of type $FP$) but
  not geometrically finite (in fact not of type $F_2$). Again, Mihailova's example \cite{Mihailova1968} of a group with unsolvable subgroup membership problem is constructed from
  the RAAG $F_2\times F_2$.  More recently results of Wise and others (see e.g. \cite{W}) led to Agol's proof of the virtual Haken conjecture: that every hyperbolic Haken $3$-manifold
is virtually fibred.  An essential  step in the argument uses the result that the fundamental groups of  so-called special cube complexes 
embed into RAAGs.
 If one is to use the $\BZ[t]$-completion as a universe 
for limit groups over a given RAAG and as a tool to study limit groups algorithmically,
it is natural to restrict to a class of RAAGs with tame subgroup structure and good algorithmic behaviour --
in our case, to the class of coherent RAAGs. A RAAG is coherent  if and only if its underlying graph contains no full subgraph isomorphic to an  $n$-cycle, for $n\ge 4$, see \cite{droms87a}. Furthermore, a RAAG is coherent if and only if it satisfies the BP property, see \cite{Blatherwick}.

In this paper, we define a new class of groups $\mathcal{C}$ which contains all toral relatively hyperbolic groups and all coherent RAAGs and prove that it is closed under extensions of centralisers and direct limits. For groups $G$ from this class,
we show that the $\BZ[t]$-completion may be built as an ICE, enabling us to
prove it is fully residually $G$. Thus, we give a general framework for the results of Baumslag, Kharlampovich, Miasnikov and Remeslennikov \cite{MR95,BMR02,KM12} and establish analogous results for coherent RAAGs. More precisely we prove the following theorem.

\begin{thm*}[see Theorem \ref{thm:tensor}]
  Let $G$ be a group from class $\mathcal{C}$ satisfying Condition \ref{cond:R} of Section \ref{sec:nabtensor}. The $\BZ[t]$-completion $G^{\mathbb Z[t]}$ of $G$
  can be built as an ICE.  It is fully residually $G$ and so it has the same universal and existential theory as $G$. 
\end{thm*}

In the case that $G$ is a coherent RAAG, we use the fact that any limit group over a RAAG is a subgroup of a graph tower, see \cite{CRK15}, to conclude that $G^{\mathbb Z[t]}$ is a universe for the class of limit groups over $G$.

\begin{cor*}[see Theorem \ref{thm:charlim}]\
Let $G$ be a coherent RAAG. Then $G^{\mathbb Z[t]}$ can be built as an ICE.  It is fully residually $G$ and so it has the same universal and existential theory as $G$. 

Furthermore, it is a universe for the class of limit groups over $G$, i.e. any finitely generated model of the universal theory of $G$ embeds into  $G^{\mathbb Z[t]}$.
\end{cor*}

To prove our main result, we follow the strategy that we have sketched for hyperbolic groups. In that case, the authors consider the class of BP and CSA groups. The CSA property has strong structural consequences, namely in such groups commutation is transitive and so the centralisers (of non-trivial elements) are abelian. The fact that centralisers are abelian (in fact, in the cases considered, they are f.g. free abelian) is essential to showing that a countable sequence of extensions of a centraliser admits a $\mathbb Z[t]$-action, then, using the malnormality of  abelian subgroups, one extends the $\mathbb Z[t]$-action, in a consistent way, from the abelian subgroups to the whole group;
thereby establishing that indeed the resulting limit is the $\BZ[t]$-completion of $G$.
The BP property is needed to show that extensions of centralisers are fully residually $G$.

In our case, we define the class $\mathcal C$ to contain groups that are torsion-free and  BP,
and such that the algebraic structure of centralisers of elements, and their intersections,
has a clear description, see Definition \ref{defn:C}. 
The main technical work is to show that the class $\cC$ is indeed closed under extension of centralisers,
see Theorem \ref{thm:lift}. We then prove that a group obtained as an ICE admits a
$\mathbb Z[t]$-action and so is the group $G^{\BZ[t]}$, see Theorem \ref{thm:tensor}.

In order to show that the group $G^{\BZ[t]}$ is a universe for the class of limit groups over $G$, we use the structure
theorem for limit groups proved in \cite{CRK15}, where
the authors show that any limit group over a RAAG $G$ is a subgroup of a \emph{graph tower} over $G$. In the case of coherent RAAGs, we provide a neat description of graph towers which is key to show that they embed into the completion $G^{\BZ[t]}$. More precisely, we prove that any graph tower can be obtained as a free product with amalgamation over a free abelian group, where one of the vertices is a tower of lower height and the other vertex group is either free abelian, or the direct product of a free abelian group and a fundamental group of a non-exceptional surface, and where the free abelian factor is contained in the edge group. The base group is a coherent RAAG $G'$ which is obtained from $G$ by extending centralisers of canonical generators of $G$,
see Lemma \ref{lem:prtower}.

In Section \ref{sec:tower}, we show that graph towers over coherent RAAGs belong to the class $\mathcal C$ and use this fact to prove the following structural results:

\begin{thm*}[see Theorem \ref{thm:abelianedge}] Let $\GG$ be a coherent RAAG and let $\Ts=(G,\HH,\pi)$ be a graph tower associated to a limit group $L$ over $\GG$.
  Then $G$ has a graph of groups decomposition {\rm(}in the same generating set as $G$ {\rm)} where:
\begin{itemize}
\item the graph of the decomposition is a tree;
\item edge groups are finitely generated free abelian;
\item vertex groups are either graph towers of lower height, or a finitely generated free abelian group or the direct product of a  finitely generated free abelian group and a non-exceptional surface group.
\end{itemize}
\end{thm*}
This result is independent of the  first Theorem stated above (Theorem \ref{thm:tensor})
and it is of interest in its own right. For instance we can deduce that 

\begin{cor*}[see Corollary \ref{cor:limit groups are fp}]
 Limit groups over coherent RAAGs are coherent, and so, in particular, finitely presented.
\end{cor*}

In Section \ref{sec:characterisation subgroups ICE}, we use the structure of the graph tower over a coherent RAAG $\GG$ to show that they embed into $\GG^{\mathbb Z[t]}$ and obtain a new characterisation of limit groups over coherent RAAGs - they are precisely finitely generated subgroup of $\GG^{\mathbb Z[t]}$.

Finally, in Section \ref{sec:AB}
we follow the approach of Alibegovi\v{c}-Bestvina, see \cite{AB06}, to prove the following corollary.

\begin{cor*}[see Corollary \ref{cor:cat0}]
  Limit groups over coherent RAAGs are CAT$(0)$ groups. 
\end{cor*}

One may consider the relation with other generalisations of hyperbolicity such as acylindrically hyperbolicity or hierarchical hyperbolicity, see \cite{Osin16} and \cite{behrstock2017}. A RAAG is acylindrically hyperbolic if and only if it is directly indecomposable and not cyclic,
see \cite{CapraceSageev11,Osin16}.
From the structure of limit groups over coherent RAAGs, see Lemma \ref{lem:prtower}, and results from \cite{MinasyanOsin15},
it follows that limit groups over directly indecomposable,  non-cyclic coherent RAAGs are acylindrically hyperbolic. On the other hand, all RAAGs are hierarchically hyperbolic groups (HHG for short)
and one would expect this also to be true of towers over RAAGs.
In general, finitely generated subgroups of RAAGs need not be HHGs
but we expect subgroups of coherent RAAGs, and limit groups over them, to be HHGs. More precisely we ask the following question. 

\begin{question}
Are limit groups over coherent RAAGs hierarchically hyperbolic?
\end{question}

In fact, D.~Wise \cite{W} showed that limit groups have a quasiconvex hierarchy,  and thence deduced
that they are virtually compact special. Similarly, we expect a positive answer to the following. 

\begin{question}
Are limit groups over coherent RAAGs virtually (compact) special?
\end{question}

Notice that virtually compact special groups are hierarchically hyperbolic \cite{behrstock2017}. 

\smallskip

Since RAAGs are linear so are limit groups over them, see \cite[Proposition 8]{BMR} and \cite{Mal40}.
In particular the word problem is decidable for limit groups over RAAGs \cite{Rabin}.
Furthermore, from the fact that limit groups over coherent RAAGs are $CAT(0)$, we deduce the following.

\begin{cor*}[see Corollary \ref{cor:cat0}]
Limit groups over coherent RAAGs have decidable word and conjugacy problems.
\end{cor*}

As we mentioned, we expect that the embedding of limit groups in $\BZ[t]$-completions
will be useful in addressing algorithmic problems and establishing  residual properties. For instance, we ask the following.

\begin{question}\
\
Are the following statements true? 
\begin{itemize}
\item Quasi-convex subgroups of graph towers over (coherent) RAAGs are subgroup separable.
(Notice that we consider the metric induced by the canonical generators of the graph tower (instead of the $CAT(0)$ metric). In particular, this metric induces  the $\ell_1$-metric on abelian subgroups. Algebraically, we consider subgroups $H$ such that if $h=h_1\cdots h_k \in H$ and $h_i$ are pairwise commuting blocks, then $h_i^{k_i} \in H$ for some $k_i\in \mathbb N$ and for all $i=1, \dots, k$.)
\item Cyclic subgroups of limit groups over (coherent) RAAGs are closed in the profinite topology.\\
  (Since this paper was  submitted, Fruchter \cite{fruchter} has shown, using the class $\cC$ introduced  below, that this statement is true.)
\item Limit groups over coherent RAAGs are conjugacy separable.
\item The subgroup membership problem in limit groups over coherent RAAGs is decidable.
\item The isomorphism problem in the class of limit groups over coherent RAAGs is decidable.
\end{itemize}
\end{question}

\paragraph{\textbf{Acknowledgements.}}
Support for this work comes from a Basque Government Grant (no. IT1483-22) and Spanish Government grants (nos. PID2019-107444GA-I00 and PID2020-117281GB-I00).

The authors are grateful to the anonymous referees for careful and constructive comments which prompted significant improvements to the exposition of results.

\section{Preamble}
In this section we recall some basic definitions and properties of RAAGs, see \cite{DKR5, CRK08} for further details. 
\subsection{Right-angled Artin groups}
\begin{defn}\label{defn:raags} \label{defn:raag}   
  Let $\G$ be a finite simple (no loops, no multiple edges) graph with vertices $X$ and edge set $E$. The \emph{right-angled Artin group} (RAAG)
  $\GG(\G)$ with \emph{commutation graph} $\G$ is the group with presentation $\langle X\,|\, R\rangle$, where $R=\{[a,b]\,:\,\{a,b\}\in E\rangle$.
\end{defn}

If $Y$ is a subset of $X$, denote by $\G(Y)$ the full subgraph of $\G$ with vertices $Y$.  Then $\GG(\G(Y))$ is the RAAG with graph $\G(Y)$. One can show that $\GG(\G(Y))$ is the subgroup $\la Y \ra$ of $\GG(\G)$ generated by $Y$. We call $\GG(\G(Y))$ a {\em canonical parabolic} subgroup of $\GG(\G)$ and, when no ambiguity arises, denote it $\GG(Y)$.
The elements of $Y$ are termed the {\em canonical} generators of $\GG(Y)$.

A subgroup $P$ of  $\GG$ is called {\em parabolic} if it is conjugate to a canonical parabolic subgroup $\GG(Y)$ for some $Y\subseteq X$.

If $w$ is a word in the free group on $X$ then we say that $w$ is \emph{reduced} in $\GG(\Gamma)$ if $w$ has minimal length amongst all words representing the same element of $\GG(\Gamma)$ as $w$.  Notice that reduced words correspond to geodesics in the corresponding Caley graph.

If $w$ is reduced in $\GG(\Gamma)$, then we define $\alpha(w)$ to be the set of elements of $X$ such that $x$ or 
$x^{-1}$ occurs in $w$. It is well-known that all reduced words representing a particular element have the same length, and  that if $w=w'$ in $\GG(\Gamma)$ then $\alpha(w)=\alpha(w')$.  We say that elements  $u$ and $v$ of a RAAG \emph{disjointly} commute if $\alpha(u)\cap \alpha(v)=\emptyset$ and $[u,v]=1$. As shown in \cite{Baudisch77}, if $u$ and $v$ disjointly commute then
$[\alpha(u),\alpha(v)]=1$. 

A reduced word $w$ is \emph{cyclically reduced} in $\GG(\Gamma)$ if the length of $w^2$ is twice the length of $w$.

An element $w\in \GG$ is called \emph{root} of $v\in \GG$ if there exists a positive integer $m\in \mathbb{N}$ such that $v=w^m$ and there does not exist $w'\in \GG$, $1\ne w'$ and $0,1\ne m'\in \mathbb{N}$ such that $w={w'}^{m'}$. In this case we write $w=\sqrt{v}$. By a result from \cite{dk93b}, RAAGs have least (or unique) roots, that is the root element of $v$ is defined uniquely.

The \emph{link} of a vertex $x$ of $\G$ is the set of vertices $\lk(x)$ of $\G$ adjacent to $x$, and the \emph{star}
of a $x$ is $\st(x)=\{x\}\cup \lk(x)$.  For a set of vertices $Y\subseteq V(\Gamma)$ define the link $\lk(Y)$ and star $\st(Y)$ of $Y$ to be intersections of the links and stars, respectively, of all vertices in $Y$. By the link and star of   $w\in \GG$  we mean the link  and star, respectively, of $\alpha(w)\subseteq V(\Gamma)$.

The complement of a  graph $\Gamma$ is the graph $\overline{\Gamma}$ with the same vertex set as $\Gamma$ and an edge joining different vertices $x$ and $x'$ if and only if there is no such edge in $\Gamma$. The RAAG $\GG(\Gamma)$ is said to have \emph{non-commutation  graph} $\overline{\Gamma}$.

Consider cyclically reduced $w \in \GG$ and the set $\alpha(w)$. For this set, consider the graph $\overline{\Gamma} (\alpha(w))$
(that is, the full subgraph of $\overline{\Gamma}$ induced by $\alpha(w)$). If this graph is connected we call $w$ a \emph{block}.

If $\overline{\Gamma} (\alpha(w))$ is not connected, then $w$ is a product of commuting words
\begin{equation} \label{eq:bl}
w= w_{j_1} \cdot w_{j_2} \cdots w_{j_t};\ j_1, \dots, j_t \in J,
\end{equation}
where $|J|$ is the number of connected components of $\overline{\Gamma} (\alpha(w))$ and 
$\alpha(w_{j_i})$ is the set of vertices of the $j_i$-th connected component. Clearly, the words $\{w_{j_1}, \dots, w_{j_t}\}$
pairwise disjointly commute. Each word $w_{j_i}$, $i \in {1, \dots,t}$ is a block and so we refer to presentation (\ref{eq:bl})
as a block normal form of $w$. Notice that a block normal form is unique up to the order of the commuting blocks and the order of letters within the blocks. Therefore, abusing the terminology, we usually refer to ``the'' block normal form.

From \cite{Baudisch77}, if the block normal form of a cyclically reduced element $w$ of $\GG$ is $w=w_1\cdots w_t$ and $\sqrt{w_i}=v_i$ then the centraliser of $w$ is
\begin{equation}\label{eq:raagcr}
  C(w)=\la v_1\ra\times \cdots \la v_t\ra\times \la \lk(w)\ra.
  \end{equation}

\subsection{Coherent RAAGs}
Recall that a group is \emph{coherent} if any finitely generated subgroup is finitely presented. 
In \cite{droms87a}, Droms gives a characterisation of RAAGs that are coherent in terms of the defining commutation graph. We recall this characterisation in the following theorem. A graph is \emph{chordal} if it has no full subgraph isomorphic to a cycle graph of more than $3$ vertices. 
\begin{thm}[Theorem 1, \cite{droms87a}]
A RAAG $\GG(\Gamma)$ defined by the commutation graph $\Gamma$ is coherent if and only $\G$ is chordal. 
\end{thm}
Notice that among coherent RAAGs one finds free groups, free abelian groups as well as all RAAGs which are fundamental groups of a three-manifold (see \cite[Theorem 2]{droms87a}).

To motivate the axioms introduced in Section \ref{sec:C} below we give some details of centralisers of elements of coherent RAAGs. Let $\G$ be a (simple) chordal graph and let $\overline{\G}$ be its complement. 
Then, as $\G$ is chordal, $\overline{\G}$ has at most one connected component with more than one vertex. Therefore the block normal form of an element $g$ of the RAAG $\GG(\G)$ has at most one block which is not a power of a generator. Moreover, if $g$ is an element which has a block containing more than one generator, then $\lk(g)$ is a clique (a set of vertices spanning a complete subgraph). Indeed, if $\lk(g)$ is not a clique, it contains two non-adjacent vertices, say $u$ and $v$. Since we assume that $g$ has a block $w$ with more than one generator, then from the definition of block we have that $\overline \Gamma(\alpha(w))$ is connected and so it follows that $\alpha(w)$ contains two generators whose associated vertices are not adjacent in $\Gamma$, say $x$ and $y$. But then, since $u,v \in \lk(g) \subset \lk(w) \subset \lk(\{x,y\})$, we have that $u$ and $v$ are both adjacent to $x$ and $y$ and so the full subgraph defined by $\{u,v,x,y\}$ in $\Gamma$ is a square, contradicting that the RAAG $\GG$ is coherent.

\begin{lem}\label{lem:descentr}
  Let $\GG$ be a coherent RAAG, $g$ a cyclically reduced element of $\GG$ and $C(g)$ the centraliser of
  $g$ in $\GG$.  
  \begin{enumerate}[label=(\roman*)]
  \item \label{it:descentri}
    If $C(g)$ is non-abelian then $g$ is the product of powers of generators that pair-wise commute.
    In this case 
    $C(g)=\prod_{x\in \alpha(g)}\la x\ra \times \la\lk(g) \ra$. 
  \item\label{it:descentrii}
    If $C(g)$ is abelian then $\lk(g)$ is a clique (possibly empty) and either
    \begin{enumerate}[label=(\alph*)]
  \item\label{it:descentriia} $g$ is a product of powers of generators that pair-wise commute, in which case
    $C(g)=\prod_{x\in \alpha(g)}\la x\ra \times \prod_{y\in \lk(g)}\la y\ra$; or 
  \item\label{it:descentriib}  $g=w^rv$, where $w$ is a  root, block element of length greater than $1$, $r$ is a non-zero  integer,  $v$ is a product of powers of  generators that pair-wise commute and belong to $\lk(w)$,  and $C(g)=\la w\ra \times \prod_{x\in \alpha(v)}\la x\ra \times \prod_{y\in \lk(g)}\la y\ra$.
    \end{enumerate}
    \end{enumerate}
\end{lem}

\begin{proof} 
From the remarks preceeding the Lemma and \eqref{eq:raagcr}, if $C(g)$ is non-Ablelian then
    the blocks of $g$ are powers of
    generators, which commute by definition, giving \ref{it:descentri}.
   From \eqref{eq:raagcr},  when $C(g)$ is Abelian $\lk(g)$ must be a clique.
    As pointed out above, either $g$ has a single block with more than one generator
    in which case we have  \ref{it:descentrii}\ref{it:descentriib}; or $g$ is a product
    of powers of commuting generators, in which case \ref{it:descentrii}\ref{it:descentriia} holds.
\end{proof}

\begin{rem}
Let $\GG(\G)$ be a coherent RAAG. Then if  $g_1,g_2,h_1,h_2$ are elements of $\GG$ so that $[h_1,h_2]\neq 1$, $[g_1,g_2]\neq 1$, then there exists $i,j=1,2$ so that  $[g_i,h_j]\neq 1$.

Indeed, assume to the contrary that $[h_1,h_2]\neq 1$, $[g_1,g_2]\neq 1$ and $[g_i,h_j]=1$ for all $i,j=1,2$. Since $[h_1,h_2]\neq 1$ (correspondingly, $[g_1,g_2]\neq 1$), it follows that the centralisers of $g_1$ and $g_2$ (correspondingly, $h_1$ and $h_2$) are non-abelian as they contain $h_1$ and $h_2$ (correspondingly, $g_1$ and $g_2$. By Lemma \ref{lem:descentr}, it follows that $g_i$ is the product of powers of pair-wise commuting generators, i.e. $g_i=\prod_{x\in \alpha(g_i)} x^{r(x)}$ and $C(g_i)= \prod_{x\in \alpha(g_i)}\la x\ra \times \la \lk(g_i) \ra$; similarly, for $h_i$, $i=1,2$. 

Since $g_2 \notin C(g_1)$, it follows that there exist $x_2 \in \alpha(g_2)$ such that $x_2 \notin \lk(g_1)$. Since by definition, $\lk(g_1)= \bigcap\limits_{x\in \alpha(g_1)} \lk(x)$ and $x_2\notin \lk(g_1)$, it follows that there exists $x_1\in \lk(g_1)$ such that $x_2 \notin \lk(x_1)$, that is $x_1$ is not adjacent to $x_2$ and so $[x_1,x_2]\ne 1$. A symmetric argument shows that there exist $y_i\in \alpha(h_i)$, $i=1,2$ such that $y_1$ and $y_2$ are not adjacent and so $[y_1, y_2]\ne 1$.  

Since $h_i \in C(g_j)$ for $i=1,2$, from the description of centralisers we have that $y_i$ and $x_j$ commute, for $i=1,2$. Furthermore, $y_i\ne x_j$ as $[x_1,x_2]\ne 1$ and $[y_i, x_j]=1$ for $i,j=1.2$. Therefore, $x_1, x_2, y_1,y_2$ are different and the full subgraph that they define is a square -- a contradiction.
\end{rem}

We shall later make use of the following  property of centralisers of sets of generators of a RAAG.
\begin{lem}\label{lem:ACprop}
  Let $\GG$ be a RAAG with canonical generating set $X$ and let $Y\subseteq X$ be a finite set which generates a free abelian subgroup.
  Then  there exists $g\in \la Y\ra$ such that $C(Y)=C(g)$.
\end{lem}
\begin{proof}
  Let $Y=\{y_1,\ldots,y_m\}$ and set $g=y_1\cdots y_m$. Then $C(Y)\subseteq C(g)$ and if $a\in C(g)$ then Lemma \ref{lem:descentr} implies  $a=bc$, where $b\in \la Y\ra$ and $c\in \la \lk(Y)\ra$. As $Y$ is a clique, $[b,y]=1$ and by definition $[c,y]=1$, for  all $y\in Y$, so $a\in C(Y)$.
\end{proof}

\begin{rem}[Representatives in a RAAG]\label{rem:representatives}
A key element of the main construction of this paper is a choice of representatives of centralisers of a group; see
\ref{it:C6}\ref{it:C6a} below. As an initial example we describe how such a set of representatives may be chosen in a coherent RAAG with commutation graph $\G$ and generating set $X=V(\G)$.
To begin with let $\cK=\{C\subset X\,:\, C \textrm{ is a clique}\}$, define an equivalence relation $\sim$ on $\cK$ by $C\sim D$ if and only if $\st(C)=\st(D)$, and for each element $C=\{x_1,\ldots, x_k\}\in \cK$ let $g_C=x_1\cdots x_k$. 
Let $\bar\cK$ denote the set of equivalence classes of $\sim$ and, for each equivalence class $[C]$ of $\sim$ let
$[C]_{\Min}$ be the set of cliques of minimal cardinality in $[C]$.
Let $\st([C])=\st(C)$, where $C$ is some (hence any) element of $[C]$. 
If $\st([C])$ is a clique, set $W_{[C]}=\{g_D\}$, for some $D\in [C]_{\Min}$. If $\st([C])$ is not a clique, set $W_{[C]}=\{g_D\,:\,D\in [C]_{\Min}\}$. Define $W_\cK$ to be the union of the sets $W_{[C]}$, as $[C]$ ranges over $\bar\cK$.
Now let $B$ be the set of cyclically reduced, root, block elements of length at least $2$ in  $\GG$ and let $\sim_B$ be the equivalence relation on $B$ given by $b\sim_B c$ if and only if $c$ is a conjugate of $b$ or $b^{-1}$. Let $W_B$ be a set of representatives of equivalence classes of $\sim_B$ on $B$.
Finally, let $W_\GG=W_\cK\cup W_B$. 
\end{rem}
\begin{expl}
  Let $P_4$ be the path graph on $4$ vertices and
  $\GG=\GG(P_4)=\langle a,b,c,d\mid[a,b],[b,c],[c,d]\rangle$. The cliques of $P_4$ are
  $\{a\}, \{b\},\{c\},\{d\},\{a,b\}, \{b,c\},\{c,d\}$ and $\st(a)=\st(a,b)$, $\st(d)=\st(c,d)$, whence
  $\{a\}\sim \{a,b\}$ and $\{d\}\sim \{c,d\}$, while all other  equivalence classes are singletons.
  Thus $\bar\cK=\{[\{a\}],[\{b\}],[\{c\}],[\{d\}],[\{b,c\}]\}$ and, with an obvious abuse of notation,
    $g_a=a$, $g_b=b$, $g_c=c$, $g_d=d$, $g_{bc}=bc$. As $C_{\min}$ is a singleton for each equivalence
    class,  $W_\cK=\{a,b,c,d,bc\}$.  
\end{expl}

\begin{lem}\label{lem:Wpc}
  Let $\GG$ be the RAAG with commutation graph $\G$ and let $g$ be  an  element of $\GG$.
  Then there is $w\in W_\GG$ such that  $C(g)$ is conjugate to $C(w)$. If $C(g)$ is  abelian then $w$ is unique. 
    \end{lem}
\begin{proof}
  Without loss of generality we may assume $g$ is cyclically reduced. 
  If $C(g)$ is non-abelian  then from Lemma \ref{lem:descentr}, $g$ is the product of powers of generators that pair-wise commute, and so $\alpha(g)$ is a clique.
    Let $D$ be of minimal cardinality in $[\alpha(g)]$. Then $g_D\in W_{\cK}$ and  $C(g_D)=C(g)$, as required. 
   Similarly, if  $C(g)$ is abelian and canonical, $g_D\in W_\cK$, for a unique $D$ of minimal cardinality in $[\alpha(g)]$, and  again $C(g_D)=C(g)$. From the definitions, $C(g_D)\neq C(w)$, for all other $w\in W$, so the lemma holds if $C(g)$ is canonical abelian.
  
  If $C(g)$ is abelian and non-canonical, then $g=b^rx^{r_1}_1\cdots x_k^{r_k}$, for some cyclically reduced
  root, block element $b$ of length at least two, $x_i\in \lk(b)$ and non-zero integers $r,r_i$. 
  From Lemma  \ref{lem:descentr}, $C(g)\le C(b)$. On the other hand if $y\in C(b)$, since Lemma \ref{lem:descentr} implies  $C(b)=\la b\ra\times \la \lk(b)\ra$ and $\lk(b)$ is a  clique, we have $[y,x_i]=1$, for all $i$, so $y\in C(g)$.
  Hence $C(g)=C(b)$. There is a unique element $d\in W_B$ such that $d$ is a conjugate of $b$ or $b^{-1}$ and
  so $C(g)=C(b)=C(d)^h$, for some $h\in \GG$. If $w\neq d$ is an element of $W$ and $C(w)$ is conjugate to $C(d)$, then  $w\in  W_B$, as $C(d)$ is non-canonical, so  some conjugate $w'$ of $w$ belongs to $C(d)$, and  therefore  $w'=d^rv$, for some $v\in \la \lk(d)\ra$. It follows, as both $w$ and $d$ are cyclically reduced, root, block elements,  that $w'=d^{\pm 1}$, so $w\sim_B d$, and by definition of $W_B$, we have $w=d$.
  Hence the lemma holds in all cases.
\end{proof}

The following lemma is well-known.
\begin{lem}\label{lem:coh}
Let $C$ be a coherent group and let $A$ be a finitely generated abelian group. Then $C\times A$ is coherent.
\end{lem}
\begin{proof}
Since $A$ is finitely generated, using induction, it suffices to consider the case when $A$ is cyclic. Let $\pi$ be the natural epimorphism from $C \times A$ onto $C$. The group $\pi(H)$ is finitely presented by coherence of $C$, and the kernel $K$ of $\pi_{|H}$ is cyclic (possibly finite or even trivial) as a subgroup of the cyclic group $A$. Hence
there is an exact sequence $1 \to K \to H \to \pi(H) \to 1$, which shows that $H$ is finitely presented.
\end{proof}

\subsection{Limit groups}\label{sec:limit}

Limit groups have played an important role in the classification of finitely generated groups elementarily equivalent to a free group, see \cite{KM, Sela}.

They can be characterised from many different points of view. In this section, we briefly recall some of these equivalences, see \cite{Daniyarova} for further details.

Let $F(X)$ be a free group with basis $X$ and denote by $G[X]$ the free product $G \ast F(X)$. A \emph{system of equations with coefficients in the group} $G$ is defined as set of formal equalities $\{s(X)=1 \mid s(X) \in S(X)\}$, where $S(X) \subset G[X]$ is a (possibly infinite) subset. A \emph{solution} of the system of equations, is a tuple of elements $\bar g\in G^{|X|}$ such that  $S(\bar g)=1$ in $G$, or equivalently, a solution is a homomorphism from $G[X]$ to $G$ (defined by $X \to \bar g$) such that $S(X)$ is contained in the kernel. The set of solution of a system of equations is called an \emph{algebraic set}.

The group-theoretic counterpart to the notion of a Noetherian ring is the notion of an equationally Noetherian group: a group $G$ is called {\em equationally Noetherian} if every system of equations $S(X) = 1$ with coefficients in $G$ is equivalent to a finite subsystem $S_0(X) = 1$, where $S_0(X) \subset S(X)$, i.e. the algebraic set defined by $S$ coincides with the one defined by $S_0$. It is known, see \cite{BMR}, that all linear groups are equationally Notherian.

Let $G$ and $H$ be groups.  We say that $H$ is \emph{discriminated} by $G$ if for every finite set of non-trivial elements $H_0 \subset H$ there exists a homomorphism $\phi:H\to G$ injective on $H_0$, that is, $h^\phi \neq 1$ for every $h \in H_0$. In this case we also sometimes say that $H$ is \emph{fully residually} $G$. Following the case of free groups, finitely generated fully residually $G$ groups are termed \emph{limit groups over $G$}. 

The universal theory of a group $G$ is the set of all universal first-order sentences (that is sentences with only $\forall$ quantifiers in the prenex normal form) satisfied by $G$. We say that $H$ is a \emph{model of the universal theory} of $G$ if $H$ satisfies all universal first-order sentences satisfied by $G$. 

One can prove, see \cite{AG2}, that if $G$ is equationally Noetherian and $H$ is a $G$-group, then the following statements are equivalent:
\begin{itemize}
	\item $H$ is a limit group;
	\item $H$ is a finitely generated model of the universal theory of $G$.
\end{itemize}

  Limit groups over RAAGs, and their actions on higher dimensional
  analogues of real trees, have been studied in \cite{CRK15}.

\section{The class $\mathcal C$}\label{sec:C}

In this Section we define the class of groups $\mathcal{C}$.  Our general aim is to give a combinatorial description of a class of groups that contains coherent RAAGs and is closed under extensions of centralisers (of some elements), see Definition \ref{defn:extension centraliser}, and direct limits. These closure properties ensure that an appropriate iterated centraliser extension,
(ICE for short) of a group in the class, remains in the class.
One also needs to ensure that the ICE over $G$
admits a $\mathbb Z[t]$-action, that it coincides with the exponential group $G^{\mathbb Z[t]}$, and that it is fully residually $G$.

The definition of the class $\cC$, see Definition \ref{defn:C}, is rather technical.
As RAAGs are described via a specific presentation, the definition of the groups in the class depends on the existence of a generating set for which some properties are satisfied. In particular, condition \ref{it:C4} asks for some minimality of the generating set.

As in the case of hyperbolic groups, we require the groups in the class to be torsion-free;  see condition \ref{it:C1}.
Furthermore, to ensure that extensions of centralisers of elements in a group $G$ from the class are fully residually $G$,
we require the groups to satisfy the BP property, see condition \ref{it:C2}. This condition also 
implies that centralisers are isolated which means, in the presence of  condition \ref{it:C3}, 
 that only extensions of centralisers of root elements need be considered.

As we mentioned in the introduction, for our approach, it is essential that centralisers have a  tractable structure.
We extract conditions on the structure of centralisers of elements in a RAAG so that,
on the one hand, this structure is preserved under iterated centraliser extensions and, on the other hand,
after countably many extensions of centralisers, the centre of a centraliser admits a $\mathbb Z[t]$-action; see conditions
\ref{it:C6} and \ref{it:C8} in Definition \ref{defn:C}. 

To ensure that the $\mathbb Z[t]$-action is well-defined in the group, we require a weak form of malnormality on centralisers,
namely condition \ref{it:C7}.

In some sense, conditions \ref{it:C6}, \ref{it:C7} and \ref{it:C8} generalise the CSA property of torsion-free hyperbolic groups
(see Lemma \ref{lem:csainc}).

\smallskip

We next define the properties required from the groups in the class.

\begin{defn}
The centraliser of an element $g\in G$ is \emph{isolated} if for all $w^n \in C(g)$, $n\in \BZ\setminus{ \{0\}}$ it follows that $w\in C(g)$.
\end{defn}
\begin{defn}\label{defn:bp}
  We say that a $k$-tuple $u = (u_1, \ldots, u_k)$ of elements of a group is \emph{generic} if
$$
[u_i,u_{i+1}] \neq  1 \ \  for  \ \ i = 1, \ldots,  k-1.
$$ 
A  group $G$ is said to have the \emph{big powers} (BP) property  if, for any positive  integer $k$ and any generic $k$-tuple $u = (u_1, \ldots, u_k)$ of non-trivial elements of $G$ there exists an integer $n = n(u)$ such that, for positive integers $\alpha_1, \ldots, \alpha_k$
\begin{equation*}
u_1^{\alpha_1}\cdots u_k^{\alpha_k}=1
\end{equation*}
implies $\alpha_i<n$, for some $i$. 
\end{defn}

All torsion-free abelian groups are BP, as is the direct product of a BP group with a torsion-free abelian group.
From \cite{KMR05}, groups $G$ and $H$ are BP if and only if the free product $G*H$ is BP. However,
if $G$ and $H$ are non-abelian groups then the direct product $G\times H$ is not a BP group. From \cite{Blatherwick}, a RAAG is BP if and only if its commutation graph is chordal. 

If $G$ is a BP group and $g\in G$, then the centraliser of $g$ is isolated \cite{KMR05}. (In fact the centralisers of sets of elements are ``strongly isolated'', but we don't use this.)

\begin{defn}
Let $G = \langle X \rangle$ be a group. We say that a subgroup $K$ is \emph{canonical} (with respect to the generating set $X$) if $ K = \langle X'\rangle$, for some $X' \subset X$.
\end{defn}

\begin{defn} \label{defn:C}
We define $\mathcal C$ to be the class of groups satisfying the following. 
A group $G$ belongs to $\mathcal C$ if and only if $G$ has a presentation with generating set  $X$ such that properties \ref{it:C1}--\ref{it:C8} below hold. 
\begin{enumerate}[label=${\mathcal C}$\arabic*,ref=${\mathcal C}$\arabic*]
\item\label{it:C1} $G$ is torsion-free.
\item\label{it:C2} $G$ satisfies the $BP$-property.
\item\label{it:C3} $G$ has unique roots.
\item\label{it:C4} 

\begin{enumerate}[label=(\alph*),ref=(\alph*)]
\item
\label{it:C4a}
Let $Y$ and $Y'$ be subsets of $X$ such that $[y,y']=1$ for all $y\in Y, y'\in Y'$ and $Y \cap Y' = \emptyset$. 
Then $\langle Y,Y' \rangle = \langle Y \rangle \times \langle Y' \rangle$.
\item\label{it:C4b} If $x\in X$ then $x\in \langle Y \rangle $ implies $x\in Y$, for all $Y\subseteq X$. 
 (In particular, with $Y=\emptyset$ we have $x\neq 1_G$, for all $x\in  X$, that is $1_G\notin X$.)
\end{enumerate}
 
There exists a subset $W$ of $G$ such that the following hold. 

\item\label{it:C6} \begin{enumerate}[label=(\alph*),ref=(\alph*)]
   \item\label{it:C6a} 
     For every $g\in G$, there exist elements $w_1,\ldots ,w_k\in W$, $k\ge 1$,   and $h\in G$ such that $C(g)=h^{-1}C(w_i)h$, for $1\le i\le k$. 
      If  $C(g)$ is abelian then $k=1$.
   \item\label{it:C6b}
     For all $w\in W$,  $C(w)$ can be written as a direct product $Z(w)\times O(w)$ where $Z(w)$ and $O(w)$ are defined as follows.
     \begin{enumerate}[label=(\roman*),ref=(\roman*)]
        \item\label{it:C6bi}
          If $C(w)$ is a canonical abelian subgroup then $Z(w)=1$ and $O(w)=C(w)$;
        \item\label{it:C6bii}
          If $C(w)$ is abelian and non-canonical then $Z(w)$ is cyclic and not canonical and $O(w)$ is the maximal canonical subgroup of $G$ satisfying: for each minimal (by inclusion) subset  $Y\subseteq X$ such that  $Z(w)\subseteq \langle Y\rangle$, 
          \begin{enumerate}
             \item\label{it:C6bii_a}
               every generator of $O(w)$ commutes with each generator of $Y$ and  
             \item\label{it:C6bii_b} no generator of $O(w)$ belongs to $Y$.
             \end{enumerate}
             [Although $X$ may be infinite,  maximal subgroups always exist.]
        \item\label{it:C6biii} If $C(w)$ is non-abelian then $Z(w)$ is defined in \ref{it:C8} below and $O(w)$ is the maximal canonical subgroup of $G$ 
          satisfying: for each minimal (by inclusion) subset  $Y\subseteq X$ such that  $w\in \langle Y\rangle$, 
          \begin{enumerate}
             \item\label{it:C6biii_a}
               every generator of $O(w)$ commutes with each generator of $Y$ and   
             \item\label{it:C6biii_b} no generator of $O(w)$ belongs to $Y$.
          \end{enumerate}
          
     \end{enumerate}
   \item\label{it:C6c} If $w\in W$, 
     $g\in O(w)$ and  $C(g)$ is not conjugate to $C(w)$ then there exist $h\in G$, $w_0\in W$ such that $C(g)=h^{-1}C(w_0)h$ and $h, w_0\in O(w)$.   
\end{enumerate}
\begin{itemize}
\item
There may be several minimal canonical subgroups containing an element of $G$  but, for fixed $w\in W$ the maximal canonical  subgroup $O(w)$ satisfying the properties of \ref{it:C6} \ref{it:C6bii} or \ref{it:C6}  \ref{it:C6biii} is unique (indeed, if there were two maximal subgroups $O_1(w), O_2(w)$, then the group $O(w)=\langle O_1(w), O_2(w)\rangle$ would also satisfy the required properties).
\end{itemize}
\item\label{it:C7} If $w\in W$ and $C(w)$ is abelian, then  $C(w)$ satisfies the property that if $a, a^h \in C(w)$ and $h \notin C(w)$, then $a\in O(w)$ and $[h,a]=1$.
\begin{itemize}
\item
It follows that for all $g\in G$, if $C(g)$ is abelian, $a, a^h \in C(g)$ and $h \notin C(g)$, then  $h \in C(a)$ and $a\in O(g)$.
\end{itemize}
\item\label{it:C8} If  $w\in W$ and $C(w)$ is non-abelian, then the following hold.
\begin{enumerate}[label=(\alph*)]
\item
  $C(w)$ is a canonical subgroup: 
  $C(w) = \langle Y(w)\rangle \times O(w)$, where $Y(w)$ is a  minimal subset of $X$ such that $w\in \langle Y(w)\rangle$ and,  by definition, $\Z(w)=\langle Y(w)\rangle$. \\
\item The centre $Z(C(w))$ of $C(w)$ is a canonical subgroup.
\end{enumerate}
\begin{itemize}
\item
In this case $Y(w)$ is the unique minimal subset of $X$ such that $w\in \langle Y(w)\rangle\le C(w)$. Indeed, 
suppose that $Y_1$ is another such subset. If $y_1\in Y_1$ then $y_1\in C(w)$ implies $y_1\in Y(w)\cup K$ (using \ref{it:C4}\ref{it:C4b}), where $K$ is a set of canonical generators of $O(w)$, and so $y_1\in Y(w)$. As $Y(w)$ is minimal it follows that $Y_1=Y(w)$.    
\item If $C(w)=\langle V\rangle$, for some subset $V$ of $X$, then let $V'$ be a minimal subset of $V$ such that  $w\in \langle V'\rangle$. Then $V'\subseteq Y(w)$ (using \ref{it:C4}\ref{it:C4b}) and \ref{it:C6}\ref{it:C6b}) and minimality of $Y(w)$ implies $V'=Y(w)$.
  Hence $\Z(w)\subseteq Z(C(w))$. If $Z(C(w))$ is finitely generated then $w$ can be chosen such that $\Z(w)=Z(C(w))$.
\end{itemize}
\end{enumerate}
\end{defn}
\begin{rem}\label{rem:ZO} \
  \begin{enumerate}[label=(\arabic*)]
\item\label{it:ZO1} The set $W$ is not uniquely determined by the conditions above, but for fixed $W$ satisfying the conditions, $Z(w)$ and $O(w)$ are uniquely determined, for all $w\in W$. This follows directly from the definitions  for $O(w)$ and  from \ref{it:C6}\ref{it:C6bi} and  \ref{it:C8}, if $C(w)$ is canonical abelian or  $C(w)$ is non-abelian.
If $C(w)$ is abelian and non-canonical then let $Y$ be a minimal subset of $X$ such  that $Z(w)\subseteq \la Y\ra$ and let $Z(w)=\la z\ra$. 
Suppose that $Z'(w)$ also satisfies the conditions of \ref{it:C6}\ref{it:C6bii} with $C(w)=Z'(w)\times O(w)$, let $Z'(w)=\la z'\ra$ and let $Y'$ be a minimal subset of $X$ such that $Z'(w)\subseteq \la Y'\ra$.
From \ref{it:C4}\ref{it:C4b}, $\la Y\cup Y'\cup O(w)\ra=\la Y\cup Y'\ra\times O(w)$.  It follows that $z=z'o$, for some $o\in O(w)$. Then $z'^{-1}z=o\in \la Y\cup Y'\ra\cap O(w)$, so $z=z'$ and $Z(w)=Z'(w)$, as claimed.
We may then choose $w=z$ in this case.
  \item\label{it:ZO2}
    For $g\in G$ let $w\in W$ be such that $C_G(g)=C_G(w)^{h}$, for some $h\in G$.
    If $C(w)$ is abelian and $C_G(g)^{h_1}=C_G(w)^{h_2}$ then $w$ and $w^{h_1h_2^{-1}}$ belong to $C_G(w)$, so \ref{it:C7} implies that $h_1h_2^{-1}\in C_G(w)$, and so $w^{h_1}=w^{h_2}$.
    In this case we  define $Z(g)=Z(w)^h$  and $O(g)=O(w)^{h}$. Then $h_1=zoh_2$, where $o\in O(w)$ and $z\in Z(w)$, so $Z(w)^{h_1}=Z(w)^{h_2}$ and $O(w)^{h_1}=O(w)^{h_2}$, so $Z(g)$ and $O(g)$ are well-defined.

If $C(w)$ is non-abelian, consider first $g\in G$ such that $C(g)=C(w)$. Then $g\in  Z(C(w))$ which is canonical,
so  a minimal subset $Y(g)$ of $X$  such that $g\in \langle Y(g)\rangle$ and $Y(g)\subseteq Z(C(w))$ may be chosen.
Define $Z(g)=\langle Y(g)\rangle$ and define $O(g)$ to be the subgroup of $C(w)$ generated by $(Y(w)\cup K(w))\backslash Y(g)$, where $K(w)$ is a canonical generating set for $O(w)$. Then $C(w)=Z(g)\times O(g)$. In general let $T(w)$ be a transversal (containing $1$) for right cosets of the subgroup $U(w)=\{u\in G\,:\, C_G(w)^u=C_G(w)\}$ in $G$.
When $C_G(g)=C_G(w)^{h}$, let $h=ut$, for $t\in T(w)$, $u\in U(w)$ so $C_G(g)=C_G(w)^t$. Then  $C_G(g^{t^{-1}})=C_G(w)$ and  we may define $\Z(g)=\Z(g^{t^{-1}})^{t}$ and $O(g)=O(g^{t^{-1}})^{t}$.  
\end{enumerate}
\end{rem}

If $G$ belongs to the class $\cC$, has generating set $X$ and subset $W$ satisfying \ref{it:C1}--\ref{it:C8} above, we say that $G$ is in $\cC(X,W)$.

The following lemmas will be useful in Sections \ref{sec:preserve} and \ref{sec:tower}.
\begin{lem}\label{lem:intcent}
  Let $G$ be a group in class $\cC$ and let $w$ and $g$ be elements of $G$ such that $w\in W$, $C(w)$ is abelian and  $g\in C(w)$. If $C(g)$ is non-abelian then $g\in O(w)$ and there exists $w_0\in W\cap O(w)$ such that $C(g)=C(w_0)$ and, using the notation of Remark \ref{rem:ZO}\ref{it:ZO2}, $C(g)=\Z(g)\times O(g)$ is  canonical with $\Z(g)\le O(w)$.
\end{lem}
\begin{proof}
  As $C(g)$ is non-abelian there exists $x\in C(g)$, $x\notin C(w)$, so $g$ and $g^x$ belong to $C(w)$ and, from \ref{it:C7},  then $g\in O(w)$. From \ref{it:C6}\ref{it:C6c}, there exists $w_0\in O(w)\cap W$ and $h\in O(w)$, such that $C(g)=h^{-1}C(w_0)h$, and  as $C(w)$ is abelian, $C(g)=C(w_0)$. From Remark  \ref{rem:ZO}\ref{it:ZO2} we have $\Z(g)\le Z(C(w_0))$. This implies that if $s\in \Z(g)$, then  $w\in C(g)\le C(s)$ and  so $C(s)$ is non-abelian. Hence $s$ and $s^x$ belong to $C(w)$ and $s\in O(w)$, as claimed.
\end{proof}

\smallskip

In the rest of this section, we give examples of groups that belong to the class $\cC$, namely, RAAGs and toral relatively hyperbolic groups.  Recall that a toral relatively hyperbolic group is a torsion-free group which is hyperbolic relative to a finite family $\{A_\lambda\,:\,\lambda \in\Lambda\}$ of finitely generated free abelian groups.

\begin{expl}
  Free abelian groups are in $\cC(X,W)$, where $X$ is a free basis and $W=\{0\}$. In this case  $Z(a)=\{0\}$ and $O(a)$ is the entire group, for all group elements $a$.
\end{expl}
\begin{lem}\label{lem:pcinC}
  A coherent RAAG belongs to $\mathcal C(X,W)$, where $X$ is the vertex set of the   commutation graph of the group and $W$ is the set defined in Remark \ref{rem:representatives}.
\end{lem}
\begin{proof}
  Properties \ref{it:C1}, \ref{it:C3}  and \ref{it:C4} hold for all RAAGs \cite{Baudisch77}. Blatherwick \cite{Blatherwick}  proves that a RAAG satisfies the BP property if and only if it is coherent. Property \ref{it:C6}\ref{it:C6a} follows  directly from Lemma \ref{lem:Wpc}. In the terminology of the preamble to Lemma \ref{lem:Wpc}, if $w\in W_\cK$ then, if $C(w)$ is canonical abelian, then $O(w)=C(w)$; if $C(w)$ is abelian and non-canonical, set $Z(w)$ to be cyclic (and not canonical) generated by the root of $w$ and $O(w)=\langle \lk(g)\rangle$; and if $C(w)$ is non-abelian, set $Z(w)=\la \alpha(w)\ra$ and $O(w)=\la \lk(w)\ra$.
  Then \ref{it:C6}\ref{it:C6b}\ref{it:C6bi},  \ref{it:C6}\ref{it:C6b}\ref{it:C6biii},  \ref{it:C7} and \ref{it:C8} follow immediately from  Lemma \ref{lem:Wpc}.

  To see that \ref{it:C6}\ref{it:C6c} holds, when $w\in W_\cK$, assume that $g\in O(w)$ and $C(g)$ is not conjugate to $C(w)$. Then $g=g_1^{-1}g_0g_1$,  for some $g_i\in O(w)$ such that $g_0$ is cyclically reduced. There is  $w_0\in W$ such that $C(g_0)$ is conjugate to $C(w_0)$ and, as both $g_0$ and $w_0$ are cyclically reduced, it follows from Lemma \ref{lem:descentr} that either $C(w_0)=C(g_0)$ or (as words) $w_0$ is
  a cyclic permutation of $g_0$, and in both cases this implies  $w_0\in C(w)$.  In the case where $C(w)$ is abelian, by definition, $w_0\in O(w)$.  
  Assume then that $C(w)$ is non-abelian.
  If $\alpha(g_0)$ is not a clique, then $w_0\in W_B$ and so by definition $\alpha(w_0)\subseteq \alpha(g_0)$, which implies  $w_0\in O(w)$. If $\alpha(g_0)$ is a clique  then there exists a minimal element $D$ of $[\alpha(g_0)]$ such that  $D\subseteq \alpha(g_0)$ and by definition $g_D\in W_\cK$.
  Taking $w_0=g_D$, we have  $w_0\in O(w)$ and $C(g_0)=C(w_0)$, as required. 
  Hence $w_0\in O(w)$ in all cases and \ref{it:C6}\ref{it:C6c} holds when $w\in W_\cK$. 
 
  If $w\in W_B$ then set $Z(w)=\la w\ra$ and $O(w)=\la \lk(w)\ra$, and \ref{it:C6}\ref{it:C6b}\ref{it:C6biii}
  follows. For  \ref{it:C6}\ref{it:C6c} assume that $g\in O(w)$, which in this case is torsion-free abelian  so $g$ is cyclically reduced. Then $\alpha(w)\subseteq \st(g)$,  so $\alpha(g)$ is a clique and $\st(g)$ is not a clique; and \ref{it:C6}\ref{it:C6c} follows as in the previous case. 
\end{proof}
\begin{expl}\label{ex:C6c}
  To see that in Lemma \ref{lem:pcinC} the set $W$ cannot be simplified in such a way that
  every centraliser (of an element) is
  conjugate to the centraliser of a unique element of $W$, consider the graph $\G$ of
  Figure \ref{fig:C6c} and the group $\GG=\GG(\G)$. There are non-abelian centralisers
  $C(d_1d_2)=\la a, c_1,c_2,d_1,d_2\ra$, $C(ad_2)=\la a,b_2,c_1,c_2,d_1,d_2\ra$, with
  $O(d_1d_2)=\la a, c_1,c_2\ra$ and $O(ad_2)=\la b_2,c_1,c_2,d_1\ra$. Then $ac_1\in O(d_1d_2)$ and
  $d_1c_1\in O(ad_2)$ and $C(ac_1)=C(d_1c_1)=\la a,b_1,c_1,d_1,d_2\ra$. The elements of $\GG$ with
  centraliser equal to $C(d_1d_2)$ are the elements of $\la d_1,d_2\ra$. Indeed, for $v$ a vertex of $\G$,  
  $\st(d_1,d_2)\subseteq \st(v)$ if and only if $v=a$, $d_1$ or $d_2$. As $\st(a,d_i)\neq \st(d_1,d_2)$ the claim
  follows from Lemma \ref{lem:descentr}. Similarly, the elements of $\GG$ with centraliser equal to $C(ad_2)$ are
  the  elements of  $\la a,d_2\ra$. For all elements $g\in \la d_1,d_2\ra$ we have $O(g)=O(d_1d_2)$ and
  for all $h\in \la a,d_2\ra$ we have  $O(h)=O(ad_2)$. To satisfy \ref{it:C6}\ref{it:C6c} the set $W$ must then contain
  an element $w_1$ conjugate to $ac_1\in O(g)$ and an element $w_2$ conjugate to $ac_1$ in $O(h)$; forcing 
  $C(w_1)$ and $C(w_2)$ to be conjugate. 
\end{expl}
\begin{figure}
\begin{center}
\begin{tikzpicture}[scale=1]%
  \tikzstyle{every node}=[circle, draw, fill=black, color=black,
  inner sep=0pt, minimum width=6pt]
  \pgfmathsetmacro{\l}{1.5}
  \foreach \a/\b/\c in {0/a/x,1/b_2/b,2/c_2/y_2,3/d_2/z_2,4/d_1/z_1,5/c_1/y_1,6/b_1/b}
  {
    \pgfmathtruncatemacro{\tmp}{\a+7}  
    \node at ({90+\a*(360/7)}:\l) (\a) {};
    \ifthenelse{\NOT \a=1 \AND \NOT \a=6}{
      \node at ({90+\a*(360/7)}:1.5*\l) (\tmp) {};
      \node[draw=none,fill=none,color=black] at ({90+\a*(360/7)}:1.5*\l+0.3){$\c$};
      \draw (\a) --  (\tmp);}
    {}
    \ifthenelse{\NOT \a=1 \AND \NOT \a=6}{
      \node[draw=none,fill=none,color=black] at ({98+\a*(360/7)}:\l+0.25){$\b$};}
    {\node[draw=none,fill=none,color=black] at ({90+\a*(360/7)}:\l+0.3){$\b$};}
  }
\begin{pgfonlayer}{background}
  \draw (0) -- (1) -- (2) -- (3) -- (4) -- (5) -- (6) -- (0);
  \draw (1) -- (3) -- (0) -- (4) -- (6);
  \draw (3) -- (5) -- (0) -- (2) -- (4);
 \end{pgfonlayer}
\end{tikzpicture}
\end{center}
    \caption{Example \ref{ex:C6c}}\label{fig:C6c}
\end{figure}

\begin{lem}\label{lem:csainc}
  Let $G$ be a non-abelian group, generated by  a finite set $X$, satisfying \ref{it:C1}, \ref{it:C2}, \ref{it:C4} and  the condition that,
  \begin{itemize}
  \item for all $g\in G$, either $C(g)$ is conjugate to a canonical subgroup or  $C(g)=\la g_0\ra$, where  $g_0$ is not a proper power.
 \end{itemize}
  If $G$ is a CSA group then $G$  belongs to $\cC$. 
\end{lem}
\begin{proof}
From \cite{MR95}, as $G$ is CSA and torsion-free it satisfies \ref{it:C3} and, for all $g\in G$,  the centraliser $C(g)$ is maximal abelian and malnormal.
If $C(g)$ is conjugate to a canonical centraliser $C(h)$ then, as $G$ is a CSA group, there exists $x\in X\cap C(h)$ such  that $C(h)=C(x)$. Let $W_0$ be a subset of  $X$ such that if $y,z\in W_0$ then $C(y)$ is not conjugate to $C(z)$ and if  $x\in X$ then $C(x)$ is conjugate to $C(w)$ for some $w\in  W_0$. Then every canonical centraliser $C(g)$ is  conjugate to $C(w)$, for precisely one $w\in W_0$.

Now let $C(g)$ be a non-canonical centraliser, so from the hypothesis $C(g)=\la g_0\ra$,  where $g_0$ is the root of $g$. 
Choose a subset $Y$ of $X$ which is minimal with the property that there exists a root element $g_1\in \la Y\ra$ with $C(g_1)$ conjugate to $C(g_0)$, and define $W'(g)=g_1$. Now let $W'$ be a subset consisting of elements $w$ of $G$ such that  $w=W'(g)$, for  some $g\in G$ and such that no two elements of $W'$ have conjugate centralisers. Using Zorn's lemma the set of all such $W'$ has maximal elements. If $W_1$ is a maximal subset of this form then it follows that, for all $g\in G$, if $C(g)$ is  non-canonical then $C(g)$ is conjugate to $C(w)$, for a unique element of $W_1$. 

Set $W=W_0\cup W_1$, if $w\in W_0$ set $Z(w)=\{1\}$ and $O(w)=C(x)$, and if $w\in W_1$ set $Z(w)=\la w\ra$ and $O(w)=1$. Then \ref{it:C6} \ref{it:C6a} and \ref{it:C6b} follow directly. If $g\in O(w)$, for some non-trivial $g\in G$ and $w\in W$, then $w\in W_0$ and as $G$ is a CSA group $C(g)=C(w)$, so \ref{it:C6} \ref{it:C6c} holds. Finally \ref{it:C7} follows immediately from the CSA property. 
\end{proof}
\begin{lem}\label{lem:torelhypBP}
  Let $G$ be a torsion-free group which is hyperbolic relative to a finite family $\{A_\lambda\,:\,\lambda \in\Lambda\}$ of finitely generated free abelian groups
  ($G$ is toral relatively hyperbolic).
Then $G$ is a BP group.
\end{lem}
\begin{proof}
Let $k$ be a positive integer and let $u=u_1,\ldots,u_k$ be a generic sequence of elements of  $G$. Define $T=\{u_1,\ldots, u_k,[u_1,u_2], \ldots ,[u_{k-1},u_k]\}$. From \cite{Osin06}, if $u_i$ is hyperbolic then $u_i$ is   contained in a maximal elementary subgroup $E(u_i)$ and  $G$ is hyperbolic relative to $\{A_\lambda\}\cup \{E(u_i)\}$.
Moreover, as  $G$ is torsion-free hyperbolic, $E(u_i)$ is cyclic. 
Adding these $E(u_i)$ to the family of peripheral subgroups, we may assume that $u_i$ is contained in a peripheral subgroup $A_i$, $i\in\Lambda$ for all $i$. Theorem 1.1 of \cite{Osin07} states that there exists a finite subset $\mathfrak F$ of non–trivial elements of $G$ with the following property. Let $\mathfrak N = \{N_{\lambda}\}_{\lambda \in \Lambda}$ be a collection of subgroups $N_\lambda \triangleleft A_\lambda$ such that $N_\lambda \cap \mathfrak F = \emptyset$ for all $\lambda \in \Lambda$.  Write $G(\mathfrak N)$ for the quotient $G/\ncl\langle\cup_\lambda N_\lambda\rangle$ (where $\ncl$ denotes the normal closure in $G$). Then, for each $\lambda\in \Lambda$, the natural map from $A_\lambda / N_\lambda$ to  $G(\mathfrak N)$ is injective and $G(\mathfrak N)$ is hyperbolic relative to the collection $\{A_\lambda/N_\lambda\}_{\lambda \in \Lambda}$. Moreover, for any finite subset $S\subset G$,  there exists a finite subset $\mathfrak F(S)$ of non–trivial elements of $G$ such that the restriction of the natural homomorphism  $G \to G(\mathfrak N)$ to $S$ is injective whenever $N_\lambda \cap \mathfrak F(S) = \emptyset$ for $\lambda \in \Lambda$.  Let $T\subset G$ be the finite set defined above and $\mathfrak F(T)$ be the set given by the aforementioned theorem. For each $\lambda\in \Lambda$ let $T_\lambda=(\cF \cup \cF(T)\cup T)\cap A_\lambda$. From \cite{BMR06} it follows that free abelian groups are discriminated by cyclic groups, so for all $\lambda$, there exists a homomorphism $\phi_\lambda$  from $A_\lambda$ to a cyclic group $C_\lambda$ such that $\phi_\lambda$ restricted to $T_\lambda$ is injective. Let $N_\lambda$ be the kernel of $\phi_\lambda$ and let $H=G/\ncl\langle\cup_\lambda N_\lambda\rangle$. Then $N_\lambda\cap\cF(T)=\emptyset$ and from \cite[Theorem 1.1]{Osin07}, the canonical map $\phi$ from $G$ to $H$ induces an embedding from $A_\lambda/N_\lambda$ to $H$. It follows that $\phi(u_i)$ is of infinite order in $H$, for all  $i$, and that $\phi(u_1),\ldots, \phi(u_k)$ is a generic sequence of elements of $H$. From \cite[Corollary 1.2]{Osin07}, the group $H$ is hyperbolic and from \cite{Ol93} hyperbolic groups have the big powers property for tuples of elements of infinite order; so there exists $n(u)$ such that, whenever $\alpha_i\ge n(u)$, for all  $i$, we have $\phi(u_1)^{\alpha_1}\cdots \phi(u_k)^{\alpha_k}\neq 1$.  It follows that $u_1^{\alpha_1}\cdots u_k^{\alpha_k}\neq 1$, and therefore $G$ is a BP group. 
\end{proof}
\begin{cor}\label{cor:torelhypinC}
  Let $G$ be a torsion-free group which is hyperbolic relative to a finite family $\{A_\lambda\,:\,\lambda \in\Lambda\}$ of finitely generated free abelian groups.
  Then $G$ is in $\cC$.
\end{cor}
\begin{proof}
  From Lemma \ref{lem:torelhypBP}, the group $G$ is in BP and 
 \ref{it:C1} and \ref{it:C4} hold from the definitions.  The centraliser of a non-trivial element $g$ satisfies that either $C(g)=A_\lambda$ for some $\lambda\in \Lambda$, or $C(g)=\la g_0\ra$, where $g_0$ is the root of $g$, and from  \cite[Lemma 2.5]{Groves09}, such groups are CSA.
  Hence we may apply Lemma \ref{lem:csainc}.
  \end{proof}

\section{Preservation operations}\label{sec:preserve}
\begin{defn}[Extension of Centralisers]\label{defn:extension centraliser}
Let $G$ and $H$ be  groups, $u\in G$, $C=C_G(u)$ and  $\phi:C\rightarrow H$ a monomorphism such that $\phi(u)\in Z(H)$. 
The \emph{extension of the centraliser} $C$ by $H$ (\emph{with respect to }$\phi$) is 
\[G(u,H)=G\ast_\phi H,\] 
the group with relative presentation $\la G, H\,|\, \phi(g)=g, \forall g\in C\ra$.

If $H=\phi(C)\times A$, for some subgroup $A$ of $H$, then the extension is said to be \emph{direct}. 
\end{defn}
An element of an amalgamated free product is said to be cyclically reduced if it has no reduced form which begins
and ends with an element from the same factor. Every element of an amalgam is conjugate to a cyclically reduced element
(\cite[Theorem 4.6]{MKS}).
\begin{thm}\label{thm:lift}
  Let $G \in \mathcal C$ (with respect to the generating set $X_G$ and subset $W_G$ satisfying \ref{it:C1}--\ref{it:C8} above)
  and let 
  $u\in W_G$ be
  such that  $C_G(u)$ is abelian.
  Let $B$ be a free abelian group, let $\phi:C_G(u)\rightarrow B$ be a monomorphism such that $B=\phi(C_G(u))\times A$,
  for some $A\le B$, and let $G(u,B)=G\ast_\phi B$ be the direct centraliser extension of $C=C_G(u)$ by $B$. Then the following hold.  
\begin{enumerate}[label=(\arabic*)]
\item \label{it:lift1}
 $G(u,B) \in \mathcal C$, with respect to the generating set $X(u,B)=X_G\cup X_A$, where $X_A$ is a free generating set for $A$, and cyclically reduced elements $W(u,B)=W_G\cup W_*$, where  $W_*$ is the set of elements $z\in G(u,B)$ that satisfy the following conditions. 
 \begin{enumerate}[label=(\roman*)]
\item\label{it:w*i} $z$ is a cyclically  reduced,  root element with a factorisation $z=g_1a_1\cdots g_ra_r$, where $r\ge 1$,  $g_i\notin C$, $a_i\in A$, $a_i\neq 0$, for all $i$, such that;
\item\label{it:w*ii} if $Y(z)$ is a minimal subset of  $X_G$ such that $g_i\in \langle Y(z)\rangle$, for all $i$, and
  $O'(z)=\cap_{i=1}^r C_G(g_i) \cap C_G(u)$ then $\langle Y(z),O'(z)\rangle=\langle Y(z)\rangle\times O'(z)$.
\item\label{it:w*iii} If $z$ satisfies \ref{it:w*i} and \ref{it:w*ii} above then exactly one element of  the set of elements $v\in G(u,B)$ such that $v$ is conjugate to $z$ or $z^{-1}$ and  satisfies \ref{it:w*i} and \ref{it:w*ii}, 
  belongs
  to $W_*$. 
  \end{enumerate}
 
\item \label{it:lift2}
  Assume $\phi(u)=v\in B$ and identify the cyclic subgroup $\langle v\rangle$ of $B$ with $\BZ$. Let
  $\psi_i : B \rightarrow \BZ$, $i\in I$,
  be a discriminating family of (additive group) homomorphisms of $B$ by its subgroup $\BZ$, indexed by a set $I$.
  For $(i,m)\in I\times \BN$, define $\lambda_{i,m}: G(u,B) \longrightarrow G$ to be the
  homomorphism induced by  the identity homomorphism on $G$ and the composition of the inverse image of $\phi$
  with the $m$-th scalar multiple $m\psi_i$  of the homomorphism $\psi_i$ on $B$:
  that is $\lambda_{i,m}(g)=g$ and $\lambda_{i,m}(b)=\phi^{-1}(m\psi_i(b))=u^{m\psi_i(b)}$,
  for $g\in G$ and $b\in B$.
Then $G(u,B)$ is discriminated by $G$ via the family $\lambda_{i,m}$, $(i,m)\in I\times \BN$.
\end{enumerate}
\end{thm}
\begin{proof} Note that, although there may be more than one choice of  $w\in W_G$  such that $C_G(w)=C$, the group $G(u,B)$ and subsets $X(u,B)$, $W(u,B)$, $Z_G(u)$ and $O_g(u)$  are independent of the choice $w=u$. \\[1em]
  \ref{it:lift2} 
Let $\ww=(w_1,\ldots, w_{n+1})$ be an $n+1$ tuple of elements of $G$, such that $w_i\notin C$, for $i\ge 2$, and let $\bm=(m_1,\ldots, m_n)$ be an $n$-tuple of non-zero integers. Then, setting $u_i=u^{m_i}$ the tuples  $(w_1,\ldots, w_{n+1})$ and $(u_1,\ldots, u_{n})$ satisfy the  condition that $[u_i^{w_{i+1}}, u_{i+1}] \neq 1$ in $G$; since the fact that centralisers in $G$ are isolated implies  $[u_i^{w_{i+1}}, u_{i+1}] = 1$ if and only if $({u^{m_i}})^{w_{i+1}} \in C(u_{i+1})=C$ if and only if $u^{w_{i+1}}, u \in C$ if and only if $[w_{i+1}, u] = 1$. 
Since $G$ satisfies the BP-property there exists an integer $K(\ww,\bm)$ such that 
$w_1u_1^m\cdots u_n^mw_{n+1}\neq 1$, for all $m\ge K(\ww,\bm)$.

Let $g$ be a non-trivial element in $G(u,B)$. One can write it in a reduced form
\[
g=w_1a_1w_2 \cdots a_nw_{n+1},
\]
where  $1 \neq a_i \in A$ and $w_i \in G$, for all $i$,   and $w_i \not\in C_G(u)$ for $i=2, \ldots, n$, and either
$w_{n+1}=1$ or $w_{n+1}\notin C_G(u)$.
Let $\psi_j$ discriminate the elements $a_1, \ldots, a_n$ in $\BZ$. Define $m_i=\psi_j(a_i)$,  for $i=1,\ldots,n$. 
By definition,
\[\lambda_{j,m}(a_i)= u_i^m =u^{mm_i}\in C,\]
so, with $\ww$, $\bm$ and $K(\ww,\bm)$ as above, $\lambda_{j,m}(g)\neq 1$, for all $m\ge K(\ww,\bm)$. 
(If $w_{n+1}=1$, then we replace $g$ with $a_nga_n^{-1}$ and obtain  $\lambda_{j,m}(a_nga_n^{-1})\neq 1$,
so again $\lambda_{j,m}(g)\neq 1$.)
Thus, we can separate any given non--trivial element $g \in G(u,B)$ by  $\lambda_{j,m}$, for all $m \geq K(\ww,\bm)$.

Consequently, if we have a finite number of elements $g_1, \ldots, g_k \in G(u,B)$, say 
\[
g_i=w_{i,1}a_{i,1}w_{i,2} \cdots a_{i,n_i}w_{i,n_i+1},
\]
then, choosing $\psi_j$ to discriminate the elements $a_{1,1}, \ldots, a_{k,n_k}$ in $\BZ$, it follows that 
$\lambda_{j,m}$ discriminates $g_1, \ldots, g_k$, as long as $m \geq max\{K(\ww_1,\bm_1), \ldots, K(\ww_k,\bm_k) \}$, where $\ww_i$ and $\bm_i$ have the obvious definitions.  Hence $G(u,B)$ is discriminated by $G$.

\medskip
\ref{it:lift1}: 
\ref{it:C1} and \ref{it:C2}.\\
Since $G(u,B)$ is discriminated by $G$, it follows that $G(u,B)$ is torsion-free and since $G$ satisfies the BP-property, so does $G(u,B)$ (\cite[Lemma 11]{BMR02}), so \ref{it:C1} and \ref{it:C2} hold.
\medskip

\ref{it:C4}\ref{it:C4a}.\\
Let $Y$ and $Y'$ be disjoint and commuting  subsets of $X(u,B)$. Then  $Y=L_1\cup L_2$ and $Y'=L_1'\cup L_2'$, where $L_1, L_1'\subset X_G$ and  $L_2,L_2'\subset X_A$. If $Y,Y'\subseteq G\cup B$, then immediately from the definitions \ref{it:C4}\ref{it:C4a}  holds. Assume then that $L_1$ and $L_2$ are both non-empty and that $a\in L_2$ ($a\neq 1$). Then, $[a,y']=1$ and $a\in A$, so $y'\in C$ for all $y'\in L_1'$. Hence $Y'\subseteq C\times A=B$. If also $L_2'\neq \emptyset$, then $Y\subseteq B$, a contradiction. 
Hence, without loss of generality, we have $Y'\subseteq C$. In this case if  $w\in \langle Y\rangle \cap \langle Y'\rangle$, then $w\in \langle Y'\rangle < C$ implies that $w$ has syllable length $0$, and so $w\in \langle Y\rangle \cap G=\langle L_1\rangle$, so $w\in \langle L_1\rangle\cap \langle Y'\rangle$. 
Since $L_1 \cap Y'=\emptyset$ and $L_1\cup Y'\subseteq X_G$, it follows from \ref{it:C4}\ref{it:C4a} in $G$ that $w=1$ and so $\langle Y,Y'\rangle = \langle Y\rangle \times \langle Y' \rangle$. 

\medskip

\ref{it:C4}\ref{it:C4b}.\\
Let $Y\subset X(u,B)$ and let $x\in X(u,B)\cap \langle Y\rangle$. Write $Y_G=Y\cap X_G$ and $Y_A=Y\cap X_A$ and let $x=g_0a_0  \cdots g_n a_n$ be an expression for $x$ with $a_i\in \langle Y_A\rangle$ and $g_i\in\langle Y_G\rangle$.  If $n >0$ then $a_n=1_B$ and $g_n\in C$, so $x=g_0a_0\cdots g_{n-1}g_na_{n-1}$; and  continuing this way we see $x=g_0'a_0$, with $g'_0\in \langle Y_G\rangle$. 
If $x\in X_A$ then $g'_0\in C$ and $x=\phi(g'_0)a_0$, from which it follows that $\phi(g'_0)=1_B$ and  then that $x\in Y_A$. If $x\in X_G$, then $a_0=1_B$ and $x\in Y_G$, by \ref{it:C4}\ref{it:C4b} in $G$.  
\medskip

\ref{it:C3}.\\
Let $g\in G(u,B)$ be in a conjugate of a factor. Then, without loss of generality, we may assume that $g$ belongs to a factor; and it follows that $g$ has a unique root, using \ref{it:C3} in the case $g\in G$. Hence we may assume that $g$ is not in a conjugate of a factor.  First note that for all  $h,f\in G(u,B)$ and $s\ge 1$, if $h^s=f^s$  then $[f^s,h^s]=1$, so $[f,h]=1$,  as centralisers in $G(u,B)$ are isolated. Hence $(fh^{-1})^s=1$ and, as $G(u,B)$ is torsion-free, $h=f$. Now if $g$ is  not in a conjugate of a factor then,  without loss of generality, we may assume  that $g$ is cyclically reduced so, for all $h$ and $r$ such that $g=h^r$, the syllable length of $g$ is $k_0|r|$, where $k_0$ is the syllable length of $h_0$. Therefore there  is a unique maximal positive integer $r(g)$, such that  $g=h^{r(g)}$, for some $h\in G(u,B)$. Suppose that $g=h^r=f^s$, for some $h,f\in G(u,B)$, where $r=r(g)$  and $1\le s \le r$. 
Let $d=\gcd(r,s)$ and $a,b\in \BZ$ be such that $d=ar+bs$. Then $h^d=h^{ar}h^{bs}=g^ah^{bs}=(f^{a}h^{b})^s$ and, setting $h_1=(f^{a}h^{b})^{s/d}$, we have $h^d=h_1^d$, so $h=h_1$. Thus $h= (f^{a}h^{b})^{s/d}$ and, by maximality of $r(g)$, we have $s=d$, so $s|r$ and $f^s=(h^{r/s})^s$ implies $f=h^{r/s}$. Therefore $g$ has a unique root $\sqrt{g}=h$.

To verify that \ref{it:C6}--\ref{it:C8} hold a description of centralisers in $G(u,B)$ is needed.  
We will use the following description of commuting elements in free products with amalgamation 
(see \cite{MKS}): 
if $[x, y] = 1$ then one of the following conditions holds:
\begin{enumerate}[label=\Roman*.,ref=\Roman*]
\item\label{it:colli} $x$ or $y$ belongs to some conjugate of the amalgamated subgroup $C = C_G(u)$;
\item\label{it:collii}  neither $x$ nor $y$ is in a conjugate of $C$, but $x$ is in a conjugate of a factor ($G$ or $B$), in which case $y$ is in the same conjugate of the factor;
\item\label{it:colliii}  neither $x$ nor $y$ is in a conjugate of a factor, in which case $x = g^{-1}cg z^n$ and  $y = g^{-1}c'gz^m$, where $c, c' \in C$, and $g^{-1}cg, g^{-1}c'g$ and $z$ commute pairwise.
\end{enumerate}

Considering each of these three possible cases in turn we shall prove the following lemma and show that \ref{it:C6}--\ref{it:C8} hold in $G(u,B)$.
\begin{lem}\label{lem:centdesc}
  Let $v$ be a cyclically reduced element of $G(u,B)$ and let $C(v)$ denote the centraliser of $v$ in $G(u,B)$.
  \begin{enumerate}[label=(\roman*)]
\item\label{it:centdesca} If $v\in  C$ then either
  \begin{enumerate}[label=(\alph*)]
  \item\label{it:centdescai} $v\notin O_G(u)$, in which case $C(v)=C\times A=B=C(u)$; or
  \item\label{it:centdescaii} $v\in O_G(u)$, in which case $C(v)=C(w)$, for some $w\in W_G$, and
    $C(v)=\Z_G(v)\times \langle O_G(v),A\rangle$.
    \end{enumerate}
  \item\label{it:centdescb}  If $v\in G\backslash C$ then $C(v)=C_G(v)$ and if $v\in A$, and is not the identity, then $C(v)=B=C(u)$.
  \item\label{it:centdescc} If $v$ is not in $G\cup B$ then there exists $z\in W_*$ such that $C(v)=\langle z\rangle \times O'(z)$
    (as defined in Theorem \ref{thm:lift} \ref{it:lift1} \ref{it:w*ii} above). 
    \end{enumerate}
    In case \ref{it:centdesca}\ref{it:centdescaii} $C(v)$ is non-abelian and canonical, while in all other cases $C(v)$ is abelian. 
\end{lem}

\noindent \textbf{Case \ref{it:colli}}. 
If $x$ belongs to some conjugate of the amalgamated subgroup $C=C_G(u)= Z_G(u) \times O_G(u)=\phi(C)\le B$:  without loss of generality we may  assume that $x \in C$. 
In this case, as $C$ is abelian $B=C\times A \subseteq C(x)$.

Let $[x,y]=1$ and let $y=g_1a_1 \cdots g_n a_n$ be any reduced form of $y$, that is $a_i\in A$, $g_i \in G$, $a_i\neq 1_A$, for $i<n$, and for $i\ge 2$, $g_i\notin C$. From $yx=xy$ and the theory of amalgamated products, it follows that either ($n=1$ and $g_1\in C$) or ($n\ge 1$, $g_n\notin C$ and $x^{g_n^{-1}} \in C$). In the latter case, as \ref{it:C7} implies that $g_n\in C_G(x)$, the centraliser of $x$ in $G$ is non-abelian: $g_n \in C_G(x), u \in C_G(x)$, but $[g_n,u]\ne 1$. Hence, if $C_G(x)$ is abelian  
\begin{equation}\label{eq:xinCAb}
  y\in C\times A, \textrm{ and } C(x)=C\times A=C(u).
\end{equation}

Now assume that $C_G(x)$ is non-abelian, so contains $y$ as above where $n\ge 1$ and $g_n\notin C_G(u)$, so $g_n \in C_G(x)$.
From Lemma \ref{lem:intcent}, $C_G(x)=\Z_G(x)\times O_G(x)$ is canonical, there exists $x_0\in O_G(u)$ such that
$C_G(x)=C_G(x_0)$ and $\Z_G(x)\le O_G(u)$. 

Since $g_n \in \Z_G(x) \times O_G(x)$, we have $g_n = x_n o_n$, where $x_n \in \Z_G(x) < O_G(u)$ and $o_n \in O_G(x)$. Then $y = y' x_n o_n a_n$ and $[x,y']=1$. By induction on syllable length we have that $y'= K_{n-1} O_{n-1}$, where $K_{n-1}\in \Z_G(x)$, $O_{n-1} \in \langle O_G(x), A\rangle$. Then 
\[
y = K_{n-1} O_{n-1}  x_n o_n a_n =K_nO_n,
\] 
where $K_n=K_{n-1}x_n \in \Z_G(x)$ and  $O_n=O_{n-1}o_na_n \in\langle O_G(x), A\rangle$. 

Hence,  if $C_G(x)$ is non-abelian then there exists $x_0\in W_G$ such that 
\begin{equation}\label{eq:xinCnAb}
  C(x)= \Z_G(x) \times \langle O_G(x), A \rangle =C(x_0)\textrm{ with  } x,x_0\in O_G(u).
\end{equation}
In particular,  for any $x\in C$, $C_G(x)$ is abelian if and only if $C(x)$ is abelian, and if $C_G(x)$ is canonical, so is $C(x)$. 

Now let us verify that conditions \ref{it:C6} to \ref{it:C8} hold for $x$ in Case I. 
If $C_G(x)$ is abelian, so $C(x)=C(u)$ is given in \eqref{eq:xinCAb}, define $Z(u)=Z_G(u)$ and $O(u)=O_G(u)\times A$. As $C_G(u)\neq C_G(v)$, for all $v\in W$, $v\neq u$, it follows that if $w\in W$ and $w\neq u$ then $C(w)\neq C(u)$. Hence \ref{it:C6}\ref{it:C6a} and \ref{it:C6}\ref{it:C6b} hold. 

For \ref{it:C6}\ref{it:C6c}, if $g\in O(u)$ let $g=oa$, where $o\in O_G(u)$ and  $a\in A$. From \ref{it:C6}\ref{it:C6c} in $G$ and the fact that $C_G(u)$ is abelian it follows that $C_G(o)=C_G(w_0)$, for  some $w_0\in W_G\cap O_G(u)$ and from \eqref{eq:xinCnAb}, if $a$ is trivial, then $C(g)=C(w_0)=\Z_G(w_0)\times \langle  O_G(w_0),A\rangle$, and $w_0\in W$; so  \ref{it:C6}\ref{it:C6c} holds. Otherwise $a$ is non-trivial and $g\in B$ so  $C(g)=B=C(u)$ and  \ref{it:C6}\ref{it:C6c} follows immediately.

To see that \ref{it:C7} holds, when $C(x)$ is abelian, assume that $w\in C(u),  w^h \in C(u)$ and $h\notin C(u)$. Then  $w = w_G a$, where  $a\in A$ and $w_G \in C$.
Also,  $h\notin C(u)$ implies $h= g_1 a_1 \cdots g_n a_n$, where $g_i\in G$, $a_i\in A$ and either $n=1$ and $g_1\notin C$; or $n>1$ and $g_i\notin C$, for $i\ge 2$.
 Then
 \[w^h= a_n^{-1} g^{-1}_n \cdots a_1^{-1} g_1^{-1} w_G a g_1 a_1 \cdots g_n a_n\in C(u)=C\times A\]
so $g_1^{-1} w_G a g_1\in C$, whence $a=1$ or $g_1\in C$.
If $n=1$ then $g_1\notin C$ so $a=1$, $w=w_G$ and $w_G^{g_1}\in C$, so $w_G\in O_G(u)\subseteq O(u)$ and $[g_1,w_G]=1$; so $[h,w]=1$.
 If $n>1$ and  $g_1\in C$ then 
\[g^{-1}_n \cdots a_1^{-1} w_G a a_1 \cdots g_n=g^{-1}_n \cdots g_2^{-1} w_G a g_2 \cdots g_n \in C\times A\]
and, by induction on syllable length, $w\in O(u)$ and $[w, g_2 \cdots g_n]=1$; so $[w,h]=1$.

On the other hand, if $g_1\notin C$, as before  $a=1$ and $[g_1,w_G]=1$ so  
\[g^{-1}_n \cdots g_2^{-1} w_G  g_2 \cdots g_n \in C\times A.\]
As
$g_2 \cdots g_n\notin C(u)$,  by induction $[w_G,g_2 \cdots g_n]=1$. 
It follows that in all cases $[w,h]=1$. 

Thus \ref{it:C6} and \ref{it:C7} hold in Case I when $C_G(x)$ is abelian.

Assume then that $x\in C$ and $C_G(x)$ is non-abelian, so \eqref{eq:xinCnAb} describes $C(x)=C(x_0)$.
If $w\in W$ such that $C_G(w)=C_G(x_0)$, then $w\in W_G$ since, as  we shall see in the final case of this proof, $v\in W_*$ implies that $C(v)$ is abelian. Hence for all such $w$ we have $w\in W_G$ and we may set $Z(w)=\Z_G(w)$ and  $O(w)=\langle O_G(w), A\rangle$.
If $w\neq x_0$ note that $\Z_G(w)\le Z(C_G(w))=Z(C_G(x_0))\le C_G(u)$, so $C_G(w)=\Z_G(w)\times \la O_G(w),A\ra$. 
Hence,  in this case \ref{it:C6}\ref{it:C6a}, \ref{it:C6}\ref{it:C6b} and \ref{it:C8} follow  directly from the definitions and we defer consideration of  \ref{it:C6}\ref{it:C6c} until we have completed our description  of centralisers of elements of $G(u,B)$, below.

  \smallskip

  \textbf{Case \ref{it:collii}}. If $x$ is in a conjugate of $G$ but not in a conjugate of $C$ then $x\in G^g$, 
  for some $g\in G(u,B)$. In this  case  $C(x^{{g}^{-1}})=C_G(x^{{g}^{-1}})$ and   \ref{it:C6} and \ref{it:C8} hold for $C(x)$, as they hold in $G$.
  To see that \ref{it:C7} holds assume that $C(x)$ is abelian and conjugate to $C(w)=C_G(w)$, for some $w\in W_G$, and that  $b$ and $b^h$ are in  $C(w)$, but $h\notin C(w)$. If $h\in G$, then \ref{it:C7} implies $b\in O_G(w)=O(w)$ and $[h,b]=1$, as  required. Assume  then that  $h=g_1a_1\cdots g_na_n$ in reduced form, with $a_i\in A$ and $g_i\in G$, and $a_1\neq 1$.
 Then $b^h\in C(w)\le G$ implies that $b^{g_1}\in C$ and so $[a_1,b^{g_1}]=1$ and then $b^h=(b^{g_1})^{h_1}\in C(w)$,  where $h_1=g_2a_2\ldots g_na_n$.  Hence $(b^{g_1})^{g_2}\in C$ and  $g_2\notin C$, so  $[g_2,b^{g_1}]=1$.
  Continuing this way we obtain $[g_i,b^{g_1}]=1$, $2\le i\le n$, so $b^h=b^{g_1}\in C(w)$. If $g_1\notin C(w)$ then,  from \ref{it:C7} in $G$, we have $[b,g_1]=1$, so $[b,h]=1$, and $b\in O_G(w)=O(w)$. Otherwise $g_1\in C(w)$ so  $b=b^{g_1}=b^h\in C$ and then $b, b^u\in C_G(w)$, but $u\notin C(w)$, as $w\notin C$, so again $b\in O(w)$ and $[b,h]=1$.
  Therefore, \ref{it:C7} holds for $w$ in $G(u,B)$.

 If $x$ is in a conjugate of $B$, but not in $C$, then $C(x)$ is conjugate to $C(u)$ and we have Case I again.

\smallskip

\textbf{Case \ref{it:colliii}}. If $x$ does not belong to any conjugate of a factor, without loss of generality we may assume that $x$ is cyclically reduced (as an element of the amalgamated free product $G(u,B)$) and its reduced factorisations begin with an element of $G\backslash C$. 
Let $x=g_1a_1 \dots g_n a_n$ be any reduced form of $x$, that is, 
$n\ge 1$,  $g_i \in G\backslash C$ and $a_i\in A\backslash \{1_A\}$, for $i=1, \ldots, n$.    

Let $O'(x)=\bigcap\limits_{i=1, \dots, n} C_G(g_i) \cap O_G(u)$. We claim that $O'(x)$ is a  canonical subgroup of $G$ with canonical generating set  $U_x=X_G\cap O'(x)$, and that  $C(x) =  O'(x) \times \langle z_x \rangle$,
where $z_x$ is the unique root element ($\sqrt{z_x}=z_x$) such that
\begin{itemize}
\item  $x=dz^m_x$, with $m\ge 1$, $d\in O'(x)$; and 
\item there is a minimal subset $Z$ of $X_G$ such 
that $z_x\in \langle Z,X_A\rangle$, and $ \langle U_x,Z,X_A \rangle = O'(x) \times \langle Z,X_A\rangle$. 
\end{itemize}

By the description of centralisers in amalgamated products, if $[x,y]=1$, then $x= d v^l$ and $y = d' v^m$, 
where $d,d'\in C$ and $d,d', v$  pairwise commute. As $x$ is cyclically reduced, so is $v$ and so both $l$ and $m$ divide $n$. 
As $[d,x]=[d',x]=1=[d,u]=[d',u]$ but $[x,u]\neq 1$ (as $C(u)=C\times A$), both  $C(d)$ and $C(d')$ are non-abelian, so  from Lemma \ref{lem:intcent},  $\Z_G(d)\cup \Z_G(d')\subseteq O_G(u)$. 
Since $v \in C(d)$, it follows from \eqref{eq:xinCnAb} that $v= d_1 w$ where $d_1 \in \Z_G(d)\le O_G(u)$ and 
$w= g_1'a_1' \dots g_r' a_r' \in O(d)=\langle O_G(d),A\rangle$; that is $g_i' \in O_G(d)$ and $a_i'\in A$. 
 Since $v = d_1 w$, we have 
$xv^{-l}d^{-1}=xw^{-l}d_1^{-l}d^{-1}=1$, so $a'_r=a_n$ and 
$g_r' g_n^{-1} \in C$.  Since $d \in C$ and $C$ is abelian, we deduce that $[g_r' g_n^{-1}, d]=1$. Now since $[g_r', d]=1$, we have that $[g_n, d]=1$ and so $d\in C_G(g_n)$. 
Continuing this process, we conclude that $d \in C_G(g_i)$ and, similarly, $d'\in C_G(g_i)$, for $i=1, \dots, n$. 
Hence $d,d'\in O'(x)$.

Now, let $p$ be any element of $O'(x)$. Then $p\in O_G(u)$ so Lemma \ref{lem:intcent} gives $C_G(p)=\Z_G(p)\times O_G(p)$ and $\Z_G(p)\le O_G(u)$. 
Let $s\in \Z_G(p)$. Then $s\in O_G(u)$ and, by definition, $\Z_G(p)\le Z(C_G(p))$ so $[p,g_i]=1$ implies $[s,g_i]=1$.
Hence $s\in O'(x)$ and, as $\Z_G(p)$ is canonical, it follows that $O'(x)$ is a canonical subgroup of $G$.

Let $U_x=O'(x)\cap X_G$, the canonical generating set for $O'(x)$ and let $x=dv^l$, where $d\in C$ and $[d,v]=1$, so as above  $d\in O'(x)$ and $v=d_1w$,  where $d_1 \in \Z_G(d)$ and $w\in \langle O_G(d),A\rangle$. Using the description of $C(d)$ again, and the uniqueness of roots in $G(u,B)$,  if $r(w)=q$ and $\sqrt{w}=w_0$, so $w=w_0^q$, then we may factorise $w_0= g_1'a_1' \dots g_r' a_r'$ with  $g_i' \in O_G(d)$ and $a_i'\in A$. Let $Y$ be a minimal subset of $X_G$ such that $g_i'\in \langle Y\rangle$, for $i=1,\ldots ,r$, let $Y_1=U_x\cap Y$ and let $Y_2=Y\backslash Y_1$. Suppose $s\in U_x$ and $t\in Y_2$. As $s\in O'(x)$ we have, as above, $[s,g'_i]=1$, so $g_i'=x_i'o_i'$, with $x_i'\in \Z_G(s)$ and $o_i'\in O_G(s)$, for $i=1,\ldots, r$. 
By minimality of $Y$, there is $i$ such that $g_i'\in \langle Y\rangle$ but $g_i'\notin \langle Y\backslash\{t\}\rangle$. For such $i$ we have $g_i'=x_i'o_i'$, with $x_i'\in \Z_G(s)$ and $o_i'\in O_G(s)$, and as $t\notin \Z_G(s)$ (since $t\in Y_2$) it follows that $t\in O_G(s)$; so $[s,t]=1$. Therefore $\langle U_x,Y_2\rangle = O'(x)\times \langle Y_2\rangle$. For all $i$, by definition  $g_i'\in \langle Y\rangle \le \langle U_x,Y_2\rangle$ so we have $g_i'=x_io_i$, $x_i\in O'(x)$, $o_i\in  \langle Y_2\rangle$ and we may further factorise $w_0=d_2w_1$, where $d_2=x_1\cdots x_r$ and $w_1=o_1a'_1\cdots o_ra_r'$. Thus $x=dv^l=d_3w_1^{lq}$, where $d_3=dd_1^ld_2^{lq}\in O'(x)$ (as $d_1\in \Z_G(d)\le C$) and $w_1\in \langle Y_2, X_A\rangle$. Replacing $w_1$ by $w_1^{-1}$ if necessary we may also assume $lq\ge 1$.
Hence $C(x)=\langle w_1\rangle \times O'(x)$; which is abelian.

This  implies that $C(x)\le C(w_1)$. Applying this argument to the element $w_1$ in place of $x$, with
initial reduced factorisation $w_1=o_1a'_1\cdots o_ra_r'$ and $Y=Y_2$, gives $C(w_1)=\langle w_1\rangle \times O'(w_1)$, where $O'(w_1)= \bigcap\limits_{i=1, \dots, r} C_G(o_i) \cap O_G(u)$.
If $b\in O'(w_1)$ then, for all $i$, we have $[b,o_i]=1$ and $g_i'=x_io_i$, with $x_i\in O_G(u)$, which is abelian, so $[b,g'_i]=1$ and $b\in C(x)$.
As before, as $b\in C$ we obtain $[b,g_i]=1$, so $b\in O'(x)$. Conversely, if $b\in O'(x)$, then by definition of $o_i$, we have $[b,o_i]=1$, so we conclude that $O'(x)=O'(w_1)$. 
Therefore
$C(x)=C(w_1)=\langle w_1\rangle \times O'(w_1)$, with $O'(x)=O'(w_1)$.

Now suppose $x=d'z^m$, where $m\ge 1$, $d'\in O'(x)$, $z$ is a root element, $z\in \langle Y', X_A\rangle$ for some minimal subset $Y'$ of $X_G$, such that $Y'\cap U_x=\emptyset$ and $[y',y'']=1$, for all $y'\in Y'$ and $y''\in U_x$. Then $d_3^{-1}d'=w_1^{lq}z^{-m}$; so from \ref{it:C4}\ref{it:C4a} and the disjointness of $U_x$ and $Y_2\cup Y'\cup  X_A$ we have $d_3=d'$ and $w_1^{lq}=z^m$. As $w_1$ and $z$ are root elements  and both $m$ and $lq$ are positive we have $w_1=z$, completing our claim.

Given the validity of the claim, to show in addition that $w_1$ or $w_1^{-1}$ is in $W_*$ it is enough to show that given a minimal subset $Y'$ of $X_G$ such that $o_i\in \langle Y'\rangle$, for all $i$,
we have $\langle Y'\cup U_x\rangle=\langle Y'\rangle\times\langle U_x\rangle$. To see this, given such a subset
$Y'$ let $Y_2'=Y'\backslash (Y'\cap U_x)$, so  as above, $\langle Y_2'\cup U_x\rangle=\langle Y_2'\rangle\times\langle U_x\rangle$ and we may write $o_i=x_i'o_i''$, with $x_i'\in O'(x)$ and $o_i''\in \langle Y_2'\rangle$. Therefore $x=d'_3w_1'^{lq}$, where $d_3'=(x_1'\cdots x_r')^{lq}d_3\in O'(x)$ and $w_1'=o_1''\cdots o_r''\in \langle Y_2'\rangle$. From the previous paragraph it follows that $d_3'=d_3$ and $w_1'=w_1$, so $Y'_2=Y'$ and $\langle Y'\cup U_x\rangle=\langle Y'\rangle\times\langle U_x\rangle$, as required. Hence either $w_1$ or $w_1^{-1}\in W_*$. This completes the proof of Lemma \ref{lem:centdesc}. 

Now we are in a position to verify properties \ref{it:C6} and \ref{it:C7}.
From the previous two paragraphs it follows that $C(x)$ is conjugate to $C(w_1)$ for a unique element $w_1$
of $W_*$, so \ref{it:C6}\ref{it:C6a} holds. 
 To simplify notation we may now assume $x=w_1=o_1a_1\cdots o_ra_r\in W$, with $o_i\in \la Y_2\ra$, $a_i\in A$,  so $x\in W_*$, and let $U_x$ be a canonical generating set for $O'(x)$ as before.
 Then $C(x)$ is abelian and non-canonical and we set $Z(x)=\langle x\rangle$ and $O(x)=O'(x)$. For \ref{it:C6}\ref{it:C6b}\ref{it:C6bii}  let  $V$ be a minimal subset of $X$ such  that $x\in \langle V\rangle$, so $x$ has a reduced factorisation
  $x=h_1b_1\cdots h_rb_r$ with $h_i\in \la V\cap  X_G\ra$ and $b_i\in \la V\cap X_A\ra$.  
Setting $V_G=X_G\cap V$, it is enough to show that 
$\la V_G\cup U_x\ra=\la V_G\ra\times\la U_x\ra$. Let $V_1=V_G\cap U_x$ and $V_2=V_G\backslash V_1$. Since, from the above, $O'(x)=O_G(u)\cap \bigcap_{i=1}^rC(h_i)=\la U_x\ra$ we have $h_i=y_ip_i$, $y_i\in \la V_1\ra$, $p_i\in \la V_2\ra$ and may rewrite $x$ as $x=d_2'w_1'$, where $d_2'=y_1\cdots y_r\in O'(x)$ and $w_1'=p_1\cdots p_r$. As in the proof of the final part of the claim above, we have  $d_2'=1$, $w_1'=x$ and $V_2=V_G$, and   \ref{it:C6}\ref{it:C6b}\ref{it:C6bii}follows. 
If $g\in O(x)$ then $g\in O_G(u)$, so $C(g)$ is given by \eqref{eq:xinCnAb}, and as \ref{it:C6}\ref{it:C6c} is satisfied
by the centraliser $C_G(g)$ in $G$ it holds also for $C(g)$. 

To see that $C(x)$ satisfies property \ref{it:C7}, assume that $h\notin C(x)$, $v\in C(x)$ and $v^h\in C(x)$.
Then, for some integers $p,q$ and elements $b,c\in O'(x)$ we have $v=bx^p$ and $h^{-1}vh=cx^q$. This means that  $bx^p$ is conjugate to $cx^q$ and both are cyclically reduced. From the conjugacy criterion for free products with amalgamation, $p=q$ and  $bx^p$ is obtained from  $cx^p$ by cyclic permutation, followed by conjugation by an element of $C$. As $b\in O'(x)$, it follows that  $x$ has a factorisation $x=x_1x_2$, where $x_1$ and $x_2$ are both reduced and begin with elements of $G$, such that $cx^p=d^{-1}b(x_2x_1)^pd$, for some $d\in C$. As $C$ is abelian, we have $b^{-1}cx^p=  d^{-1}(x_2x_1)^pd$. As $O_G(u)$ and $O'(x)$ are canonical we may choose a minimal subset $D$ of $X\cap O_G(u)$ such that $d\in \langle D \rangle$, and write $D=D_1\cup D_2$, where $D_2=D\cap O'(x)$ and $D_1\cap D\backslash D_2$. Then $d^{-1}(x_2x_1)^pd=d_1^{-1}(x_2x_1)^pd_1^{-1}$, so we  may assume $d=d_1$ and $D=D_1$. We have $x\in \langle Y(x)\cup X_A\rangle$ and it follows that
$d^{-1}(x_2x_1)^pd\in \langle D\cup Y(x)\cup X_A\rangle$ and, by definition of $Y(x)$, that
$\langle O'(x)\cup D\cup Y(x)\cup X_A\rangle =  O'(x)\times \langle D\cup Y(x)\cup X_A\rangle$ (using \ref{it:C4}\ref{it:C4a}).
Therefore $b^{-1}c=1$, and we have $h^{-1}bx^ph=bx^p$. However, $h\notin C(x)$, so this implies that $p=0$, so $v=x^pb=b\in O(x)$ and $[h,v]=1$. That is \ref{it:C7} holds for $C(x)$. 

Finally we  show that \ref{it:C6}\ref{it:C6c} holds when $x\in W_G\cap C$ and $C(x)$ is non-abelian, in which
case \eqref{eq:xinCnAb} implies $C(x)=\Z_G(x)\times \langle O_G(x),A\rangle$.
Assume  $v\in O(x)$ and that $C(v)$ is not conjugate to $C(x)$.
Either $v\in O_G(x)$, $v\in A$ or $v$ has reduced factorisation of length at least $2$.
In the case $v\in O_G(x)$, from \ref{it:C6}\ref{it:C6c} in $G$, there exist $v_0,v_1$ in $O_G(x)$, such that $v_0\in W_G$ and
$C_G(v)=C_G(v_0)^{v_1}$. If $v_0\in W_G\backslash C$ then $C(v_0)=C_G(v_0)$ and so $C(v)=C(v_0)^{v_1}$, as required. 
If $v_0\in C$  and $C_G(v_0)$ is non-abelian then $v_0\in O_G(u)$ and, from \eqref{eq:xinCnAb},
$C(v^{v_1^{-1}})=C(v_0)=\Z_G(v^{v_1^{-1}})\times\langle O_G(v^{v_1^{-1}}), A\rangle$, so $C(v)=C(v_0)^{v_1}$.

If $v_0\in C$ and $C_G(v_0)$ is abelian, or if $v\in A$, then $C(v_0)=C\times A=C(u)$ and 
so to show that \ref{it:C6}\ref{it:C6c} holds it suffices to show that $u\in O(x)$.
From  \eqref{eq:xinCnAb}, we have $u\in C(x)=\Z_G(x)\times O(x)$ and $\Z_G(x)\le O_G(u)$. If $Y$ is a minimal subset of $X$ such that $u\in \la Y\ra$, then $Y\cap O_G(u)=\emptyset$, so $Y\cap \Z_G(x)=\emptyset$, whence $u\in O(x)$,
as required. 

Otherwise $v$ has a reduced factorisation of length at least $2$ and we  choose $U$ to be a minimal subset of $X$ such that $v\in \langle U\rangle \le O(x)$. Then we may write $v=v_1^{-1}v_0v_1$, where $v_0,v_1\in \langle U\rangle$ and $v_0$ is cyclically reduced (as an element of the free product with amalgamation $G(u,B)$) and belongs to $O(x)$.
Thus $C(v)=C(v_0)^{v_1}$ and there is an element $z\in W_*$ such that $C(v_0)=C(z)$, with $v_0=dz^p$, for some $d\in O(z)=O'(v_0)$ and non-zero $p\in \BZ$. As $x,d\in C$ we have $[x,d]=1$,  so  $z\in C(x)$ (as centralisers are isolated). 
To demonstrate that \ref{it:C6}\ref{it:C6c} holds it suffices to show that $z\in O(x)$. 
As $z\in W_*$, we may assume that $z$ has a factorisation $z=g_1a_1\cdots g_ra_r$ and there is a minimal subset
$Y(z)$ of $X_G$ satisfying conditions \ref{it:w*i} and \ref{it:w*ii} of Theorem \ref{thm:lift}.\ref{it:lift1}. As $v_0$ belongs to $O(x)=\la O_G(x),A\ra$ it follows that $g_i\in C_G(x)$, for all $i$, and as $C_G(x)$ is canonical we may assume that $Y(z)\subseteq C_G(x)$. 
 As $\Z_G(x)\le Z(C(x))$, if $s\in \Z_G(x)$ then $[s,g_i]=1$, as $g_i\in C_G(x)$, and $[s,u]=1$, as $\Z(x)\le O_G(u)$; so $s\in O'(z)=O(z)$. Hence $\Z_G(x)\le O'(z)$ and $Y(z)\cap \Z_G(x)=\emptyset$. Therefore \ref{it:C4}\ref{it:C4b} implies  $Y(z)\subseteq O_G(x)$, and so $z\in O(x)$. 
\end{proof}

\begin{cor}\label{cor:attSinC}
Let $G$ be a non-abelian group in $\mathcal C(X_G,W_G)$   and let $u\in W_G$  such that  
$C_G(u)=\la u \ra\times O_G(u)$ is abelian and non-canonical.
    Let $\mbf{x}=\{x_1,\ldots ,x_m\}$ be a set disjoint from $G$,  either let $D=\prod_{i=1}^g[x_{2i-1},x_{2i}]$, where $m$ is even, $m\ge 4$ and $g=m/2$, or let $D=\prod_{i=1}^m x_i^2$, where $m\ge 3$,  and let $S$ be the  surface group with presentation $\la \mbf{x}, u\,|\, D=u\ra$.  Write $C=C_G(u)$ and define $T$ to be the free product with  amalgamation
  \[T=G\ast_{C} (S\times O_G(u)),\]
  identifying $u\in G$ with $u\in S$. 
  Let $Z=X\cap O_G(u)$ be the canonical generating set for $O_G(u)$ and let $z\in Z$. 
  Then
  \begin{enumerate}
  \item $S\times O_G(u)$ is in $\cC(\mbf{x} \cup Z,W_S\cup \{z\})$, where $W_S$ is a subset of the free group on $\mbf x$ and  $z\in Z$, and
  \item $T$ is  in $\cC(X_G\cup \mbf{x}, W_T)$, where $W_T=W_G\cup W_S\cup W'$, for some subset $W'$  of the set of  elements  of $T$ of reduced length at least $2$. 
\end{enumerate}
  \end{cor}
  \begin{proof}
    For the first statement note that $S$ is free on $\mbf x$ and $O_G(u)$ is free abelian with basis $Z$, 
    whence $S\times O_G(u)$ is a coherent RAAG. From Lemma \ref{lem:pcinC},  then $S\times O_G(u)$ is in $\cC$, with respect to   the generating set $\mbf x \cup Z$ and the set, $W_S'$ say, defined above \ref{lem:Wpc}. From the definition, for some $z\in Z$,   the set $W_S'$ is the union of  $W_{\cK}=\mbf x \cup \{z\}$ with a set $W_B$ of cyclically reduced, root elements of the free group on   $\mbf x$, of length at least $2$. Taking $W_S=\mbf x \cup W_B$, we have the first statement.

To prove the second statement  we shall show that $T$ embeds in an extension of centralisers of $G$, and use this together with Theorem \ref{thm:lift} to prove the statement.
From the hypotheses, $C$ is a maximal abelian subgroup of $G$; so the assumptions of
\cite[Lemma 3.2]{KM12} hold. Let $\rho: T\to G$ be the natural retraction with $\rho(S)$ being non-abelian, see \cite{CRK15}.
It is shown in \cite[Section 7]{CRK15} that $\rho$ is injective on both vertex groups of $T$, $G$ and $S\times O_G(u)$. 
Set $T^*=\langle G, y\mid [C_{G}(\rho(C_S(u))),y]\rangle$.
Then, by \cite[Lemma 3.2 (1)]{KM12} the group $T$ embeds into $T^*$ via the map $\psi$ given by $\psi(g)=g$, for all $g\in G$ and $\psi(s)=(\rho(s))^y$, for all $s\in \la \mbf{x}\ra$. Moreover, as $C_S(u)=C\subseteq G$, we have
\begin{equation}\label{eq:T*} T^*=\langle G, y\mid [C,y]\rangle=G*_C(C\times \la y\ra).\end{equation}
From Theorem \ref{thm:lift}, $T^*\in \cC(X,W)$, where set $X=X_G\cup \{y\}$, $W=(W_G\backslash W_{G,C})\cup \{y\}\cup W_*$, and $W_{G,C}$ and  $W_*$ are defined in the statement of the Theorem. Also, $T^*$ is discriminated by $G$, whence $T$ is discriminated by $G$ and it remains to show that $T$ is in $\cC$. 
 Properties \ref{it:C1} and \ref{it:C2} hold in $T$ as they hold in $T^*$. To see that property \ref{it:C3} holds in $\psi(T)$ it suffices  to show that if $w$ is an element of $\psi(T)$ then the unique root of $w$ in $T^*$ belongs to $\psi(T)$. Suppose then that $w=\psi(t)$, for  some $t\in T$, say $t=g_1s_1 \cdots g_ms_m$ in reduced form, with $g_i\in T$ and $s_i\in S\times O_G(u)$. Then $w=\psi(t)=g_1t_1^y\cdots g_mt_m^y$, where $t_i=\rho(s_i)$. Moreover this is a reduced factorisation of $w$, since the given factorisation of $t$ is reduced. In this case, for all $c\in C$ and $j$ such that $1\le j\le m$, $g_jc$, $c^{-1}g_j$, $(c^{-1}t_j)^y$ and $(t_jc)^y$ are in $\psi(T)$, and it follows that if $w=g'_1(t'_1)^y\cdots g_jc (c^{-1}t)^yg'_m(t'_m)^y$, is any reduced factorisation of $w$ in $T^*$, then $g_i'\in G$ and $t_i'\in \psi(S \times O_G(u))$. Now suppose  that $t$ is cyclically reduced, in which case so is $w$,  that the root of $w$ in $T^*$ is $w_0$ and that $w_0^d=w$.
 If $w_0$ has reduced factorisation $w_0=h_1y_1\cdots h_ky_k$, with $h_i\in G$ and
 $y_k\in \la y\ra \times C$ then $w$ has reduced factorisation $(h_1y_1\cdots h_ky_k)^d$. As $w$ is in $\psi(T)$ it follows that the $y_i$ are of the form $c_iy^{-1}$ when $i$ is odd and $c_iy$ when $i$ is even, with $c_i\in C$. It follows that 
 $w_0=h'_1y^{-1}h'_2y \cdots h_{2l-1}'y^{-1}h_{2l}'y$, where $h_i'\in G$ and $h_i'\in \rho(S\times O_G(u))$, whenever $i$ is even.
 That is, $w_0$ is in $\psi(T)$. As roots in $T^*$ are unique, $\psi^{-1}(w_0)$ is the unique root of $t$ in $T$. That is, cyclically reduced elements of $T$ have unique roots, whence all elements of $T$ have unique roots, confirming that \ref{it:C3} holds.

 To see that \ref{it:C4} \ref{it:C4a} holds assume that $Y$ and $Y'$ are disjoint commuting subsets of $X_G\cup \mbf x$. If $Y\cap \mbf x\neq \emptyset$, then $Y'\subseteq O_G(u)$ and $\la Y\ra\cap \la Y'\ra\subseteq O_G(u)$. 
  Otherwise $Y$ and $Y'$ are contained in $X_G$. In both cases  we have  $\la Y\ra\cap \la Y'\ra=\{1\}$, using \ref{it:C4} \ref{it:C4a} in $G$.
 Again, \ref{it:C4} \ref{it:C4b} follows from the fact that it holds in both $G$ and $S\times O_G(u)$. 

 To verify that \ref{it:C6}, \ref{it:C7} and \ref{it:C8} hold we consider the centraliser of an element $v\in T$. If $v$ is in $G$ but is not conjugate  to an element of $C$ then $C_T(v)=C_G(v)$ and it follows that there are elements $w\in W_G$ and $h\in G$ such that $C_G(v)=C_G(w)^h$.
 Also, $w$ cannot be in $C$ so $C_T(w)=C_G(w)$ and so $C_T(v)$ is conjugate to $C_T(w)$, as required. 
Moreover, \ref{it:C6}, \ref{it:C7} and \ref{it:C8} hold for $w$ in $T$, as they hold for $w$ in $G$. If 
$v\in S\times O_G(u)$ but is not conjugate to an element of $C$, then
$C_{S\times O_G(u)}(v)$ is conjugate to $C_{S\times O_G(u)}(w)$, for some $w\in W_S\cup \{z\}$, and $w\neq z$, so $w\in W_S$.

Suppose now that $v$ is conjugate to an element of $C$. Without loss of generality we may assume $v\in C$, so $C_G(v)\ge C$. In this case, if $b\in C_T(v)$ and has reduced factorisation $b=s_1g_1\cdots s_mg_m$, with $g_i\in G$ and $s_i\in \la \mbf x\ra$, then either $m=1$ and $s_1\in C$; or $m\ge 1$ and  $s_m\notin C$. In the first case $g_1s_1\in C_G(v)$, so $g_1\in C_G(v)$. In the second case $v^{g_m}\in C$, so either $g_m\in C\le C_G(v)$ or $g_m\notin C$, and then \ref{it:C7} in $G$ implies $v\in O_G(u)$ and  $g_m\in C_G(v)$. In all cases then $v^{s_m}\in C$,  and \ref{it:C7} in $S\times O_G(u)$ implies  $v\in O_G(u)$ and  $C_G(v)$ is non-abelian.
On the other hand if $v\notin O_G(u)$ a similar argument shows that $C_T(v)=C_G(v)$ and from \ref{it:C7} and \ref{it:C8} in $G$ and the fact that $C$ is abelian, it follows that $C_G(v)$ is abelian, so equal to $C$.
Now if $C_G(v)$ is abelian then $C_T(v)=C_T(u)=C$, and $u\in W_T=W_G\cup W_S\cup W'$. In this case \ref{it:C6} holds for $u$ in $T$, as it holds in $G$. For \ref{it:C7}, suppose  $g\in C$ and $b\in T$ such that $g^b\in C$ and $b\notin C$.
If $b\in G$ or $S\times O_G(u)$, then we have $b\in C_G(g)$ and $g\in O_G(u)$, from \ref{it:C7} in these groups.  Assuming then that $b$ has a reduced factorisation $b=s_1g_1\cdots s_mg_m$, as above with $m\ge 1$,  it follows again that $g\in O_G(u)$ and $g_i\in C_G(g)$, so $b\in C_G(g)$, whence \ref{it:C7} holds for $u$ in $T$.

If $v\in O_G(u)$ then $C_T(v)\ge \la C_G(v),\mbf x\ra$. Repeating the argument with a reduced factorisation of an element $b\in C_T(v)$ we see that $C_T(v)=\la C_G(v),\mbf x\ra$. We have $C_G(v)=\Z_G(v)\times O_G(v)$ and from the definitions
$\Z_T(v)\cap \la O_G(v),\mbf x\ra=\{1\}$ and $[\Z_T(v),\la O_G(v),\mbf x\ra]=1$, so setting $\Z_T(v)=\Z_G(v)$ and $O_T(v)=\la O_G(v),\mbf x\ra$
we have $C_T(v)=\Z_T(v)\times O_T(v)$. As the only generators of $T$ not in $X_G$ are those in $\mbf x$, both \ref{it:C6} \ref{it:C6a} and \ref{it:C6b} follow. We defer consideration of \ref{it:C6} \ref{it:C6c} till the end of the proof. To see \ref{it:C8} holds, note that $Z(C_T(v))=O_G(u)$, which is canonical. 

Finally suppose that $v\notin G$ or $S\times O_G(u)$, and let $v=g_1s_1\cdots g_ms_m$, with $g_i\in G$ and $s_i\in \la\mbf x\ra$ be a reduced factorisation of $v$. Without loss of generality we may assume that $v$ is cyclically reduced. Then, setting $t_i=\rho(s_i)$, it follows that $\psi(v)=g_1t_1^y\cdots g_mt_m^y$ is cyclically reduced in $T^*$. Let $W_*$ be the subset of $T^*$ defined in Theorem \ref{thm:lift} and define $W'=\psi^{-1}(W_*)$. 
From the proof of Lemma \ref{lem:centdesc}, $\psi(v)$ is conjugate to an element $dz^m$, say $\psi(v)^h=dz^m$, for some $z\in W_*$, $d\in C$, $h\in T^*$, with $C_{T^*}(\psi(v))^h=C_{T^*}(z)=\la z\ra\times O_{T^*}(z)$
and $O_{T^*}(z)\le C$. As both $\psi(v)$ and $dz^m$ are cyclically reduced, $dz^m$ may be obtained from $\psi(v)$  by cyclic permutation followed by conjugation by an element of $C$. Hence $h$ and $dz^m$ are in $\psi(T)$ and as $d\in C$, we have $z^m\in \psi(T)$ whence,  from the argument verifying \ref{it:C3},  $z\in \psi(T)$.  Therefore $C_{T^*}(z)\le \psi(T)$ and 
$C_{\psi(T)}(z)=C_{T^*}(z)\cap\psi(T)=C_{\psi(T)}(z)$. As $T^*\in \cC$ both \ref{it:C6} and \ref{it:C7} hold for $z\in T^*$ and hence in $\psi(T)$. Setting $w=\psi^{-1}(z)$ and $g=\psi^{-1}(h)$ we have $C_T(v)^g=C_T(w)$, with $w\in W'$, and \ref{it:C6} and \ref{it:C7} hold for $w\in T$, as $\psi$ is injective. 

To complete the proof  we show that \ref{it:C6} \ref{it:C6c} holds when $v\in O_G(u)$, so $C_T(v)$ is non-abelian. Then $\psi(v)=v$ so $C_{T^*}(v)$ is non-abelian and $C_{T^*}(v)=\Z_G(v)\times O_{T^*}(v)$, while $C_{\psi(T)}(v)=\psi(\Z_G(v)\times \la O_G(u),\mbf x\ra)=
\Z_G(v)\times \la O_G(u),\rho(\mbf x)\ra$. As $C_{\psi(T)}(v)=C_{T^*}(v)\cap \psi(T)$ and $\Z_G(u)\le \psi(T)$ this implies
$\la O_G(u),\rho(\mbf x)\ra=O_{T^*}(v)\cap \psi(T)$. If $a\in O_T(v)$, then $b=\psi(a)\in O_{T^*}(v)\cap \psi(T)$, so there exists $w_0\in W$ and $h\in T^*$, both in $O_{T^*}(v)$, such that $C_{T^*}(b)=C_{T^*}(w_0)^h$. From the above descriptions of centralisers, we also have $w_0$ and $h$ in $\psi(T)$, whence $w_0,h\in  \la O_G(u),\rho(\mbf x)\ra$. Therefore $w_1=\psi^{-1}(w_0)$ is in $W_T$ and $h_1=\psi^{-1}(h)$ is in $T$, both $w_1$ and  $h_1$ belong to $O_T(v)$ and $C_T(a)=C_T(w_1)^{h_1}$.
\end{proof}

\begin{thm}\label{thm:inC}
  If $G_i$ in $\cC(X_i,W_i)$  is a family of groups then the free product $\ast G_i$ is in $\mathcal C(X,W)$, where $X=\cup_{i}X_i$
  and $W=\cup_{i}W_i \cup W_*$, where $W_*$ is the set of all cyclically reduced root elements of $\ast G_i$).
  
  If $G$ is in $\cC(X_G,W_G)$ and $A$ is a free abelian group with basis $X_A$ then the direct product $H=G\times A$
  is  in $\mathcal C(X_G\cup X_A,W_G)$.

 If $\{A_v\}_{v\in V(\Gamma)}$ is a family of free abelian groups and $\Delta$ is a chordal graph, then the 
graph product $\GG(\Delta, \{A_v\})$ is  in $\mathcal C$.
\end{thm}
\begin{proof}
Let $G_i$ be in $\mathcal C$ and let $G=\ast G_i$. Then \ref{it:C1} and \ref{it:C3} follow immediately. For finite families $G_i$, property \ref{it:C2} follows from \cite{KMR05} and induction, and in general, if $u$ and $w$ are a pair of tuples, as in the definition of BP groups above, then the support of $u$ and $v$ is a finite sub-family of $G_i$, and the BP property follows since it holds for this finite sub-family. If $g\in G$ and $g$ is in a conjugate of some $G_i$, then \ref{it:C6} -- \ref{it:C8} hold because they hold in $G_i$. If $g$ is not in a conjugate of a factor then $g$ is conjugate to a unique cyclically reduced element $g_1$ and the centraliser of $g$ is conjugate to $C_G(g_1)=\langle g_0\rangle$, where $g_0=\sqrt{g_1}$, so is abelian, and setting 
$Z(g_0)=\langle g_0 \rangle$ and $O(g_0)=\{1\}$, we see that $g_0\in W$  and \ref{it:C6} -- \ref{it:C8} hold.

Conditions \ref{it:C1},  \ref{it:C2} and \ref{it:C3} for $H$ follow immediately. 
As the generating set for $G\times A$ is $X_G\cup X_A$, property \ref{it:C4} lifts from $G$ and $A$ to $H$. 
For $g\in G$ and $a\in A$, we have $C_H(ga)=C_G(g)\times A=C_H(g)$. If  $g\in W_G=W_H$ we set
$Z_H(g)=Z_G(g)$ and $O_H(g)=O_G(g)\times A$ and 
\ref{it:C6}, \ref{it:C7} and \ref{it:C8} follow. 

Let $X_v$ be a basis for $A_v$ (not necessarily finite), for each vertex $v$. 
  Expanding each vertex $v$ of $\Delta$ to a complete  subgraph on $|X_v|$ we may regard $G$ as a RAAG
  defined by a  chordal graph.
  (In more detail, for each vertex $v$ of $\Delta$ let $K_v$ be a complete graph with vertex set $X_v$. Let $\Delta'$ be the graph  obtained from the disjoint union of the $K_v$, over $v\in V(\Delta)$, by adding edge joining $x\in X_u$ to  $y\in X_v$ for all vertices $u\neq v$ of $\Delta$ and all $x\in X_u$, $y\in X_v$.)
  From Corollary \ref{lem:pcinC}, then $G$ is in $\cC$. 

\end{proof}

\begin{prop}\label{prop:dirlim}
  Let $\{G_i,\lambda^i_j\}$ be a direct system of groups and monomorphisms, indexed by a set $I$, where
  $\lambda^i_j$ maps $G_i$ to  $G_j$,  for $i\le j$,  and let  $\bar G$ be the direct limit  of this system. Assume that
  \begin{enumerate} 
\item\label{it:dirlim1} $G_i$ is in $\cC(X_i,W_i)$ and $\lambda_j^i(X_i)\subseteq X_j$, for all $i,j\in I$; 
  and 
\item\label{it:dirlim2} for all $j\in I$ and $g\in G_j$ there exists $i\ge j$ such that, for all $k\ge i$, the centraliser
  $C_{G_k}(\lambda^j_k(g))=\lambda^i_k(C_{G_i}(\lambda^j_i(g))$. In other words, centralisers of elements of the $G_i$ are eventually  stable.
  \end{enumerate}
  Let $\psi_i$ be the canonical homomorphism from $G_i$ to $\bar G$, let $\bar X=\bigcup_{i\in I} \psi_i(X_i)\subseteq \bar G$ and let  $\bar W=\bigcup_{i\in I} \psi_i(W_i)\subseteq \bar G$. 
Then $\bar G$ is in $\cC(\bar X,\bar W)$ is in $\cC$. 
\end{prop}
\begin{proof}
  From the definition, $\bar G$ is generated by $\bar X$. If $g\in \bar G$ then (taking a representative on its equivalence class we may  assume that) $g\in G_j$, for some $j\in I$. Moreover, $j$ may be chosen large enough so that every element involved in  \ref{it:C1} and  \ref{it:C2} is also in $G_j$. As \ref{it:C1} and \ref{it:C2} hold in $\lambda^i_j(G_j)$, for  all $i\ge j$, they also hold in $\bar G$. 
  Similarly, if $Y, Y'\subseteq \bar X$, satisfy the conditions of \ref{it:C4}\ref{it:C4a} and
  $y\in \langle Y\rangle \times \langle Y'\rangle$, then there exist finite subsets $Y_0\subseteq Y$ and  $Y'_0\subseteq Y'$  such that $y\in  \langle Y_0\rangle \times \langle Y'_0\rangle$. Choosing $j$ large enough to contain $Y_0$ and $Y'_0$  it follows, from the  fact that $G_j$ is in $\cC$, that $y=1$ in $G_j$. Hence \ref{it:C4}\ref{it:C4a} holds. A similar argument shows that
  \ref{it:C4}\ref{it:C4b} and \ref{it:C6} \ref{it:C6a} hold. 
  Moreover   if $x\in G_j$ and  $\psi_i(x)=\psi_j(x')$, for some  $j$ and $x'\in G_j$, then there exists $k$ in $I$ such that $k\ge i,j$ and so  $\lambda^i_k(x)=\lambda^j_k(x')$.
  Hence if $x$ has unique root $y$ in $G_j$ then $\lambda^i_k(x)$ and $\lambda^j_k(x')$ both have unique
  root $\lambda^i_k(y)$ in $G_k$; and $\psi(y)$ is the unique root of $\psi(x)$ in $\bar G$; giving 
  \ref{it:C3}.

  If $g\in G_i$ then we may assume that $i$ has been chosen large enough so that, for all $k\ge i$, we have
  $C_{G_k}(\lambda^i_k(g))=\lambda^i_k(C_{G_i}(g))$. As 
  \ref{it:C6}\ref{it:C6b},
  \ref{it:C6}\ref{it:C6c}, \ref{it:C7} and  \ref{it:C8} hold in $G_i$ it follows that they
  hold in $\bar G$. Therefore $\bar G\in \cC$.
\end{proof}

The following definition is taken from \cite{MR95} 
to which the reader is referred for further details.
\begin{defn}\label{def:tec}
  Let $\mC=\{C_i=C_G(g_i)\}_{i\in I}$ be a set of centralisers in the group $G$ and let $\phi_i:C_i\maps H_i$ be
  a monomorphism from $C_i$ to a group $H_i$, such that $\phi_i(g_i)\le Z(H_i)$, for all $i\in I$. 
  Let $T$ be the tree with vertices $v, v_i$, and  directed edges
  $e_i=(v,v_i)$, from $v$ to $v_i$, for $i\in I$. 
  Then there is a graph of groups with vertex groups $G(v)=G$, $G(v_i)=H_i$; edge groups $G(e_i)=C_i$;   
   and  edge maps the inclusion of $C_i$ into $G$,  and $\phi_i$ from $C_i$ into $H_i$, for all $i\in I$.
  The fundamental group of this graph of groups is called a \emph{tree extension of centralisers} and is denoted
  $G(\mC, \Hc, \Phi)$, where  $\Hc=\{H_i\,:\,i\in I\}$ and  $\Phi=\{\phi_i\,:\, i\in I\}$. 
\end{defn}

\begin{prop}\label{prop:treeext}
  Let $G$ be in $\cC(X,W)$  and  let $\mC=\{C_i\,:\,i\in I\}$ be a set of abelian centralisers of $G$   satisfying the following conditions. 
\begin{enumerate}[label=(\roman*)]
\item\label{it:ccon1} For all $i\in I$, there is $g_i\in W$ such that $C_i=C_G(g_i)$;
  and  if $g\in W$, 
   $C_G(g)=C_i$ and $i\neq j$ then $g\notin C_j$; 
  and 
\item\label{it:ccon2}  no two centralisers of $\mC$ are conjugate.
\end{enumerate}
For each $i$ in $I$, let $H_i$ be a free abelian group  
and $\phi_i:C_i\maps H_i$ be
a monomorphism from $C_i$ to  $H_i$, such that 
$H_i=\phi(C_i)\times K_i$, for some subgroup $K_i$ of $H_i$.
Then, in the notation of Definition \ref{def:tec}, the group  $G(\mC, \Hc, \Phi)$ is in $\cC$.
Moreover,  $G(\mC, \Hc, \Phi)$ is discriminated  by $G$. 
\end{prop}
\begin{proof} Choose a well-ordering of the set $I$, 
  let $G_0=G$ (assuming $0\notin I$ and is taken to be less than the least element of $I$) and 
  recursively define groups $G_i\in \cC$ and monomorphisms  $\lambda^j_i$, for all $i\in I$ and $j<i$, as follows.
  For a fixed successor ordinal $i\in I$ and all $j\le i$  assume that groups $G_j$ in $\cC(X_j,W_j)$ 
   have been defined, and for $j\le k\le i$,   monomorphisms $\lambda^j_{k}:G_j\maps G_{k}$
  have been defined, such that $\{\lambda_k^j:j<k<i\}$ forms a direct system of monomorphisms (that is 
  $\lambda^k_l\circ \lambda_k^j=\lambda^j_l$, with $\lambda^j_j$ the identity).
  Assume further that, for all $j\le i$ and $m> j$, the centraliser $C_{G_j}(\lambda^0_j(g_m))=\lambda^0_j(C_m)$; and that  conditions \ref{it:ccon1} and \ref{it:ccon2} hold for $\mC_j=\{\lambda^0_j(C_m)\,:\,m>j\}$, with $\cC(X_j,W_j)$   in place  of $\cC(X,W)$. Finally we assume that for $g\in  W$ with $C_G(g)=C_m$, we have $\lambda^0_j(g)\in W_j$, for all $j\le i$ and $j<m$.  By abuse of notation we let $\phi_m$ denote the map from $\lambda^0_j(C_m)$ to $H_m$, given by mapping  $\lambda^0_j(c)$ to $\phi_m(c)$, for all $c\in \lambda^0_j(C_m)$, where $j\le i$ and $j<m$.

  Now define $G_{i+1}=G_i*_{\phi_{i+1}}H_i$, noting that $\lambda^0_i(C_{i+1})\le \lambda^0_i(G_0)\le G_i$ and that  $\lambda^0_i(C_{i+1})=C_{G_i}(\lambda^0_i(g_{i+1}))$.  Also, from the definitions, $H_{i+1}=\phi_{i+1}(\lambda_{i}^0(C_{i+1}))\times K_{i+1}$. Define $\lambda^i_{i+1}$ to be the canonical embedding of $G_i$ in $G_{i+1}$.    From Theorem \ref{thm:lift}, $G_{i+1}$ is in $\cC(X_{i+1},W_{i+1})$, where  
  $X_{i+1}=\lambda_{i+1}^i(X_i)\cup X_{i+1}'$ and $W_{i+1}=\lambda_{i+1}^i(W_i)\cup K_{i+1}\cup W_{i+1}'$, with $K_{i+1}$ as in Definition \ref{def:tec} and $X_{i+1}'$ and  $W_{i+1}'$  defined as  $X_A$ and $W_*$, respectively, in the  statement of Theorem \ref{thm:lift}. 
  For $j<i$ define $\lambda^j_{i+1}=\lambda^j_i\circ \lambda^i_{i+1}$. For $m>i+1$, as $\lambda^0_i(g_m)\in G_i\backslash  \lambda^0_i(C_i)$ it follows, as in the proof of Theorem \ref{thm:lift}, that 
  \[
  C_{G_{i+1}}(\lambda^0_i(g_m))=    \lambda^i_{i+1}(C_{G_i}(\lambda^0_i(g_m)))=\lambda^i_{i+1}(\lambda^0_i(C_m))=\lambda^0_{i+1}(C_m).
  \]
  Also, $\lambda^0_{i+1}(g_m)=\lambda^i_{i+1}\circ \lambda^0_i(g_m)\in \lambda^i_{i+1}(W_i)\subseteq W_{i+1}$.
  Therefore all assumptions on $i$ above, hold for $i+1$.

  If $l\in I$ is a limit ordinal,  assume that groups $G_j$ in $\cC(X_j,W_j)$ and monomorphisms  $\lambda^j_{k}$ have been defined,  and  satisfy the properties above, for all $j\le k< l$. Let $L=\{j\in I \,:\, j<l\}$, let
  $\bar G_l=\varinjlim G_j$, where the limit is over the direct system $L$, and let $\psi_j$ be the canonical map from  $G_j$ to $\bar G_l$, for all $j\in L$. As before, we may assume that $\phi_l$ maps $\psi_0(c)$ to $\phi_l(c)\in H_l$, for  all $c\in C_l$, and define $G_l=\bar G_l*_{\phi_l}H_l$. We must check that $G_l$ is in $\cC$.  As $G_l$ is formed  from $\bar G_l$ by extension of centralisers, in the  light of Theorem \ref{thm:lift}, it suffices to show that $\bar G_l$ is in $\cC$.
  In fact, from Proposition \ref{prop:dirlim}, $\bar G_l=G(\bar X_l,\bar W_l)$ where $\bar X_l$ and $\bar W_l$ are the subsets
  $\bar X_l=\bigcup_{j<l}\psi_j(X_j)$ and $\bar W_l=\bigcup_{j<l}\psi_j(W_j)$ of $\bar G_l$.

Set
  $X_l=\bar X_l\cup X'_l$ and $W_l=\bar W_l\cup K_l \cup W'_l$; both subsets of $G'_l* H_l$.
  From Theorem \ref{thm:lift}, $G_l$ is in $\cC(X_l,W_l)$.   As in the case of successor ordinals above, the assumptions made for $j<l$ are now seen to hold also for $l$.  Therefore, by induction the direct limit $\varinjlim \{G_i\,:\,i\in I\}= G(\mC, \Hc, \Phi)$ is  in $\cC$.

For the final statement,
assume that $\{\pd_d:d\in D\}$ is a discriminating family of homomorphisms
for $G_i$ by $G_0$. From Theorem \ref{thm:lift},
   $G_{i+1}$ is discriminated by $G_i$ via a family 
   $\{\lambda_{i,m}\,:\, (i,m)\in I\times \BN\}$. 
   Then $G_{i+1}$ is discriminated by $G_0$ via  the family
  $\{\pd_d\circ \lambda_{i,m}\,:\, d\in D, (i,m)\in I\times\BN\}$. 

  Now suppose $l$ is a limit ordinal and, for all $j<l$, the family $\{\pd_{j,d}\,:\, d\in D_j\}$ is discriminating for $G_j$ by $G_0$. 
  Any finite set of elements of the direct limit $\bar G$ of  the $G_j$, may be represented by elements of $G_j$, for some fixed $j<l$.
  It follows that $\bar G$ is discriminated by $G_0$ via the family $\bigcup_{j<l}\{\pd_{j,d}\,:\, d\in D_j\}$. Now it follows  from the first part of the proof that $G_l$ is discriminated by $G_0$.
  Therefore $G(\mC, \Hc, \Phi)$ is discriminated by $G=G_0$, as required.
\end{proof}
\begin{rem}
From the proof, given a well-ordering of $I$, there is a direct system $\{G_i\}_{i\in I}$ of groups such that  
$G_i\in \cC$, for all $i\in I$, and $ G(\mC, \Hc, \Phi)=\varinjlim \{G_i\,:\,i\in I\}$. 
Consequently, the properties of tree extensions of centralisers are similar to  those of ordinary extensions of centralisers. That is, there exist canonical embeddings of $G$ and $H_i$ into $G(\mC, H, \Phi)$; the centraliser of $g_i$ in $G(\mC, H, \Phi)$ contains the group $H_i$ and the group $G(\mC, H, \Phi)$ has the corresponding universal property.
Moreover, groups $G(\mC, H, \Phi)$ for different well-orderings of $I$ have the same
universal property, so are isomorphic.
\end{rem}
\section{Exponential groups}
Following \cite{MR94,MR95}, where further detail may be found,  we define exponential groups as follows.
\begin{defn}\label{defn:A-group}
  Let $A$ be an arbitrary associative ring with identity and let $G$ be a group. Fix a map from $G\times A$ to $G$ and write
  the image of  $(g,\alpha)$ as $g^{\alpha}$. Consider the following axioms:

\begin{itemize}
\item [1.] $g^1 = g$, $g^0 = 1$, $1^{\alpha} = 1$ ;
\item [2.] $g^{\alpha + \beta} = g^{\alpha} g^{\beta}$; $g^{\alpha \beta} = {(g^{\alpha})}^{\beta}$;
\item [3.] ${(h^{-1}gh)}^{\alpha} = h^{-1}g^{\alpha}h$;
\item [4.] $[g, h] = 1 \Rightarrow {(gh)}^{\alpha} = g^{\alpha}h^{\alpha}$.
\end{itemize}

Groups that admit a map $G\times A\maps G$, satisfying the axioms $1. - 4.$ are called $A$-groups.
\end{defn}
The class of  $A$-groups over arbitrary associative rings is referred to as the class of \emph{exponential} groups.
Every group is a $\BZ$-group, and an abelian group which is an $A$-group is, by definition, an $A$-module.

Given $A$-groups $G$ and $H$, a homomorphism $f:G\maps H$ is called an $A$\emph{-homomorphism} if $f(g^\alpha)=(f(g))^\alpha$, for all $g\in G$ and $\alpha\in A$.
A subgroup $H\le G$ is called an $A$\emph{-subgroup} if $h^\alpha \in H$, for all $h \in H$ and $\alpha \in A$. It follows that an intersection of $A$-subgroups is an $A$-subgroup. The $A$-subgroup
$A$\emph{-generated} by a subset $X$ of $G$, written $\langle X\rangle_A$, is defined as the smallest $A$-subgroup of $G$ containing $X$.

A basic operation in the class of exponential groups is that of $A$-completion. Here we give a particular case of this construction (see \cite{MR94} for the general definition). Later on we always assume that the ring $A$ and its subring $A_0$ have a common identity element. 
\begin{defn}\label{def:comp}
 Let $A$ be a ring, $A_0$ a subring of $A$ and $G$ an $A_0$-group. Then an $A$-group $G^A$ is called an $(A_0,A)$\emph{-completion}  of the group $G$ if $G^A$ satisfies the following universal property.
\begin{itemize}
\item [1.] There exists an $A_0$-homomorphism $\tau : G \to G^A$ such that $\tau(G)$ $A$-generates $G^A$: that is  $\langle \tau(G) \rangle_A = G^A$; and 
\item [2.] for any $A$-group $H$ and an $A_0$-homomorphism $\varphi: G \to H$ there exists a unique $A$-homomorphism $\psi: G^A \to H$ such that $\psi \circ \tau= \varphi$.
\end{itemize}
\end{defn}
As every group is a $\BZ$-group,  one  can consider
$(\BZ,A)$-completions of arbitrary groups for each ring of characteristic zero, i.e. $\BZ \leq A$.
In practice, our use of $(A_0,A)$-completions will be restricted to the case where $A_0=\BZ$, in which case we omit $A_0$ from the notation and refer to the $(\BZ,A)$-completion simply as the $A$-completion of $G$.

If $G$ is an abelian $A_0$-group, then the group $G^A$ is also abelian, i.e. it is an $A$-module. In this case $G^A$ satisfies the same universal property as the tensor product $G \otimes_{A_0} A$ of the $A_0$-module $G$ and the ring $A$. Therefore $G^A \simeq G \otimes_{A_0} A$. In particular, if $M$ is an $A$-module then $M^A\cong M$, as $A$-modules. 

The  $A$-completion of a coherent RAAG $G$ will be constructed by defining an $A$-action on successively larger subsets of $G$, and necessarily involves groups in which $A$ acts on some, but not all elements. This brings us to the following definition.
\begin{defn}
  A group $G$ is called a \emph{partial} $A$\emph{-group}, for an associative ring $A$ (with identity $1$) if there exists a subset $P$ of $G\times A$,  such that $g^\alpha$ is defined whenever  $(g,\alpha)\in P$, and   axioms 1. to 4. in Definition \ref{defn:A-group} hold whenever the arguments belong to $P$.
  
  In this case we say the partial $A$-action is \emph{defined on} $P$. 
  Let $H$, $G$ be partial $A$-groups with $A$-actions defined on subsets $P_H$ and $P_G$, respectively.  A homomorphism of groups $\phi:G\to H$ is called a partial $A$-homomorphism if $(g^a)^\phi$ is in $P_H$ and $(g^a)^\phi= (g^\phi)^a$ for all pairs $(g, a)\in P_G$. 
  We say that $X$ is a \emph{partial} $A$\emph{-generating} set for the partial $A$-group $G$ if $G$  is generated (as a group)
  by $\{x^a\,\mid\, (x,a)\in P_G\}$. 
\end{defn}
When  the sets on which $A$-actions are  defined are clear from the context, no explicit reference to the sets $P_G$ and $P_H$
will be made.   In particular,
 when $H$ is an $A$-group (the $A$-action is everywhere defined) it is always assumed that a partial $A$-homomorphism from $H$ to $G$ is defined with respect to $P_H=H\times A$.

\begin{defn}\label{def:compart}
  Let $G$ be a partial $A$-group with action defined on $P\subset G\times A$.
  We say that an $A$-group $G^A$ is an $A$-{\em completion} of  $G$, with respect to $P$,
  if $G^A$ satisfies the following universal property:
	\begin{enumerate}
        \item there exists a partial $A$-homomorphism $\tau: G \to G^A$ such that $\tau(G)$ is an
          $A$-generating set for  $G^A$; and 
		\item given an $A$-group $H$ and a partial $A$-homomorphism $\varphi: G \to H$ there exists a
                  unique $A$-homomorphism $\psi: G^A \to H$ such that $\tau \circ \psi =\varphi$.
	\end{enumerate}
	
      \end{defn}    
      In particular, if $A_0$ is a subring of $A$, then any $A_0$-group is also a partial $A$-group and an $(A_0,A)$-completion is a partial $A$-completion of $G$, with respect to $(G,A_0)$.	
   
 \begin{thm}[\cite{MR94}]
   Let $G$ be a partial  $A$-group, with action defined on the set $P$.
   Then there exists an $A$-completion $G^A$ of $G$, with respect to $P$, 
  and it is unique up to $A$-isomorphism.
\end{thm}
(The version of this theorem proved in \cite{MR94} is restricted to the case of an $(A_0,A)$-completion of $G$ while the proof of general case, which   follows through almost word for word, is left to the reader.)

It is shown in \cite{Amaglobeli92} that the operation of  $A$-completion commutes with taking direct sums and forming direct limits of directed systems of partial $A$-groups. 
\begin{defn}
  A partial $A$-group $G$ is \emph{faithful} (over $A$) if the partial $A$-homomorphism $\tau:G\maps G^A$ is injective. 
\end{defn}
From the definition, it follows that a partial $A$-group $G$ is
faithful over $A$ if and only if there is an injective partial $A$-homomorphism from $G$ into an $A$-group.

 Given a partial $A$-group $G$ and a subgroup $H$ of $G$, we say that $H$ is a \emph{full} $A$-subgroup when $H$ is an $A$-subgroup: that is  $h^\alpha$ is defined and belongs to $H$, for all $h\in H$
and $\alpha \in A$.
 If $H$ is a full $A$-subgroup of a partial $A$-group $G$,
then $H$ is necessarily faithful and $H^A=\tau(H)$.

\section{$A$-completion of groups in $\mathcal C$}

\subsection{$A$-completion of abelian groups in $\mathcal C$}

In general, the construction of the $A$-completion of an abelian
partial $A$-group $M$ is very similar to tensor multiplication by the ring $A$. 
(As we shall see it is only necessary to make the existing partial action of $A$ on $M$ agree with the right action of $A$ on $M \otimes_\BZ A$.)

Let $M$ be a partial right $A$-module (i.e. a partial abelian $A$-group with action defined on $P\subset G\times A$).
Consider the tensor product $M \otimes_{\BZ} A$ of the abelian group $M$ by the ring $A$ over $\BZ$.
The  $A$-completion $M^A$ of  $M$ may be obtained by factorising $M \otimes_{\BZ} A$ by the right $A$-submodule generated by the set of all  elements  
$(x^\alpha\otimes 1) - (x\otimes \alpha)$, for $(x,\alpha)\in P$, 
and defining $\tau(x)$ to be the image of $x\otimes 1$ in the quotient. 
The following proposition is a copy of the corresponding result from the theory of modules.
 
\begin{prop}\label{prop:abcomp} 
  Let $M$ be a torsion-free abelian partial $A$-group, with action defined on the set $P$,
  where $A$ is a unitary associative ring with a torsion-free additive group. Assume that for
  all $x\in M$, $\alpha \in A$ and non-zero $n\in \BZ$,
  \begin{equation}\label{eq:actwell} \textrm{if }  (x^n,\alpha)\in P \textrm{ then }(x,\alpha)\in P.
  \end{equation}
  Then $M$ is  $A$-faithful, $M^A$ is a torsion-free abelian $A$-group  and $\tau(M)$ is a direct summand of $M^A$.
\end{prop}
\begin{proof}
  As a torsion-free abelian group is a direct limit of finitely generated subgroups,
  we may restrict attention to finitely generated  free abelian groups, and so to infinite cyclic groups. If $M=\la x\ra$ is infinite cyclic then the condition on $P$ implies  that either $P=\la x\ra\times A$ or $P=\emptyset$. In the first case $M=M^A$ and in the second case $M^A=M\otimes_\BZ A$.
\end{proof}
\subsection{$A$-completion of non-abelian Groups in $\mathcal C$}\label{sec:nabtensor}

From now on all rings are associative, with free  abelian additive subgroup and a multiplicative identity $1$.
For such a  ring $A$, by $\BZ\subseteq A$ we mean the characteristic subring of $A$, and we always assume that
a basis contains the element $1$; so $A=\BZ\times A'$, for some subring $A'$.

Let $A$ be a ring and let $G$ be a group satisfying condition \ref{cond:R}, namely:
\begin{enumerate}[label=\textbf{(R)},ref=(R)] 
\item\label{cond:R}
  $G \in \mathcal C$ is a partial $A$-group, such that for all $g\in G$,
  \begin{enumerate}[label=(\roman*),ref=(\roman*)]
  \item\label{it:R1} if $C_G(g)$  is non abelian, and  $C(g)= \Z(g) \times O(g)$, then the centre $Z(C(g))$ of $C(g)$ is a full  $A$-subgroup (so $A$-faithful) and
  \item\label{it:R2} if $C_G(g)$  is abelian then condition \eqref{eq:actwell} holds for all $x\in C_G(g)$, $\alpha \in A$ and $n\in
    \BZ$ ($n\neq 0)$.
    \end{enumerate}
\end{enumerate}
In this section we describe the $A$-completion of such a group $G$. 
Assume then that $G$ satisfies condition \ref{cond:R}. Zorn's Lemma guarantees the existence of a set
$\mathfrak C = \{ C_i\,:\, i \in I \}$, of centralisers $C_i$ of elements of $G$, which satisfies the following conditions $(S)$.
\begin{enumerate}[label=\textbf{(S\arabic*)}]
\item Any $C \in \mathfrak C$ is abelian but not a full $A$-subgroup. 
\item No two centralisers from $\mathfrak C$ are conjugate.
  \begin{itemize}
  \item Note that if $C\neq C' \in \mathfrak C$ then $C \cap C'^g$ is an $A$-module,
    for all $g\in G$. Indeed, assume that $C=C(a)$, $C'^g=C(b)$, and $x\in C(a) \cap C(b)$. If $[a,b]=1$, that is $a\in C(b)$ and $b \in C(a)$, then since $C(a)$ and $C(b)$ are abelian, we have that $C(a)< C(b)$ and $C(b) < C(a)$ and so $C(a) = C(b)$. If $[a,b]\ne 1$, then $C(x)$ is non-abelian and so by assumption 
    $Z(C(x))$ is an $A$-subgroup. In particular $x^\alpha$ is defined and in 
    $Z(C(x))$ and   $Z(C(x))\le C(a)\cap C(b)$, for all $\alpha \in A$.
    \end{itemize}
\item Any abelian centraliser in $G$ which is not a full $A$-subgroup is conjugate to a centraliser in $\mathfrak C$.
\end{enumerate}

Given $C_i\in \mathfrak C$, since by assumption $C_i$ is abelian, from condition \ref{cond:R} \ref{it:R2}   and  Propostion \ref{prop:abcomp}, we have $C_i^A=C_i\otimes_\BZ A$, for all $i\in I$, and we can form the set $\mathcal H_A = \{ C_i \otimes_\BZ A \mid i \in I \}$ and the set of canonical embeddings $\varPhi_A = \{ \varphi_i : C_i \hookrightarrow  C_i \otimes_\BZ A \mid i \in I \}$. As $C_i$ is a direct summand of $C_i\otimes_\BZ A$, we may consider the tree extension of centralisers  $G^*$, of this special type: 
\[
G^* = G(\mathfrak C, \mathcal H_ A,\varPhi_A).
\]
If $G^*$ satisfies condition \ref{cond:R} then we can iterate this construction up to level $\omega$: 
\begin{equation}\label{eq:extcen}
G=G^{(0)} <G^{(1)} <G^{(2)} <\dots < G^{(n)} < \dots
\end{equation}
where $G^{(n+1)} = G^{(n)}( \mathfrak C_n,\mathcal H_ {A,n},\varPhi_{A,n})$, and the set $\mathfrak C_n$ of centralisers in the group $G^{(n)}$
satisfies the condition (S). 

\begin{defn}\label{defn:ice}
  The union $\bigcup\limits_{n\in \omega} G^{(n)}$ of the chain (\ref{eq:extcen}) is called an 
  \emph{iterated centraliser extension} (ICE for short)  of $G$ by the ring $A$.
(If $A=\BZ[t]$ we refer simply to an \emph{ICE} of $G$.)
\end{defn}

\begin{thm}[{\emph{cf.} {\cite[Theorem 8]{MR95}}}]\label{thm:tensor} 
  Let $A$ be a ring (as at the beginning of this sub-section) 
  and let $G$ be a non-abelian partial $A$-group satisfying condition \ref{cond:R}. If all abelian centralisers of $G$ are faithful over $A$,  then the $A$-completion $G^A$ of $G$ by $A$ is an ICE of $G$ by $A$.
  Furthermore $G^A$ is in $\cC$ and is discriminated by $G$.
\end{thm}

To prove this theorem, we shall show that the tree extension of centralisers $G^* = G(\mathfrak C, A)$, defined above,  satisfies condition \ref{cond:R} and has an appropriate universal property. This will allow us to conclude the same for all the groups $G^{(n)}$ from the chain (\ref{eq:extcen}) and then in turn to prove the theorem. 

\begin{lem} \label{lem:Aactiong} Assume that $G$ satisfies condition \ref{cond:R}, that $\mC=\{C_i:i\in I\}$ is a set of centralisers  of elements of $G$ which satisfies (S) and that $G^*=G(\mathfrak C, \mathcal H_ A,\varPhi_A)$.
  For any $g \in G$,
  \begin{enumerate}[label=(\roman*)]
  \item if the centraliser of $g\in G$ is abelian, then  either
    \begin{enumerate}[label=(\alph*)]
    \item $C_G(g)$ is not conjugate to an element of $\mC$, in which case $C_{G^*}(g)\cong C_G(g)$, or
    \item $C_G(g)$ is conjugate to an element of $\mC$, in which case $C_{G^*}(g)\cong C_G(g) \otimes A= C_G(g)^A$. 
    \end{enumerate}
   In both cases $C_{G^*}(g)$ is an $A$-module.  Otherwise,
 \item the centraliser of $g$ is non-abelian and $C_{G^*}(g)= \Z_G(g) \times O_G(g)$, $\Z(g)=\Z_G(g)$, $g\in \Z(g)\le Z(C(g))$ and $Z(C(g))$ is a  full $A$-subgroup.
\end{enumerate}
  \end{lem}
\begin{proof}
As in the remark following the proof of Proposition \ref{prop:treeext}, there is a directed system $\{G_i\}_{i\in I}$ of groups, with $G=G_0$,  such that
  $G_i\in \cC$, for all $i\in I$, and $ G(\mC, \Hc_A, \Phi_A)=\varinjlim \{G_i\,:\,i\in I\}$. 
If the centraliser $C_{G_0}(g)$ is not conjugate to an element of $\mC$ assume that, for some $i\ge 0$, $C_{G_i}(g)\cong C_{G_0}(g)$.
Then, as in the proof of  Proposition \ref{prop:treeext},   $C_{G_{i+1}}(g)\cong C_{G_0}(g)$. 
If $l$ is a limit ordinal, then the standard argument shows that, if  $C_{G_j}(g)\cong C_{G_0}(g)$, for all $j<l$, then again $C_{G_{l}}(g)\cong C_{G_0}(g)$. Hence, if $C_G(g)$ is not conjugate to an element of $\mC$ then, for all $i\in I$, $C_{G_i}(g)\cong C_{G_0}(g)$, so $C_{G^*}(g)\cong C_{G}(g)$.
 In this case if $C_G(g)$ is abelian, so by definition of  $\mC$ a full $A$-subgroup, then this isomorphism
 induces an $A$-action on $C_{G^*}(g)$, making it into an $A$-subgroup. If $C_G(g)$ is non-abelian then $C_G(g)=\Z_G(g)\times O_G(g)$  and as 
 $Z(C_G(g))$ is a full $A$-subgroup, so is its isomorphic image $Z(C_{G^*}(g))$. 

 On the other hand suppose  $C_G(g)$ is  conjugate to an element $C_i$  of $\mC$. Then, $C_{G_i}(g)$ is conjugate to $C_i\otimes A$ and as no two elements of $\mC$ are conjugate we may unambiguously extend the action of $A$ on $C_i\otimes A$ to $C_{G_i}(g)$. 
 As above, for all $m>i$, we have $C_{G_m}(g)$ isomorphic to $G_{G_i}(g)$ under the map $\lambda^i_m$, which from the definitions is an $A$-homomorphism. Hence $C_{G^*}(g)\cong G_{G_i}(g)$  which is an abelian $A$-subgroup.
\end{proof}

\begin{lem}\label{lem:Gfull}
$G^*$ is a partial $A$-group and $G$ is a full $A$-subgroup of $G^*$. Moreover, $G^*$ is $A$-generated by $G$. 
\end{lem}
\begin{proof}
  By construction $G^*$ is a partial $A$-group. Let $g\in G\le G^*$. If $C_G(g)$ is abelian then, from Lemma \ref{lem:Aactiong},   $C(g)$ is a full $A$-subgroup. If  $C_G(g)$ is non-abelian then 
  $g\in Z(C(g))$, which is also a full $A$-subgroup. In both  cases the action of $A$ on $g$ is defined.
  As $C_i^A$ is $A$-generated by $C_i$, for all $i\in I$,  it follows that $G^*$ is $A$-generated by $G$.
\end{proof}

\begin{lem}\label{lem:G*R}
  $G^* $  satisfies condition \ref{cond:R}. In particular $G^*$ is in $\cC$. 
\end{lem}

\begin{proof}
  That $G^*$ is in $\cC$ follows from Proposition \ref{prop:treeext}. To see that the second condition  of Proposition \ref{prop:treeext} \ref{it:ccon1} holds note that if $w\in W$ (where $G$ is in $\cC(X,W)$) and $C_G(w)$ is abelian  then $C(w)$ is not conjugate to $C(w')$, for all other $w'\in W$. Hence if $w'\in W$  and $w'\in C_G(w)$, then $w=w'$. 
  That $G^*$ is a partial $A$-group satisfying the  appropriate conditions on centralisers follows from Lemmas \ref{lem:Aactiong} and \ref{lem:Gfull}. 
\end{proof}

\begin{lem}\label{lem:lem7}
  The group $G^*$ has the following universal property with respect to the canonical embedding $\tau: G \hookrightarrow G^*$.
  For any $A$-group $H$ and any partial $A$-homomorphism $f : G \longrightarrow H$ there exists a partial $A$-homomorphism  $f^* : G^* \longrightarrow H$ such that $f = f^* \circ \tau$.
\end{lem}

\begin{proof}
Write $f=f_0$ and $G=G_0$ and,  in the notation of the proof  of Proposition \ref{prop:treeext}, let $0\le i\in I$ and assume that for  all $0\le j\le i$, we have $A$-homomorphisms  $f_j:G_j\maps H$ such that,  $f_0=f_j\circ \lambda^0_j$, where $f=f_0$. Assume in addition  that, for all $j\le k\le i$ we have $f_j=f_k\circ \lambda^j_k$.
Then $f_i(C_{i+1})$ is an abelian $A$-subgroup $N_{i+1}$ of $H$. Using the universal property of $A$-completion for abelian groups,  we can find a homomorphism $\Psi_{i+1} : C_{i+1}\otimes A \longrightarrow N_{i+1}$ such that $f_{i}=\Psi_{i+1} \circ \tau_{i+1}$,  where $\tau_{i+1} : C_{i+1}\hookrightarrow C_{i+1}\otimes A$ is the canonical embedding.
By \cite[Proposition 5]{MR95},  there exists an $A$-homomorphism $f_{i+1} : G_{i+1} \longrightarrow H$ with the property $f_i = f_{i+1} \circ \lambda^i_{i+1}$.
It then follows directly from the definitions that  $f_j=f_k\circ \lambda^j_k$, for all $0\le j \le k\le i+1$.

If $l$ is a limit ordinal, the universal property of direct limits gives a homomorphism $\bar f_l:\bar G_l\maps H$ such that  $f_j=\bar f_l\circ \lambda^j_l$. Then using $\bar G_l$ and $\bar f_l$, in place of $G_i$ and $f_i$, and  $G_l$ in place of $G_{i+1}$ the   previous argument gives an $A$-homomorphism $f_l:G_l\maps H$, making the necessary diagrams commute. The result now follows by induction.
 \end{proof}

 \begin{lem}\label{lem:disc}
$G^*$ is discriminated by $G$.
\end{lem}

\begin{proof}
  This follows from Proposition  \ref{prop:treeext}.
\end{proof}

\begin{proof}[Proof of Theorem \ref{thm:tensor}]
  Let $G^{(0)}=G$ and $G^{(n+1)} = (G^{(n)})^*$, as in the chain \eqref{eq:extcen}, and let  $\bar G=\bigcup \limits_{n \in \omega} G^{(n)}$.
  The claim is that the ICE $\bar G$ coincides with the  $A$-completion  $G^A$ of $G$. Indeed,  as a union of partial $A$-groups, $\bar G$  is a partial $A$-group, but an action of $A$ on $\bar G$ is in fact defined everywhere: if $x \in \bar G$,
  then $x \in G^{(n)}$, and hence by Lemma \ref{lem:Aactiong} the action of $A$ on $x$ is defined in $G^{(n+1)}$.
  So $\bar G$ is an $A$-group. Similarly, from Lemma \ref{lem:Gfull}, $\bar G$ is $A$-generated by $G$.  

  To prove that $\bar G$ is the $A$-completion of the partial $A$-group $G$,
  it remains to verify the corresponding universal property.
  Let $H$ be an $A$-group and $f : G \longrightarrow H$ a partial $A$-homomorphism. Using Lemma \ref{lem:lem7},
  we can extend $f$ to a partial $A$-homomorphism $f_n : G^{(n)} \longrightarrow H$. 
  By construction these homomorphisms commute with the canonical maps from $G^{(n)}$ into $G^{(m)}$, for $m\ge n$:
  that  is $f_n=f_m|_{G^{(n)}}$. Since $\bar G$ is the direct limit of the $G^{(n)}$  (with these inclusion maps) there exists a unique homomorphism $\bar f: \bar G\maps H$ such that  $f_n = \bar f \circ \psi_n$, where $\psi_n$ is the canonical embedding of $G^{(n)}$ into $\bar G$. In  particular $f=f_0 = \bar f \circ \psi_0$. As all the partial $A$-homomorphisms $f_n$
  restrict to $A$-homomorphisms on $G^{(n-1)}$, it follows that $\bar f$ is  an $A$-homomorphism, so $\bar G=G^A$, as claimed.

  Lemma \ref{lem:G*R} implies that $G^{(n)}$ is in $\cC$, for all $n$. A given centraliser is extended at most once
  in the construction of chain \eqref{eq:extcen}, so condition \ref{it:dirlim2} of Proposition \ref{prop:dirlim} holds, whence  $G^A$ is in $\cC$. 
   That $G^A$ is discriminated by $G$ follows using Lemma \ref{lem:disc}, to see that $G^{(n)}$ is discriminated by $G$, for all $n>0$,  and then by an argument similar to the last part of the proof of Lemma \ref{lem:disc} to see that $\bar G=G^A$ is discriminated by  $G$. 
\end{proof}

\subsection*{Applications}

Toral relatively hyperbolic groups belong to the class $\cC$ and trivially satisfy condition \ref{cond:R} (as partial $A$-groups
with action defined on $G\times\BZ$). Therefore, our results recover the results for torsion-free hyperbolic groups (as long as $A$ is as at the beginning of Section \ref{sec:nabtensor}) see \cite{BMR02}, and for toral relatively hyperbolic groups, see \cite{KM12}.

\begin{cor}[\cite{KM12}]
  Let $A$ be 
  a ring (as at the beginning of Section \ref{sec:nabtensor})
  and let $G$ be a torsion-free toral relatively hyperbolic group. Then the $A$-completion $G^A$ of $G$ by $A$ is an ICE of $G$ by $A$. Furthermore, $G^A$ is in $\cC$ and is discriminated by $G$.
\end{cor}

\smallskip

Our results can also be applied to coherent RAAGs. Given a coherent RAAG $\GG$, we can view it as the graph product $\mathcal G(\Delta, \BZ)$, where $\Delta$ is a chordal graph. Let $G = \mathcal G(\Delta, A)$, where $A$ is a ring of the type above.
From Theorem \ref{thm:inC} $G$ is in $\cC$ and it satisfies condition \ref{cond:R}.
 
\begin{lem}
  Let $G = \mathcal G(\Delta, A)$, where $A$ is a ring (with the usual restrictions)
  and $\Delta$ a chordal graph. For all $g\in G$ such that $C(g)$ is non-abelian, we have that  $C(g) = \Z(g) \times O(g)$, $\Z(g)\le Z(C(g))$
  and $Z(C(g))$ is an $A$-module.
\end{lem}
\begin{proof}
  From Corollary \ref{lem:pcinC}, $C(g)$ is conjugate to $C(w)=\Z(w)\times O(w)$, for some cyclically reduced element $w$ of the RAAG $G$. If $x\in Z(C(w))\cap X$ then $x$ belongs  to $X_v$, for some vertex $v$ of $\Gamma$, and basis $X_v$ for  the copy $A_v$ of $A$ associated to $v$. As $\la X_v\ra=A_v\cong \la v\ra\otimes A$ the subgroup  $A_v$ is an $A$-module, and  from Lemma \ref{lem:descentr}, $A_v\le Z(C(w))$, so $A$ acts on $Z(C(w))$, hence on $Z(C(g))$. 
\end{proof}

\begin{thm}
  In the notation above, the $A$-completion $G^A$ of $G$ by the ring $A$ is an ICE of $G$ over the ring $A$.
\end{thm}

\begin{cor} \label{cor:expice}
  Let $\GG$ be a coherent RAAG.
  Then the $A$-completion $\GG^A$ of $\GG$ by the ring $A$ is an ICE of the group $G$ over the ring $A$, belongs to the class $\cC$ and is discriminated by $\GG$. \\
\end{cor}

Formally, the $A$-completion $\GG^A$ of $\GG$ is obtained in two steps:
we first embed $\GG$ into the graph product $G$ to ensure property \ref{cond:R} and then we take an ICE of $G$ over $A$.
Abusing the terminology, we say that $\GG^A$ is an ICE of $\GG$ over $A$.

\section{Coherence of limit groups over coherent RAAGs}\label{sec:tower}

In this section, we introduce graph towers over coherent RAAGs and prove that they are coherent, see Corollary \ref{cor:towers are coherent}. As a main corollary, we obtain that limit groups over coherent RAAGs are finitely presented, see Corollary \ref{cor:limit groups are fp}. The proofs of this section are (almost) independent of the previous sections. More precisely, we only use (and prove) that graph towers belong to the class $\mathcal C$, in order to deduce a hierarchical decomposition of the tower as a graph of groups with abelian vertex groups, see Theorem \ref{thm:abelianedge}. Coherence of graph towers is then deduce from that structural decomposition. 

As we mentioned, we shall define (below) a graph tower  $\Ts$ of height $l$, over a RAAG $\GG=\GG(\G)$, to be a sequence $\Ts_0,\ldots, \Ts_l$ of triples:
$\Ts_l=(G_l,\HH_l,\pi_l)$, where $G_l$ is a group, $\HH_l=\GG(\G_l)$, $\G_l$ a simple graph, is  the $l$-th RAAG  of $\Ts$ and $\pi_l$ is an epimorphism of $\HH_l$ onto $G_l$.
The group $G_l$ is called \emph{the} group of $\Ts_l$ and by abuse of notation is itself referred to as a graph tower over $\GG$.  
The definition depends on a partition of the edges of $\G_l$ into disjoint subsets $E_c(\Gamma_l)$ and $E_d(\Gamma_l)$. Given this partition, we define a \emph{co-irreducible} subgroup $\KK$ of $\HH_l$ to be a canonical parabolic subgroup which is closed  (i.e. $\KK^{\perp\perp}=\KK$) and such that $\KK^\perp$ is $E_d(\Gamma_{l})$-directly indecomposable.  
As in \cite{CRK15}, where more details are given, we make the following definition.  
\begin{defn}\label{def:tower}
  A \emph{graph tower} $\Ts$ of height $l$ is a sequence of triples $\Ts_0,\ldots, \Ts_l$, defined recursively, setting  $\Ts_0=(G_0,\HH_0,\pi_0)$,
  where  $G_0=\HH_0=\GG$, $\G_0=\G$, $\pi_0=\textrm{Id}_{\GG}$ and $E_d(\G)=E(\G)$. In addition define the subset $S_0=\emptyset$  of $\HH_0=\GG(\G_0)$. 
  To define $\Ts_l$ assume that $\Ts_{l-1}=(G_{l-1},\HH_{l-1},\pi_{l-1})$, $\G_{l-1}$, $E_d(\G_{l-1})$, $E_c(\G_{l-1})$ and $S_{l-1}\subseteq \HH_{l-1}=\GG(\G_{l-1})$ are defined, and choose a co-irreducible subgroup $\KK_l$ of $\HH_{l-1}$, canonically generated by a set $Y_l$;  and a positive integer $m_l$. Let $\G_l$ be the graph with vertex set $V(\G_{l-1})\cup \{x_1^{l},\ldots ,x_{m_l}^{l}\}$ (where the $x_i^{l}$ are new symbols). The set of $d$-\emph{edges} of $\G_l$ is $E_d(\G_l)=E_d(\G_{l-1})\cup \{(x_i^l,y)\,:\, 1\le i\le m_l, y\in Y_l\}$. The set  of $c$\emph{-edges} of $\G_l$ is $E_c(\G_l)=E_c(\G_{l-1})\cup E'_c(\G_l)$ and the set $S_l=S_{l-1}\cup S'_l$,  where $E'_c(\G_l)$ and $S'_l$ are  chosen according to one of the alternatives, \emph{Basic floor}, \emph{Abelian floor} or \emph{Quadratic floor} below. In all cases $\HH_l=\GG(\G_l)$, $G_l=\HH_l/\ncl\la S_l\ra$ and $\pi_l$ is the canonical map from $\HH_l$ to $G_l$. 
    \end{defn}

\noindent\textbf{Basic floor.}
\begin{itemize}
  \item $E'_c(\G_l)=\emptyset$ if $\KK_l^\perp$ is a directly indecomposable subgroup of $\HH_{l-1}$; and   
  \item $E'_c(\G_l)=\{(x_i^l,x_j^l)\,:\,1\le i<j\le m_l\}\cup \{(x_i^l,y)\,:\,1\le i\le m_l, y\in Y_l^\perp\}$, if
    $\KK_l^\perp$ is directly decomposable.  
  \end{itemize}
  $S'_l$ is the set of \emph{basic relators}:
  \[S'_l=[C,x_i^l], 1\le i\le m_l,\]
  where $C=\pi_{l-1}^{-1}(C_{G_{l-1}}(\KK_{l}^\perp)$.
  
\noindent\textbf{Abelian floor.}
\begin{itemize}
  \item $E'_c(\G_l)=\{(x_i^l,x_j^l)\,:\,1\le i<j\le m_l\}$, if $\KK_l^\perp$ is a directly indecomposable subgroup of $\HH_{l-1}$; and   
  \item $E'_c(\G_l)=\{(x_i^l,x_j^l)\,:\,1\le i<j\le m_l\}\cup \{(x_i^l,y)\,:\,1\le i\le m_l, y\in Y_l^\perp\}$, if
    $\KK_l^\perp$ is directly decomposable.
  \end{itemize}
  Either
  \begin{itemize}
  \item $S'_l=[C,x_i^l], 1\le i\le m_l$, where $C=\pi_{l-1}^{-1}(C_{G_{l-1}}(u))$ with $u$ is a non-trivial, cyclically reduced, block, root element of $\KK_{l}^\perp$;  or
    \item $S'_l$ is the set of basic relators.
    \end{itemize}
    
\noindent\textbf{Quadratic floor.}
$E'_c(\G_l)$ is defined as in the Basic floor case. \\[1em]
$S'_l$ consists of the set of basic relators and a relator $W$ of the form either 
\begin{itemize}
\item $W=\prod_{i=1}^g[x_{2i-1},x_{2i}]\left(\prod_{i=2g+1}^{m}u_i^{x_i}\right)u_{m+1}$, with $u_{m+1}^{-1}=\prod_{i=1}^g[v_{2i-1},v_{2i}]\prod_{i=2g+1}^{m}u_i^{w_i}$; or
\item $W=\prod_{i=1}^{g}x_i^2\left(\prod_{i=g+1}^mu_i^{x_i}\right)u_{m+1}$, with $u_{m+1}^{-1}=\prod_{i=1}^{g}v_i^2\prod_{i=g+1}^mu_i^{w_i}$,
\end{itemize}
for some  $u_i,v_j,w_k$ in  $\KK^\perp$ such that the following condition $\circledast$ holds.
\begin{itemize}
\item[$\circledast$] The Euler characteristic of $W$ is at most $-2$ or $W=[x_1,x_2]u_3$, and in both cases
  the subgroup of $G_{l-1}$ generated
  by the set of all 
  $u_i,v_j,w_k$ is non-abelian. 
\end{itemize}
  
(In \cite{CRK15} in the definition of Quadratic floor the condition $\circledast$ is slightly less restrictive.
However,  with the simpler definition given here, all the properties of  graph towers in \cite{CRK15} hold as 
before.)

We next recall some of the properties of the graph towers proven in \cite{CRK15}.  
\begin{rem}\label{rem:prop_towers} 
From the definitions it follows that, $\HH_{l-1}$ is a retraction of $\HH_l$ and that
$G_{l-1}$ is a retraction of $H_l/\ncl(S_{l-1})$. In \cite[Theorem 7.1 and Theorem 8.1]{CRK15}, it is shown that, 
given a limit group $G$ over $\GG$ there exists a graph tower $\Ts_l$ over $\GG$,  and embedding of $G$ into $G_l$  (both of which may be effectively constructed if $G$ is given by its presentation as the coordinate group of a system of equations) such that 
\begin{itemize}
\item $G_l$ and $\HH_l$ are discriminated by $\GG$;
\item If $\KK$ is a co-irreducible subgroup of $\HH_l$ then $\KK^\perp$ is either a directly indecomposable
  subgroup of $\HH_l$ or is $E_c(\G_l)$-abelian (\cite[Lemma 6.2]{CRK15});
\item there is a discriminating family $\{\varphi_i\}$ of $G_l$ by $\GG$ such that, writing $\varphi'_i$ for
    $\varphi_i\pi_l$, the following hold, see \cite[Theorem 7.1]{CRK15}.
  
\begin{itemize}
\item for any  edge $e=(x,y)$ of $\G_l$, either $x$ and $y$ are sent to the same cyclic subgroup of  $\GG$, by all homomorphisms $\varphi_i'$, in which case $e$ belongs to $E_c(\Gamma_l)$; or the images of $x$ and $y$, under all homomorphisms $\varphi_i'$, disjointly commute, in which case $e$ belongs to $E_d(\Gamma)$.
\item if $\KK$ is a canonical parabolic subgroup of $\HH_l$ then there exists a subgroup $\GG_\KK$ of $\GG$ such that, $\GG_\KK < \KK$ and $\bigcup_i\varphi_i'(\KK)=\GG_\KK$. In particular, if $\KK$ is a co-irreducible subgroup of $\HH_l$, then $\GG_{\KK}$ is co-irreducible in $\GG$ (here $E_d(\G)=E(\G)$). Furthermore, in this case, $\GG_\KK^\perp=\GG_{\KK^\perp}$ is a directly indecomposable subgroup of $\GG$, $\bigcup_i\varphi_i'(\KK^\perp)=\GG_{\KK^\perp}$ and $C_\GG(\GG_{\KK^\perp})=\GG_\KK$.
 \end{itemize}

  (In the terminology of \cite{CRK15}, the homomorphisms  $\varphi_i$  all factor through the principal branch
  of the (tribal) Makanin-Razborov diagram for $G$.)
\end{itemize}
\end{rem}

\begin{defn}
In the notation above, given a limit group $G$ over a coherent RAAG $\GG$, we call a graph tower $G_l$ into which $G$ embeds a \emph{graph tower associated to} $G$ over $\GG$, and the family of discriminating homomorphisms $\{\varphi_i\}$, as described above, a \emph{principal} discriminating family of homomorphisms.
\end{defn}

We will use the following graph of groups decomposition of graph towers, obtained in \cite{CRK15}.
\begin{lem}[cf. Lemma 5.3 in \cite{CRK15}] \label{lem:prtower}
  A graph tower $\Ts_l=(G_l,\HH_l,\pi_l)$ of height $l$ over $\GG$ admits one of the following decompositions, where $\KK=\KK_l$.
	\begin{itemize}
		\item [a1)] $G_{l-1} \ast_{C_{G_{l-1}}(\KK^\perp)} \left(C_{G_{l-1}}(\KK^\perp)\times\langle x_1^l, \dots, x_{m_l}^l \rangle\right)$
		{\rm(}in this case the floor is basic and $\KK^\perp$ is non-abelian{\rm)};
		
		\item [a2)] $G_{l-1} \ast_{C_{G_{l-1}}(\KK^\perp)} \left(C_{G_{l-1}}(\KK^\perp) \times \langle x_1^l, \dots, x_{m_l}^l \mid [x_i^l,x_j^l]=1, 1\le i,j \le m_l, i\ne j \rangle\right)$ \\ {\rm(}in this case  the floor is basic and  $\KK^\perp$ is abelian{\rm)};
		
		\item [b1)] $G_{l-1} \ast_{C_{G_{l-1}}(u)} \left(C_{G_{l-1}}(u) \times \langle x_1^l, \dots, x_{m_l}^l \mid [x_i^l,x_j^l]=1, 1\le i,j \le m_l, i\ne j \rangle\right)$ \\ {\rm(}in this case the floor is abelian and $u$ is non-trivial irreducible root element{\rm)};
		
		\item [b2)] $G_{l-1} \ast_{C_{G_{l-1}}(\KK^\perp)} \left(C_{G_{l-1}}(\KK^\perp) \times \langle x_1^l, \dots, x_{m_l}^l \mid [x_i^l,x_j^l]=1, 1\le i,j\le m_l, i\ne j \rangle\right)$ \\ {\rm(}in this case  the floor is abelian and $\KK^\perp$ is non-abelian{\rm)};
		
		\item [c)] $G_{l-1} \ast_{C_{G_{l-1}}(\KK^\perp)\times \langle u_{2g+1},\dots, u_m\rangle} \left( \langle u_{2g+1}, \dots, u_m, x_1^l, \dots, x_{m_l}^l \mid W \rangle \times C_{G_{l-1}}(\KK^\perp)\right)$ \\ {\rm(}in this case the floor is quadratic and $\KK^\perp$ is non-abelian{\rm)}.
	\end{itemize}
\end{lem}
Note that, as $\HH_l$ is discriminated by the principal family $\{\varphi_i\}$, the elements $u, u_i\in \HH_l$
from cases $b1)$ and $c)$ are root block elements.

In the remainder of the section, to distinguish between  orthogonal complements of parabolic subgroups in various $\HH_j$, given a subset $Y$ of the canonical generating set $X$ of $\HH_j$ we shall write 
$\lk_j(Y)=\{x\in X\,:\, x\notin Y \textrm{ and } [x,Y]=1\}$ and $\la Y\ra^{\perp_j}= \la \lk_j(Y)\ra\le \HH_j$.  
We shall also write 
$C_j$ for $C_{G_{j}}$, $\mbf{x}^{j}$ for $\{x^{j}_{1},\ldots , x^{j}_{m_{j}}\}$ and, in case c),  $\mbf{u}$ for $\{u_{2g+1},\ldots , u_{m}\}$.

Our next goal is to give a more precise description of a graph tower when the base group is a coherent RAAG.
In particular, we show that the graph tower can be built from the coherent RAAG $\GG$ by first building
floors of type a2) with non-abelian centraliser and of type b2) -
obtaining a graph tower which is a coherent RAAG $\GG'$ (discriminated by $\GG$);
and then building floors over $\GG'$ for which the corresponding amalgamation is
over a free abelian subgroup, see Theorem \ref{thm:abelianedge}.

\begin{lem}\label{lem:abedg} 
Let $\GG$ be a coherent RAAG and let $\Ts=(G_l,\HH_l,\pi_l)$ be a graph tower associated to a limit group $G$ over $\GG$.
Let $\KK$ be a canonical parabolic co-irreducible  subgroup of $\HH_l$. If $\KK^\perp$ is non-abelian, then $C_{G_l}(\KK^\perp)$ is finitely generated torsion-free abelian. 
\end{lem}
\begin{proof}
By Remark \ref{rem:prop_towers}, the graph tower $G_l$ is $\GG$-discriminated by the principal family of homomorphisms  $\{\varphi_i'\}$ and, if $\KK$ is co-irreducible, then $\varphi_i'(\KK) < \GG_\KK$, $\GG_\KK$ is co-irreducible and $\bigcup_i\varphi_i'(\KK^\perp)=\GG_\KK^\perp=\GG_{\KK^\perp}$, so $\KK^\perp$ is discriminated by $\GG_{\KK^\perp}$.  In particular, as $\KK^\perp$ is non-abelian, it follows that $\GG_{\KK^\perp}$ is also non-abelian. 

Then the centraliser $C_{G_l}(\KK^\perp)$ is discriminated by $C_{\GG}(\GG_{\KK^\perp})$. Since $\GG_{\KK^\perp}$ is non-abelian and $\GG$ is coherent, we have that $C_{\GG}(\GG_{\KK^\perp})$ is finitely generated torsion-free abelian and since $C_{G_l}(\KK^\perp)$ is discriminated by $C_{\GG}(\GG_{\KK^\perp})$, $C_{G_l}(\KK^\perp)$ is also torsion-free abelian and by the structure of the tower, it is finitely generated.
\end{proof}

\begin{lem} \label{lem:pctowerC}
  Let $\GG$ be a coherent RAAG, let $L$ be a limit group over $\GG$ and let $G_l$ be a graph tower of height $l$ associated to $L$. Then
  \begin{enumerate}
    \item $G_l$ is a tower of height $k\le l$ over  a coherent RAAG $\GG'$, where all the floors of $G_l$ over $\GG'$, are of type a2) with an abelian  centraliser amalgamated;   or of type b1); or of type c) and
    \item $G_l$ is in $\cC$.
  \end{enumerate}
\end{lem}
\begin{proof}
Let the tower consist of the sequence $\Ts_0,\ldots, \Ts_l$, where $\Ts_i=(G_i,\HH_i,\pi_i)$, for $1\le i\le l$. 
 
From Lemmas \ref{lem:ACprop} and \ref{lem:pcinC}, when $i=0$ the following properties are satisfied.
\begin{itemize}
  \item $G_i$ is in $\cC(X_i,W_i)$ and  
    \item
      if $Y$ is a subset of $X_i$ such that $\la Y\ra$ is  abelian, then there exists $u\in \la Y\ra$ such that 
  \begin{equation}\label{eq:CAP} C_{G_i}(Y)=C_{G_i}(u).\end{equation}
\end{itemize}
Assume that all towers of height at most $l-1$ over a coherent RAAG satisfy \eqref{eq:CAP} and say a floor $G_i\le G_{i+1}$ has type a$'$2) if it has type a2) with abelian centraliser amalgamated.
Note that from \cite[Proof of Theorem 7.1]{CRK15}, floors of type a1) occur only at the lowest levels of the tower. That is, we may assume that $G_i\le G_{i+1}$ is of type a1) for $0\le i\le k$, and that no further floors of type a1) occur in the construction. Then $G_{k+1}$ is a coherent RAAG, and replacing $\GG$ with $G_{k+1}$, we may assume there are no floors of type a1).  

If $G_{l-1}\le G_l$ is of type b1) then, since $u$ is a non-trivial block element of $\KK^\perp$, if $\KK^\perp$ is non-abelian, by Lemma \ref{lem:abedg} we have that $C_{G_{l-1}}(u)$ is abelian.
Notice that if $\KK^\perp$ is abelian, since homomorphisms $\varphi_i$ of the principal   family preserve disjoint commutation (see \cite[Lemma 6.29]{CRK15})  it follows that $\varphi_i(\KK^\perp)$ is cyclic. Hence, it follows that for all $x\in \KK^\perp$ one has  $C_{G_{l-1}}(x)=C_{G_{l-1}}(\KK^\perp)$. Therefore if $\KK^\perp$ is abelian, without loss of generality,  we can assume that $G_{l-1}\le G_l$ is of type a$'$2). Assuming then that $C_{G_{l-1}}(u)$ is abelian, from Theorem \ref{thm:lift}, $G_l$ is in $\cC$.

If $G_{l-1}\le G_l$ is of type a$'$2) with abelian centraliser  $C_{G_{l-1}}(\KK^\perp)$ then \eqref{eq:CAP} implies there is $u\in \KK^\perp$ such that $C_{G_{l-1}}(\KK^\perp)=C_{G_{l-1}}(u)$, so $G_l$ is in $\cC$, using Theorem \ref{thm:lift} again.

If $G_{l-1}\le G_l$ is of type c) then the surface $S$ corresponding to the right hand side of the graph of groups decomposition described in Lemma \ref{lem:prtower} part c)  has $m$ boundary components and genus $g$. We form a surface $S'$ from $S$ by attaching an $m+1$-punctured sphere to $S$, identifying   $m$ of the  punctures to the $m$ boundary components of $S$; so $S'$ has genus $g+m-1$ and one boundary component.
 Now $G_l$ may be obtained from $G_{l-1}$ by attaching the boundary component of $S'$ along a path representing the word $u_{2g+1}\cdots u_m$ in $G_{l-1}$.
  
 More formally, set $t=u_{2g+1}\cdots u_m \in G_{l-1}<G_l$. Observe that 
$$
G_l\simeq G_{l-1}*_{\langle z\rangle \times C_{G_{l-1}}(\KK^\perp)} (\langle x_1^l, \dots, x_{m_l}^l,v\mid V(x_1^l, \dots,x_{m_l}^l)\cdot v\rangle \times C_{\G_{l-1}}(\KK^\perp)),
$$
where $\langle z\rangle \simeq \BZ$, $z\mapsto t$ in $G_{l-1}$, $z\mapsto v$ in the second vertex group
and $V$ is a quadratic word in the $x_i^{l}$.  From the definition we have $C_{G_{l-1}}(t) =C_B(v)=\la v\ra \times C_{\G_{l-1}}(\KK^\perp)$, where $B$ is the right hand vertex group, and since $G_{l-1}\in \cC$, we have $Z_{G_{l-2}}(t)=\la t\ra$ and $O_{G_{l-1}}(t)$. Therefore we may write
\begin{equation}\label{eq:typec}
G_l=G_{l-1}\ast_{C_{G_{l-1}}(t)}(\Sigma \times O_{G_{l-1}}(t)),
\end{equation}
where $\Sigma=\langle x_1^l, \dots, x_{m_l}^l,v\mid V(x_1^l, \dots,x_{m_l}^l)\cdot v\rangle$. 
Therefore, from  Corollary \ref{cor:attSinC}, $G_l$ is in $\cC$. 

If $G_{l-1}\le G_l$ is of type b2), then $G_l=G_{l-1}\ast_{C}(C\times A)$, where $C=C_{G_{l-1}}(\KK^\perp)$, $\KK^\perp$ is non-abelian and $A$ is the free abelian group with basis $\mbf x^l$. From Lemma \ref{lem:abedg}, $C$ is finitely generated torsion-free abelian. Let $Y$ be a basis for $C$. Then, for all $y\in Y$,  $C_{G_{l-1}}(y)$ contains $\KK^\perp$, so is non-abelian and thus canonical, and $y\in K_{G_{l-1}}(y)$, which is contained in the centre of $C_{G_{l-1}}(y)$, so every generator of $K_{G_{l-1}}(y)$ commutes with $\KK^\perp$. It follows that $Y$ is a subset of the generators of $G_{l-1}$. We have $G_{l-2}\le G_{l-1}$ of type a$'$2), b1) or c) whence, from the above, $G_{l-1}=G_{l-2}\ast_D B$, where $D=C_{G_{l-2}}(u)$, for some $u\in G_{l-2}$, and $B=\la D,\mbf x^{l-1}\ra$. If $y\in  \mbf x^{l-1}$ then $C_{G_{l-1}}(y)$ is abelian, in each case, a contradiction. Hence
$y\in G_{l-2}$. Therefore $Y\subseteq G_{l-2}$ and $C=C_{G_{l-1}}(\KK^\perp)=\la Y\ra=C_{G_{l-2}}(\KK^\perp)$.
If $u\in C_{G_{l-2}}(\KK^\perp)$, then $\KK^\perp \le C_{G_{l-2}}(u)$, which is abelian. As $\KK^\perp$ is non-abelian it follows that $C_{G'_{l-1}}(u)=C_{G_{G_{l-2}}}(u)$, from which in turn we obtain $G_l=G'_{l-1}\ast_{D}B$. 
Then $G'_{l-1}$ satisfies the inductive hypotheses and  $G'_{l-1}\le G_l$ is of type a$'$2), b1) or c),
so $G_l$ is in $\cC$.

If $G_{l-1}\le G_l$ is of type a2) with non-abelian centraliser we shall show that we can replace the sequence of extensions $G_{l-2}<G_{l-1}<G_l$ by a sequence $G_{l-2}<G'_{l-1}<G_l$, where $G'_{l-1}<G_l$ is of type a$'$2), b1) or c) and $G'_{l-1}$ is in $\cC$.
To simplify notation write $R=G_{l-2}$, $S=G_{l-1}$ and $T=G_l$, so $R$ and $S$ are in $\cC$, say $R$ is in $\cC(X_R,W_R)$
and $S$ is in $\cC(X_S,W_S)$, $R\le S$ is of   type a$'$2), b1) or c) and $S\le T$ is of type a2) with non-abelian centraliser. Then (from the discussion above) there are elements $r\in R$ and $s\in S$ such that either
\begin{enumerate}[label=Case (\roman*)]
\item $S=R\ast_{C_R(r)}(C_R(r)\times A)$\label{it:b1} or
\item $S=R\ast_{C_R(r)}(O_R(r)\times \Sigma)$\label{it:c},
\end{enumerate}
where $A$ is free abelian with basis $\mbf x_A$ and $\Sigma$ is the appropriate analogue of \eqref{eq:typec}: generated by $\mbf x_A$ and $r$,  with $V$ a quadratic word in the generators $\mbf x_A$.
From Lemma \ref{lem:prtower}, $T=S\ast_{C_S(\KK^\perp)}(C_S(\KK^\perp)\times B)$, for some co-irreducible subgroup $\KK$ of $H_{l-1}$ such that
$\KK^\perp$ is abelian, $C_S(\KK^\perp)$ is non-abelian and $B$ is free abelian with basis $\mbf x_B$. From property \eqref{eq:CAP}, there exists $s\in S$ such that $C_S(s)=C_S(\KK^\perp)$, so 
\[T=S\ast_{C_S(s)}(C_S(s)\times B),\]
and, without loss of generality, we may assume $s\in W_S$.

Next define
\[S'=R\ast_{C_R(s)}(C_R(s)\times B)\]
and  either
\begin{enumerate}[label=(\roman*)]
\item $T'=S'\ast_{C_{S'}(r)}(C_{S'}(r)\times A)$, if $S$ is given by \ref{it:b1} above, and 
\item $T'=S'\ast_{C_{S'}(r)}(O_{S'}(r)\times \Sigma)$, if $S$ is given by \ref{it:c}.
\end{enumerate}
The inductive hypothesis implies that $S'\in \cC$ and satisfies condition \eqref{eq:CAP} and so, 
from Theorem \ref{thm:lift} or Corollary \ref{cor:attSinC}, $T' \in \cC$. 

 As $C_S(s)$ is non-abelian, it follows from Lemma \ref{lem:centdesc} or the proof of Corollary \ref{cor:attSinC}
 that $s\in R$ and either $s\notin C_R(r)$ and $C_S(s)=C_R(s)$ or $s\in C_R(r)$ and
$C_S(s)=\Z_R(s)\times\la O_R(s), B\ra$.
If $s\notin C_R(r)$ then $r\notin C_R(s)$ so $C_{S'}(r)=C_R(r)$ and it follows, in the case when $S$ is given by \ref{it:b1}, that $T=T'$.
In \ref{it:c}, we have $C_R(r)=Z_R(r)\times O_R(r)$, where $Z_R(r)$ is cyclic and non-canonical. If $r\notin C_{R}(s)$ then $C_{S'}(r)=C_R(r)$ and so $C_{S'}(r)=Z_S'(r)\times O_{S'}(r)$ with $O_{S'}(r)=O_R(r)$. Again this implies $T=T'$. 

Assume then that $s\in C_R(r)$ and $C_S(s)=\Z_R(s)\times\la O_R(s), B\ra$.

As $S'\in \cC$, we have $C_{S'}(r)=Z_{S'}(r)\times O_{S'}(r)$. As $r\in C_R(s)$, we have $\mbf x_B\subseteq C_{S'}(r)$. Hence, if $C_{S'}(r)$ is abelian then it is contained in $C_R(s)\times B$, so
$C_{S'}(r)=C_R(r)\times B$. Therefore, from \ref{it:C6},  $r\in Z_{S'}(r)$ and $O_R(r)\subseteq O_{S'}(r)$. As both $O_R(r)$ and $O_{S'}(r)$ are canonical and $X_{S'}=X_R\cup \mbf x_B$, we have $O_{S'}(r)=\la O_R(r),\mbf x_B\ra$. Moreover, from Remark \ref{rem:ZO}\ref{it:ZO1}, $Z_R(r)=Z_{S'}(r)$, so $C_{S'}(r)=Z_R(r)\times \la O_R(r),\mbf x_B\ra$.

Now  we have relative presentations $S=\la R, A\,|\, [C_R(r),\mbf x_A]\ra$, in \ref{it:b1}, and 
$S=\la R, A\,|\, [O_R(r),\mbf x_A],V\cdot r\ra$, in \ref{it:c}, whence 
\begin{enumerate}
\item[\ref{it:b1}] $T=\la R, A, B\,|\, [C_R(r),\mbf x_A],[\Z_R(s),\mbf x_B],[O_R(s),\mbf x_B],[\mbf x_A,\mbf x_B]\ra$ and
\item[\ref{it:c}] $T=\la R, A, B\,|\, [O_R(r),\mbf x_A],V\cdot r,[\Z_R(s),\mbf x_B],[O_R(s),\mbf x_B],[\mbf x_A,\mbf x_B]\ra$.  
\end{enumerate}
Also, $S'$ has relative presentation $S'=\la R,B\,|\,[\Z_R(s),\mbf x_B],[O_R(s),\mbf x_B]\ra$, and using the description of $C_{S'}(r)$ above, in each case we see that $T'$ and $T$ have the same relative presentation. Thus $T=T'$ and $T$ is in $\cC$. 

Therefore, in all cases $G_l$ is in $\cC$ and it remains to show that $G_l$ satisfies condition \eqref{eq:CAP}.  To simplify notation we use the notation above with  $S=G_{l-1}$ and $T=G_l$, where $S$ is generated by $X_S$ and $T$ is generated by $X_T=X_S\cup \mbf x_B$. From the above we may now assume that the floor $S< T$ is of type a$'$2), b1) or c). In all cases we have
$T=S\ast_{C}D$, where $C=C_S(s)$, for some $s\in S$ and $D$ is either of the form $D=C\times B$, or $D=O_s(s)\times \Sigma$, where $B$ is free abelian with basis $\mbf x_B$ and $\Sigma$ is a surface group generated by $\mbf x_B$ and $s$.
Assume that $Y\subseteq X_T$ is such that
$\la Y\ra$ is abelian. If $Y\nsubseteq X_S$ then $Y$ contains an element $z$ of $\mbf x_B$, so $C_T(Y)\le D$. If $S<T$ is not of type c) then, as $\la Y\ra$ and $C$ are abelian, $C_T(Y)=C\times B=C_T(z)$, as required. If $S<T$ is of type c) then,
$Y\cap \mbf x_B=\{z\}$ (as $\la Y\ra$ is abelian) and $C_T(Y)=<z>\times O_S(s)=C_T(z)$ again. Hence we may assume that
$Y\subseteq X_S$.  If $Y\nsubseteq C$ then $C_T(Y)=C_S(Y)$, and we have $u\in \la Y\ra$ such that $C_S(Y)=C_S(u)$, from \eqref{eq:CAP}
in $S$. If $u\in C$ then $C_T(u)$ contains elements outside $S$, a contradiction, so $u\notin C$ and $C_T(u)=C_S(u)=C_T(Y)$.
This leaves the case $Y\subseteq C_S(s)$. In this case there exists $u\in \la Y\ra$ such the $C_S(Y)=C_s(u)$ and
then $C_T(Y)\le C_T(u)$. As in the proof of Lemma \ref{lem:centdesc} and Corollary \ref{cor:attSinC},
if $C_S(u)$ is abelian then $C_T(u)=C_T(s)$ is abelian, in which case $C_T(s)\le C_T(Y)$, so $C_T(Y)=C_T(u)$.
Otherwise $C_T(u)$ is non-abelian and $C_T(u)=\Z_S(u)\times \la O_S(u),\mbf x_B\ra$. In this case if $g\in C_S(u)$, then
$g\in C_S(Y)\le C_T(Y)$, by definition of $u$, and $\mbf x_B\subseteq C_T(Y)$, so $C_T(u)=C_T(Y)$ again. Therefore, in all cases \eqref{eq:CAP} holds. 
\end{proof}

\begin{thm} \label{thm:abelianedge}
  Let $\GG$ be a coherent RAAG and let $\Ts=(G,\HH,\pi)$ be a graph tower associated to a limit group $L$ over $\GG$.
  Then $G$ has a graph of groups decomposition {\rm(}in the same generating set as $G$ {\rm)} where:
\begin{itemize}
\item the graph of the decomposition is a tree;
\item edge groups are finitely generated free abelian;
\item vertex groups are either graph towers of lower height, or a finitely generated free abelian group or the direct product of a  finitely generated free abelian group and a non-exceptional surface group.
\end{itemize}
\end{thm}
\begin{proof}
Note that by Lemma \ref{lem:pctowerC}, without loss of generality we can assume that all the floors of $G$  are of type a2) with an abelian  centraliser amalgamated; or of type b1); or of type c). 
We prove the statement by induction on the height $l$ of the graph tower.
If the height is 0, then $G=\GG$. In this case, $G$ is a coherent RAAG and admits a graph of groups decomposition (in the same set of generators) where all the vertex groups and edge groups are finitely generated torsion-free abelian groups, see \cite{droms87a}. 

Assume that $l>0$ and let $\Ts=(G_l,\HH_l,\pi_l)$. 
If $G_{l-1}<G_l$ is of type c) then $G_l$ has a splitting as in \eqref{eq:typec}, and this decomposition is of the required type (on the original generating set). 

If  $G_{l-1}<G_l$ is of type b1) we may assume, as in the proof  of Lemma \ref{lem:pctowerC}, that $C_{G_{l-1}}(u)$ is abelian, and in this case, as in the case when $G_{l-1}<G_l$ is of type a2) with abelian centraliser, Lemma \ref{lem:prtower} exhibits an appropriate decomposition.
\end{proof}

\begin{cor}\label{cor:towers are coherent}
  Let $\GG$ be a coherent RAAG and let $\Ts=(G,\HH,\pi)$ be a graph tower associated to a limit group $L$ over $\GG$.
  Then $G$ is coherent.
\end{cor}
\begin{proof}
We use induction on the height $l$ of the graph tower. If the height is 0, then $G=\GG$ and hence  is coherent.

Let $l>0$ and $G_l=G$. By Theorem \ref{thm:abelianedge}, the group $G_l$ admits a decomposition as an amalgamated product with finitely generated abelian edge group. By the induction hypothesis, the graph tower $G_{l-1}$ is coherent. Furthermore, in all the cases, the other vertex group is a direct product of a coherent group and an abelian group and hence, by Lemma \ref{lem:coh}, vertex groups of the decomposition of $G_l$ as an amalgamated product are coherent. 

Therefore, since an amalgamated product of coherent groups over a finitely generated abelian subgroup is coherent
(see \cite[Lemma 4.8]{Wilton08}) it follows that $G_l$ is coherent.
\end{proof}

\begin{cor}\label{cor:limit groups are fp}
Limit groups over coherent RAAGs are coherent. In particular, they are finitely presented.
\end{cor}

\begin{proof}
If $L$ is a limit group over a coherent RAAG $\GG$, then $L$ is a subgroup of a graph tower $G$ over $\GG$ by Remark \ref{rem:prop_towers}, and $G$ is coherent by Corollary \ref{cor:towers are coherent}, and hence the result follows from the fact that limit groups are finitely generated by definition.
\end{proof}

\section{Characterisation of limit groups over coherent RAAGs}\label{sec:characterisation subgroups ICE}

We are now in a position to show that limit groups over coherent RAAGs are exactly the finitely generated subgroups
of their $\BZ[t]$-completions.

\begin{thm} \label{thm:charlim}
Let $\GG$ be a coherent RAAG. Then a finitely generated group $G$ is a limit group over $\GG$ if and only if $G$ is a subgroup of the $\BZ[t]$-completion $\GG^{\BZ[t]}$ of $\GG$.
\end{thm}
\begin{proof}
It follows from Theorem \ref{thm:tensor}, that $\GG^{\BZ[t]}$ is discriminated by $\GG$. Since discrimination passes to subgroups, it follows that every finitely generated subgroup of $\GG^{\BZ[t]}$ is a limit group over $\GG$.

Let us prove the converse. By \cite[Theorem 8.1]{CRK15}, it follows that any limit group over $\GG$ is a subgroup of a graph tower $\Ts$ over $\GG$ of height $l$. Hence, we are left to show that graph towers over $\GG$ embed into  $\GG^{\BZ[t]}$. We proceed by induction on the height $l$ of the tower $\Ts=(G_l,\HH_l,\pi_l)$. Assume, by induction, that graph towers over $\GG$ of height $l-1$ embed into  $\GG^{\BZ[t]}$. In the light of Lemma \ref{lem:pctowerC}, without loss of generality we may assume that all the floors of $G$ are of type a2) with an abelian  centraliser; or of type b1); or of type c).

Suppose first that the floor $G_{l-1}<G_l$ is of type a2), where $C_{G_{l-1}}(\KK^\perp)$ is abelian. In this case, the statement is obvious, since $G_l=G_{l-1}*_A B$, where $A$ and $B$ are finitely generated free abelian groups.
Similarly, if the floor $G_{l-1}<G_l$ is of type b1) then $G_l$ is an extension of an abelian centraliser, hence the statement follows in this case. 

We are left to consider the case when the floor $G_{l-1}<G_l$ is of type c). 
In this case, as in the proof of Theorem \ref{thm:abelianedge}, $G_l$ has the decomposition \eqref{eq:typec} and, as in Corollary \ref{cor:attSinC}, $G_l$ embeds in the group $T^*=\langle G, y\mid [C,y]\rangle=G*_C(C\times \la y\ra)$ of \eqref{eq:T*}, with $G=G_{l-1}$ and $C=C_{G_{l-1}}(u)$. 

As $G_{l-1}<T^*$ is of type a2) with abelian centraliser, the statement follows.
\end{proof}

\section{Limit groups over coherent RAAGs are $CAT(0)$} \label{sec:AB}
A group that acts properly discontinuously and co-compactly  by isometries on a proper $CAT(0)$ space is called a $CAT(0)$ \emph{group}.

In this section we prove that limit groups over coherent RAAGs are $CAT(0)$ and deduce that the conjugacy problem is decidable for these groups. We recall the following two definitions, see \cite{AB06}.

\begin{defn}
	Let $X$ be a connected locally $CAT(0)$ space. A subspace $C$ of $X$ is a {\it core} of $X$ if it is compact, locally $CAT(0)$ (with respect to the induced path metric) and the inclusion $C\to X$ is a homotopy equivalence.
\end{defn}

\begin{defn} \label{def:gc}
  A connected locally $CAT(0)$ space $Y$ is geometrically coherent if for every covering space $X \to Y$ with $X$ connected and $\pi_1(X)$ finitely generated and every compact subset $K \subset X$ it follows that $X$ contains a core $C \supset K$.
\end{defn}

If $G$ is the fundamental group of a geometrically coherent space $Y$ and $H$ is a finitely generated subgroup of $G$ then,
taking
a covering space $X\to Y$ with $\pi_1(X)=H$ and a core $C$ of $X$, by definition $\pi_1(C)=H$ and $C$ is locally $CAT(0)$,
so $H$ is finitely presented and the Cartan-Hadamard Theorem (see \cite{BridsonHaefliger} for example) implies that $H$ is
$CAT(0)$.

\begin{expl}[see Example 2.4 in \cite{AB06}] \label{expl:abelian}
	A flat torus $T=\mathbb{R}^n/\Lambda$ is geometrically coherent. Indeed, any connected covering space admits a metric splitting $T'\times E$ where $T'$ is a flat torus or a point and $E$
	is Euclidean space $\mathbb{R}^k$ or a point. Thus $T'\times\{point\}$ is a (convex) core.
\end{expl}

\begin{thm}\cite[Theorem 2.6]{AB06} \label{thm:BA}
	Let $M$ and $N$ be geometrically coherent locally $CAT(0)$ spaces. Let $Y$ be the space obtained from the disjoint union of $M$, $N$ and a finite collection of $n$-tori $T_i$
	$$
	M \sqcup\bigsqcup \limits_{i} T_i \times [0,1]\sqcup N
	$$
	by gluing $T_i \times \{0\}$ to a convex space in $M$ by a local isometry and gluing $T_i \times \{1\}$ to a convex space in $N$ by a local isometry.
	
	Then $Y$ is geometrically coherent.
\end{thm}
\begin{proof}
	Proof follows from \cite[Theorem 2.6 and Remark 2.8]{AB06}.
\end{proof}

For a description of the Salvetti complex of a RAAG we refer to Charney's survey \cite{Charney07}.

\begin{lem}\label{lem:salvraag}
	The Salvetti complexes of coherent RAAGs are geometrically coherent.
\end{lem}
\begin{proof}
	Let $\Gamma$ be such that $\GG(\Gamma)$ is coherent and let $C(\Gamma)$ be the corresponding Salvetti complex.
	
	The graph $\Gamma$ contains a clique which induces a decomposition of  $\GG(\Gamma)$ as amalgamated product, where vertex groups are RAAGs corresponding to proper full subgraphs of $\Gamma$ and the amalgamation is over a free abelian group corresponding to the clique,  see \cite{droms87a}.  Hence, $\GG(\Gamma)$ is the fundamental group of a tree $S$ of free abelian groups, see \cite{droms87a}. We use induction on the number of  vertices in $S$.
	
If $S$ has one vertex, then $C(\Gamma)$ is a flat torus and the statement follows from Example
\ref{expl:abelian}. Suppose that the statement of the lemma is true for any $\Gamma'$ such that
the corresponding tree $S'$ has less than $n$ vertices.
Let $\Gamma$ be such that $S$ has precisely $n$ vertices. Let $S'$ be a subtree of $S$ obtained by removing a leaf $v$.
We have $\GG(\Gamma)=\pi_1(S')*_{\BZ^k}\BZ^l$, for some $k<l\in \BN$ and $\pi_1(S')=\GG(\Gamma')$ is a RAAG such
that the corresponding tree $S'$ has less than $n$ vertices.
Notice that the generators of $\mathbb Z^k$ are generators of $\pi_1(S')$ and of $\mathbb Z^l$.
Then $C(\Gamma)$ decomposes as $C(S') \sqcup T^k \times [0,1]\sqcup T^l/\sim$,
where $T^k$ and $T^l$ are $k$- and $l$- tori, $T^k\times \{1\}$ is identified with a $k$-torus
spanned by coordinate circles in $T^l$ and $T^k\times\{0\}$ is identified with a $k$-torus
spanned by coordinate  loops in $C(S')$. In particular, we glue the $T^k\times \{0\}$ and $T^k\times\{1\}$
$k$-tori to convex subspaces of $T^l$ and $C(S')$, respectively, each by a local isometry. Thus, the statement follows by induction and Theorem \ref{thm:BA}.
\end{proof}

Abusing the terminology, we call a group geometrically coherent if it is the fundamental group of a geometrically coherent space. 

\begin{prop} \label{prop:icecoherent}
	Finite iterated centraliser extensions of coherent RAAGs are geometrically coherent.
\end{prop}
\begin{proof}
  Let $\GG$ be a coherent RAAG. We use induction on the number of iterated centraliser extensions. To the
  base of induction $\GG$, we associate the Salvetti complex $\chi(\GG)$ of $\GG$,
  which is geometrically coherent by Lemma \ref{lem:salvraag}.
	
  Now let  $\GG=G_0<G_1<\dots< G_{m-1}<G_m$ be an ICE. Arguing as in the proof of Lemma \ref{lem:pctowerC}, we may assume that the centraliser $C_{G_{m-1}}(u)$ being extended in step $m$ is abelian. Indeed, all non-abelian centralisers correspond to canonical generators of $\GG$ and extension of such a centraliser results in a new coherent RAAG,
  as in the proof of Lemma \ref{lem:pctowerC}.
	
	Hence, without loss of generality, we may assume that  $G_m=G_{m-1} *_{C} \BZ^l$, where $C=C_{G_{m-1}}(u)=\langle u\rangle \times \BZ^{k-1}\cong \BZ^k$, with $k<l$,
        is the (abelian) centraliser of a block element $u$ in $G_{m-1}$ and $G_{m-1}$ is an ICE satisfying the statement of the proposition. By induction, the corresponding core $\chi(G_{m-1})$, such that $G_{m-1}=\pi_1(\chi(G_{m-1}))$, is already defined.
		
	We now define the core $\chi(G_m)$ as $\chi(G_{m-1}) \sqcup T^k \times [0,1]\sqcup T^l/\sim$, where $T^k$ and $T^l$ are $k$- and $l$- tori, $T^k\times \{1\}$ is identified with a $k$-torus spanned by coordinate circles in $T^l$ and the coordinate circles of $T^k\times\{0\}$ are identified with nontrivial  loops in $\chi(G_{m-1})$ that generate the abelian subgroup $C$
        in $\pi_1(\chi(G_{m-1}))=G_{m-1}$. The induction step now follows by Theorem \ref{thm:BA}.
\end{proof}

Since every limit group over a coherent RAAG $\GG$ is a subgroup of an iterated centraliser extension of $\GG$, we obtain the following
\begin{cor} \label{cor:cat0}
	Limit groups over coherent RAAGs are $CAT(0)$. In particular, the conjugacy problem is decidable in limit groups over coherent RAAGs.
\end{cor}
\begin{proof}
  If $L$ is a limit group over a coherent RAAG $\GG$, then $L$ is a subgroup of $\GG^{\BZ[t]}$ by Theorem \ref{thm:charlim} and $\GG^{\BZ[t]}$ is an ICE of $\GG$ by Corollary \ref{cor:expice} (and the subsequent comment). Moreover, it follows from the proof of Theorem \ref{thm:charlim} that $L$ is a subgroup of $G'$, where $G'$ is obtained from $\GG$ by a finite chain of centraliser extensions. Therefore, $G'$ is $CAT(0)$ and geometrically coherent by Proposition \ref{prop:icecoherent}. Hence, $L$ is isomorphic to a finitely generated subgroup of a geometrically coherent group so, from the comments following
  Definition \ref{def:gc}, $L$ is $CAT(0)$.   Finally, by \cite[Theorem 1.12, III.$\Gamma$.1]{BridsonHaefliger}, the conjugacy problem in $L$ is decidable.
\end{proof}

\bibliographystyle{plain}

\end{document}